\documentclass[leqno, 12pt, a4paper, twoside]{article}	

\usepackage{mathptmx}
\usepackage[scaled=.90]{helvet}
\usepackage{courier}
\usepackage{color}
\usepackage[T1]{fontenc}
\usepackage{soul}			
\usepackage{amsmath, amssymb,amsthm}
\usepackage[mathscr]{euscript}
\usepackage{graphicx}
\usepackage{epstopdf}
\usepackage{microtype}			
\usepackage{totpages}			
\usepackage{dsfont}			
\usepackage[left=1in,right=1in,top=1in,bottom=1in]{geometry}
\usepackage{xcolor}
\usepackage[all]{xy}
\usepackage{extarrows}
\usepackage[colorlinks = true,
            linkcolor = red,
            urlcolor  = red,
            citecolor = red,
            anchorcolor = red]{hyperref}

\hypersetup{colorlinks=false}

\makeatletter

\DeclareMathAlphabet{\mathcal}{OMS}{cmsy}{m}{n}

\SetMathAlphabet{\mathcal}{normal}{OMS}{cmsy}{m}{n}
\SetMathAlphabet{\mathcal}{bold}{OMS}{cmsy}{m}{n}

\let\amalg=\undefined
\let\coprod=\undefined
\DeclareSymbolFont{cmsymbols}{OMS}{cmsy}{m}{n}
\DeclareSymbolFont{cmlargesymbols}{OMX}{cmex}{m}{n}
\DeclareMathSymbol{\amalg}{\mathbin}{cmsymbols}{"71}
\DeclareMathSymbol{\coprod}{\mathop}{cmlargesymbols}{"60}
\let\jmath=\undefined
\DeclareSymbolFont{cmletters}{OML}{cmm}{m}{it}
\DeclareMathSymbol{\jmath}{\mathord}{cmletters}{"7C}

\newcounter{arno}
\newenvironment{arlist}{
\begin{list}{{\normalfont\bf (\arabic{arno})\hfill}}
{
\usecounter{arno}
\setlength{\topsep}{0.25ex}
\setlength{\labelwidth}{1ex}
\setlength{\leftmargin}{0ex}
\setlength{\labelsep}{1ex}
\setlength{\rightmargin}{0ex}
\setlength{\itemindent}{2ex}
\setlength{\parsep}{0ex}
\setlength{\itemsep}{0.5ex plus0.2ex minus0.1ex}
}
}{\end{list}}

\newcounter{ano}
\newenvironment{alist}{
\begin{list}{{\normalfont \bf  (\alph{ano})}}
{
\usecounter{ano}
\setlength{\topsep}{0.25ex}
\setlength{\labelwidth}{1ex}
\setlength{\leftmargin}{0ex}
\setlength{\labelsep}{1ex}
\setlength{\rightmargin}{0ex}
\setlength{\itemindent}{2ex}
\setlength{\parsep}{0ex}
\setlength{\itemsep}{0.5ex plus0.2ex minus0.1ex}
}
}{\end{list}}

\newcounter{rno}
\newenvironment{rlist}{
\begin{list}{{\normalfont(\roman{rno})}}
{
\usecounter{rno}
\setlength{\topsep}{0.25ex}
\setlength{\labelwidth}{1ex}
\setlength{\leftmargin}{0ex}
\setlength{\labelsep}{1ex}
\setlength{\rightmargin}{0ex}
\setlength{\itemindent}{2ex}
\setlength{\parsep}{0ex}
\setlength{\itemsep}{0.5ex plus0.2ex minus0.1ex}
}
}{\end{list}}

\sodef\so{}
    {.175em}                                       
    {.6em plus.1em minus .1em}                     
    {0.5em plus.1em minus .1em}                    

\newif\ifsilent \newif\ifproofmode
\newwrite\inx
\proofmodefalse

\def\cameraready{\immediate\openout\inx=index}

\renewcommand*\footnoterule{%
    \kern-1\p@                                           
    \hrule\@width.4\columnwidth                      
    \kern2\p@}                                           
\renewcommand*\@makefntext[1]{%
        \parindent 0.5em                                   
        \noindent                                            
        \hb@xt@1.25em{\hss\@makefnmark}              
        \hskip1\p@ #1}
    
\def\swappedhead#1#2#3{%
    \thmnumber{#2}%
    \thmname{\@ifnotempty{#2}{\,~}#1{~}}
    \thmnote{\,\,\the\thm@notefont\so{(#3)}}}         
\swapnumbers                                            

\def\th@plain{%
    \thm@headpunct{}                          
    \thm@preskip\topsep
    \thm@postskip\thm@preskip
    \itshape                                       
}
\def\th@definition{%
    \thm@headpunct{}%
    \thm@preskip\topsep
    \thm@postskip\thm@preskip
    \normalfont                                    
}
\def\th@remark{%
    \thm@headpunct{}%
    \thm@preskip\topsep 
    \thm@postskip\thm@preskip
    \small                                     
}
\DeclareRobustCommand{\dppqed}{
    \ifmmode \bullet                               
    \else
        \leavevmode\unskip\penalty9999 \hbox{}
        \nobreak\hfill
        \quad\hbox{$\bullet$}%
    \fi
}
\renewenvironment*{proof}[1][\proofname]{\par
    \pushQED{\dppqed}%
    \normalfont
    \topsep5\p@\@plus2\p@\relax                    
    \trivlist
     \item[\hskip\labelsep
         \bfseries
    {#1\@addpunct{\,}}]\ignorespaces
}{%
    \popQED\endtrivlist\@endpefalse
}

\theoremstyle{plain}
\newtheorem{theorem}{Theorem{\hskip-3pt}}[section]
\newtheorem{lemma}[theorem]{Lemma{\hskip-3pt}}
\newtheorem{corollary}[theorem]{Corollary{\hskip-3pt}}
\newtheorem{proposition}[theorem]{Proposition{\hskip-3pt} }

\newtheorem{Euclid'sLemma}[theorem]{Euclid's Lemma{\hskip-3pt}}
\newtheorem{Euler'sFormula}[theorem]{Euler's Formula{\hskip-3pt} }

\newtheorem{Bernstein'sEquivalenceTheorem}[theorem]{Bernstein's Equivalence Theorem{\hskip-3pt}}

\theoremstyle{definition}
\newtheorem{definition}[theorem]{Definition{\hskip-4pt}}
\newtheorem{mypar}[theorem]{}

\theoremstyle{remark}
\newtheorem{example}[theorem]{Example{\hskip-3pt}}
\newtheorem{examples}[theorem]{Examples{\hskip-3pt}}
\newtheorem{remark}[theorem]{Remark{\hskip-3pt}}
\newtheorem{remarks}[theorem]{Remarks{\hskip-3pt}}

\numberwithin{equation}{section}
\RequirePackage{dsfont}                          

\newcommand{\C}{\mathds{C}}

\newcommand{\N}{\mathds{N}}
\newcommand{\Z}{\mathds{Z}}
\newcommand{\Q}{\mathds{Q}}
\newcommand{\R}{\mathds{R}}

\newcommand{\Min}{\mathrm{Min}\,}

\renewcommand{\max}{\mathrm{max}\,}
\renewcommand{\min}{\mathrm{min}\,}

\newcommand{\ord}{\mathrm{ord}}

\newcommand{\Hom}{\mathrm{Hom\,}}

\newcommand{\Supp}{\operatorname{Supp\,}}

\newcommand{\Spec}{\mathrm{Spec\,}}
\newcommand{\KSpec}{K\operatorname{\!-Spec\,}}
\newcommand{\Spm}{\operatorname{Spm\,}}
\newcommand{\Ann}{\operatorname{Ann\,}}

\newcommand{\Dim}{\operatorname{Dim\,}}
\newcommand{\Det}{\operatorname{Det\,}}

\newcommand{\id}{\operatorname{id}}

\newcommand{\Ker}{\operatorname{Ker}}
\newcommand{\Coker}{\operatorname{Coker}}
\newcommand{\img}{\operatorname{Im\,}}
\newcommand{\im}{\operatorname{Im\,}}

\newcommand{\legsym}[2]{\displaystyle
\Bigl(\!\frac{\lower1pt\hbox{$#1$}}{\raise1.5pt\hbox{$#2$}}\Bigr)}

\font\Gmath=cmsy10 scaled 2074%
\def\Ast{\mathop{\vphantom{\sum}%
                   \lower2.5pt\hbox{\Gmath\char3}}}%
\makeatother

\def\mathraise#1#2{\setbox0=\hbox{$#1{#2}$}%
  \setbox1=\vbox to0pt{\box0\vss}%
  \box1}

\def\Sim{\mathpalette\mathraise\sim}
\def\Approx{\mathpalette\mathraise\approx}
\def\iso{\stackrel{\raise1pt\hbox{$\Sim$}}{\longrightarrow}}
\def\iiso{\stackrel{\raise3pt\hbox{$\Approx$}}{\longrightarrow}}
\def\leftiso{\stackrel{\raise1pt\hbox{$\Sim$}}{\longrightarrow}}
\def\leftrightiso{\stackrel{\raise1pt\hbox{$\Sim$}}{\longleftrightarrow}}

\newcommand{\rightiso}{\stackrel{\raise1pt\hbox{$\Sim$}}{\longleftarrow}}
\newcommand{\toiso}{\stackrel{\raise1pt\hbox{$\Sim$}}{\longrightarrow}}

\def\map#1#2{\stackrel{\raise1pt\hbox{${\scriptstyle #1}$}}{\xrightarrow{\hspace*{#2mm}}}}

\def\maprightsquigarrow{\,\raise1pt\hbox{$\shortmid$}\!\!\!\!\rightsquigarrow}
\makeatletter
\providecommand{\bigsqcap}{%
  \mathop{%
    \mathpalette\@updown\bigsqcup
  }%
}
\newcommand*{\@updown}[2]{%
  \rotatebox[origin=c]{180}{$\m@th#1#2$}%
}
\makeatother

\font\tentt=cmtt10

\newcommand{\gotha}{\mathfrak{a}}
\newcommand{\gothb}{\mathfrak{b}}

\newcommand{\gothm}{\mathfrak{m}}
\newcommand{\gothn}{\mathfrak{n}}

\newcommand{\gothp}{\mathfrak{p}}
\newcommand{\gothq}{\mathfrak{q}}

\newcommand{\gothA}{\mathfrak{A}}

\newcommand{\gothP}{\mathfrak{P}}

\title{\Large\bf Nullstellens\"atze and Applications
\thanks{This article grew out of discussions with late Prof.\,Dr.\,Uwe Storch (1940-2017) and lectures delivered by the second and the third author in various workshops and conferences.
Prof.\,Uwe Storch was known for his work in commutative algebra,  analytic and algebraic geometry, in particular derivations, divisor class group and  resultants.}}

\author{Kriti Goel$^{1}$\thanks{The first author is supported by UGC Fellowship of Government of India.} \ , Dilip P. Patil$^2$  and  Jugal Verma$^{3}$  \\[10mm]
\textit{\small $^{1}$\,Department of Mathematics, Indian Institute of Technology Bombay} \\[-1mm]
{\tentt kritigoel.maths@gmail.com} \\[-1mm]
\textit{\small $^{2}$\,Department of Mathematics, Indian Institute of Science Bangalore} \\[-1mm]
{\tentt patil@iisc.ac.in} \\[-1mm]
\textit{\small $^{3}$\,Department of Mathematics, Indian Institute of Technology Bombay} \\[-1mm]
{\tentt jkv@math.iitb.ac.in} \\[2mm]
Dedicated to the memory of Prof. Dr. Uwe Storch}

\cameraready
\date{}

\newcommand{\ncom}{\newcommand}

\ncom{\p}{\mathfrak{p}}
\newcommand{\rsqrt}[1]{{\mathop{\hbox{{\rm r-}}\sqrt{#1}}}}
\newcommand{\rV}[1]{{\mathop{\hbox{{\rm r-V}}({#1})}}}
\newcommand{\radI}[1]{{\mathop{\mathscr{R}ad\hbox{-}\mathscr{I}\,(#1)}}}
\newcommand{\affalg}[2]{{\mathop{\mathscr{A}{\rm ff}\hbox{-}\mathscr{A}{\rm lg}_{#1}(#2)}}}
\newcommand{\zclosed}[1]{{\mathop{\mathscr{F}_{{\rm Z}}(#1)}}}

\newcommand{\rSpec}{{\mathop{\hbox{{\rm r-}}{\rm Spec}\,}\nolimits}}
\newcommand{\CSpec}{{\mathop{\C\hbox{{\rm -}} {\rm Spec}\,}\nolimits}}
\newcommand{\RSpec}{{\mathop{\R\hbox{{\rm -}} {\rm Spec}\,}\nolimits}}
\newcommand{\Kalg}{{\mathop{K\hbox{{\rm -}} {\rm alg}}\nolimits}}

\newcommand{\Aalg}{{\mathop{A\hbox{{\rm -}} {\rm alg}}\nolimits}}

\newcommand{\hSpec}{{\rm h}\operatorname{-Spec\,}}

\usepackage{fancyhdr}
\begin{document}
\maketitle

\begin{abstract}
 In this expository paper, we present simple proofs of the Classical, Real, Projective  and Combinatorial Nullstellens\"atze. Several applications are also presented such as a classical theorem of Stickelberger for solutions of polynomial equations in terms of eigenvalues of commuting operators, construction of a principal ideal domain which is not Euclidean, Hilbert's $17^{th}$ problem,  the Borsuk-Ulam theorem in topology and solutions of the conjectures of Dyson, Erd\"{o}s and Heilbronn.
\end{abstract}

\section{Introduction}

Hilbert's Nullstellensatz (HNS) is one of the fundamental results of Hilbert which paved the way for a systematic introduction of algebraic techniques in algebraic geometry. It was proved in the third section of his landmark paper on invariant theory  \cite{Hilbert1893}. The proof runs into five pages. In fact, Hilbert proves it for homogeneous polynomials. He applies induction on the number of indeterminates and uses elimination theory and resultants. Since the appearance of this proof, several new proofs have appeared in the literature. Notable among them are: (1) proof by A. Rabinowitsch \cite{rabinowitsch}, (2)
Krull's proof based on dimension theory of algebraic varieties, Noether normalization lemma and the concept of integral dependence \cite{krull}, (3) proof by Krull and Van der Waerden for uncountable fields \cite{artin}, (4) proof by E. Artin and J. Tate based on the Artin-Tate lemma \cite{artinTate} (5) O. Zariski's proof based on field theory \cite{zariski}, (6)
proof by R.  Munshi \cite{munshi} and its exposition by P. May \cite{may} and (7) a remarkably simple proof by Arrondo \cite{arrondo} using resultants. (8) Krull \cite{krull52}\; and independently Goldman \cite{goldman} introduced the notion of Jacobson ring, a ring in which every prime ideal is an intersection of maximal ideals. They proved that a finitely generated algebra over a Jacobson ring is a Jacobson ring which implies the HNS.

\medskip

The objective of this paper is to present an exposition of four variations of the Hilbert's Nullstellensatz, namely, the classical, real, projective and combinatorial. Each of these versions has given rise to
new techniques and insights into the basic problem of understanding the common solutions of polynomial equations.

\medskip
Analogues of the
HNS have been investigated for non-algebraically closed fields. Notable among them are the real Nullstellensatz \cite{knebusch}, \cite{lam} and the combinatorial Nullstellensatz \cite{alon}.  There is a Nullstellensatz for partial differential equations \cite{shankar} and most recently a tropical Nullstellensatz \cite{shustin} has also been proved. We have selected   simple and short proofs and a few striking applications for each of these versions which are accessible to students with basic background in algebra. 

\medskip

We now describe the contents of various sections. In section 2, we  discuss the  classical version of the Nullstellensatz over algebraically closed fields. We present a proof due to E. Arrondo \cite{arrondo} which uses two lemmas about polynomials and their resultants. This proof is very much in the spirit of Hilbert's original proof.
We gather six versions of Classical Nullstellensatz and show that they are all  equivalent to  the weak Nullstellensatz.
As an application, we present a theorem of Stickelberger about systems of polynomial equations which have finitely many solutions. This theorem converts the problem of construction of the solutions to the problem of finding  common eigenvectors of commuting linear operators acting on a finite dimensional vector space. We also discuss a general  construction of a Principal ideal domain that is not a Euclidean domain.
\medskip

We present the Real Nullstellensatz in section 3. It answers the question about existence  of a real solution of  a system of real polynomial equations. The Real Nullstellensatz  has a weak version and a strong version which are similar to the corresponding versions of the classical HNS for algebraically closed fields. The central concepts here are those of real fields, real closed fields and real radical of an ideal. We shall present the proofs of both the versions assuming the Artin-Lang homomorphism theorem.   We present a modern solution of Hilbert's $17^{th}$ problem. The Real Nullstellensatz was proved only in the 1970's. A systematic study of real algebraic varieties was started soon after.

\medskip
We shall discuss the Projective Nullstellensatz in section 4. This answers the question of  existence of a nontrivial solution  of a system of homogeneous polynomial equations. We shall prove that if $f_1, f_2,\dots, f_n$ are homogeneous polynomials in $K[X_0, X_1, \dots, X_n]$ where $K$ is a $2$-field then there is a nontrivial solution to the system $f_1=f_2=\dots=f_n=0.$ We follow the approach given in \cite{pfister} which uses Hilbert functions and multiplicity of a graded ring. As an application, we prove the  the Borsuk-Ulam Theorem  in topology.

\medskip
Section 5 is devoted to the most recent version of the Nullstellensatz, namely the Combinatorial Nullstellensatz.  We present a proof of Noga Alon's formulation \cite{alon} using the Classical Nullstellensatz. We describe  two  striking applications of the Combinatorial Nullstellensatz:  a proof of Dyson's conjecture about the constant term of a Laurent polynomial and a solution of a conjecture of Erd\"{o}s and Heilbronn about a lower bound on the cardinality of $A+B$ where $A$ and $B$ are subsets of a finite field.

\medskip
This expository article is not intended to be a survey paper on the Nullstellensatz. There are important works which we do not discuss, for example, the Tropical Nullstellensatz and the Nullstellensatz for partial differential equations, the Eisenbud-Hochster's paper about Nullstellensatz with nilpotents \cite{EisenbudHochster}, role of Gr\"{o}bner bases in computation of radical ideals and testing whether an ideal is the unit ideal of a polynomial ring and works of many authors about Effective Nullstellensatz\cite{kollar}, \cite{zelo}, \cite{brown}, \cite{einlaz}, \cite{Kol1} etc. A version of Nullstellensatz for finite fields has been discussed in \cite{ghorpade}.

\section{Nullstellens\"atze}

Hilbert's Nullstellensatz is the starting point of the classical algebraic geometry, it provides a bijective correspondence between {\it affine algebraic sets} which are geometric objects and {\it radical ideals} in a polynomial algebra (over a field) which are algebraic objects.
\smallskip

In this section we formulate several versions of Nullstellensatz and prove their equivalence. First we recall some standard notation, definitions and preliminary results. For other undefined terms and notions we recommend the reader to see the books   \cite{AM} and \cite{PS2010}.

\begin{mypar}\label{mypar:2.1}\,{\bf Notation and Preliminaries}\,
All rings considered  in this article are commutative rings with unity. The letter $K$ will always denote a field and the letters $A$, $B$, $C$, $R$ will be generally used for rings.
As usual we use $\N$, $\Z$, $\Q$, $\R$ and $\C$ to denote the set of non-negative integers, the ring of integers, the fields of rational, real and complex numbers respectively.

{\small

\begin{arlist}
\smallskip

\item {\bf Algebras over a ring}\,
Let $A$ be a ring. An $A$-{\it algebra} $B$  is a ring  together with a ring homomorphism  $\varphi:A\to B$ called the  {\it structure homomorphism}  of the $A$-algebra $B$.   Overrings and residue class rings of $A$  are considered $A$-algebras with natural inclusions and surjections as the structure homomorphisms, respectively. The {\it polynomial ring} $A[X_{i}\mid i\in I]$ in the indeterminates $X_{i}$, $i\in I$, is an $A$-algebra with the natural inclusion $A\hookrightarrow A[X_{i} \mid i\in I]$ as the structure homomorphism.
\smallskip

Let $B$ and $C$ be $A$-algebras. An $A${\it -algebra homomorphism} from  $B$ to $C$ is a ring homomorphism $\theta:B\to C$  such that the diagram
\vspace*{-4mm}
\begin{align*}
\xymatrix{
B \ar[rr]^{\theta}   & & C  \\
& A \ar[ul]^{\varphi}\ar[ur]_{\psi}  & \\
}
\vspace*{-2mm}
\end{align*}
is commutative, that is, $\,\theta\circ\varphi  = \psi\,$,  or equivalently $\theta$ is $A$-linear.
\smallskip

The set  of all $A$-algebra homomorphisms from $B$ to $C$ is denoted by $\Hom_{\Aalg}(B,C)$.
\smallskip

\item  {\bf Polynomial algebras}\,
Polynomial algebras are the free objects (in the language of categories) in the category of (commutative) algebras over a ring $A$ with the following universal property\,:
\smallskip

{\bf Universal property of polynomial algebras}\, {\it Let $B$ be an $A$-algebra and let $\,x=(x_{\,i})_{i\in I}$, be a family of elements of $B$. Then there exists a unique $A$-algebra homomorphism $\,\varepsilon_{x}:A[X_{i}\mid i\in I] \to B$ such that $X_{i}\mapsto x_{\,i}$ for every $i\in I$.}
In particular, we can identify $\Hom_{\Aalg}(A[X_{i} \mid i\in I\,], B)$ with $B^{\,I}$. For $I=\{1,\ldots , n\}$, we can identify $\Hom_{\Aalg}(A[X_1,\ldots,X_n],  B)$ with $B^{\,n}$.
The unique $A$-algebra homomorphism $\varepsilon_{x}$ is called the {\it substitution homomorphism} or the {\it evaluation homomorphism} defined by $x$.
\smallskip

The image of $\varepsilon_{x}$ is the smallest $A$-subalgebra of $B$ containing $\{x_{\,i}\mid i\in I\}$ and is denoted by $A[x_{\,i}\mid i\in I\,]$. We call it the $A${\it -subalgebra generated} by the family $x_{\,i}$, $i\in I$.  We say that $B$ is an $A${\it -algebra generated} by the family $x_{\,i}$, $i\in I$, if $B=A[x_{\,i} \mid i\in I\,]\,$. Further, we say that $B$ is a {\it finitely generated} $A${\it -algebra} or an $A${\it -algebra of finite type}  or an {\it affine algebra over $A$}\, if there exists a finite family $x_1,\ldots,x_{\,n}$ of elements of $B$ such that $B=A[x_1,\ldots,x_{\,n}]$. A ring homomorphism $\varphi:A\to B$ is called a {\it homomorphism of finite type} if $B$ is an $A$-algebra of finite type with respect to $\varphi$.
\smallskip

\item  {\bf Prime, maximal and radical Ideals}\,  Let $A$ be a ring.
The set  $\Spec A$ (resp. $\Spm A$) of prime (resp. maximal) ideals in $A$ is called the {\it prime} (resp. {\it maximal}) {\it spectrum} of $A$. Then $\Spm A\subseteq \Spec A$ and a well-known theorem asserts that if $A\neq 0$ then $\Spm A\neq \emptyset$. For example, $\Spm \Z$ is precisely the set $\mathds{P}$  of positive prime numbers  and $\Spec \Z=\{0\}\cup \mathds{P}$. The ring $R$ is a field if and only if $\Spm A=\{0\}$. The ring $A$ is an integral domain if and only if $\{0\}\in \Spec A$.
For an ideal $\,\mathfrak{a}\,$ in $\,R\,$, the ideal
$\sqrt{\mathfrak{a}}:=\{f\in R \mid f^{\,r}\in \mathfrak{a} \hbox{ {\rm for some integer} } r\geq 1\}$ is called the {\it radical} of $\mathfrak{a}$. Clearly $\mathfrak{a}\subseteq \sqrt{\mathfrak{a}}$. If $\sqrt{\mathfrak{a}} = \mathfrak{a}$, then $\mathfrak{a}$ is called a {\it ra\-di\-cal ideal}.  Obviously,
$\sqrt{\sqrt{\mathfrak{a}}} = \sqrt{\mathfrak{a}}$. Therefore the radical of an ideal is a radical ideal. Prime ideals are radical ideals. An ideal $\gotha$ in $\Z$ is a radical ideal if and only if $\gotha=0$ or $\gotha$ is generated by a square-free integer.
\smallskip

The radical $\mathfrak{n}_{A}:= \sqrt{0}$ of the zero ideal is the ideal of nilpotent elements and is called the {\it nilradical} of $A$. {\it The nilradical $\mathfrak{n}_A=\cap_{\,\gothp\in\Spec A}\,\gothp$ is the intersection of all prime ideals in $A$.}   More generally,  ({\it formal Nullstellensatz})\, $\sqrt{\gotha }= \cap_{\,\gothp\in\Spec A}\,\{\gothp \mid \gotha\subseteq \gothp\}$ for every ideal $\gotha$  in $A$.
\smallskip

The intersection $\displaystyle\gothm_{A}:=\cap_{\,\gothm\in\Spm A}\,\gothm$ of maximal ideals in $A$ is called the {\it Jacobson radical} of $A$. Clearly, $\gothn_{A}\subseteq \gothm_{A}$. The Jacobson radical of $\Z$ (resp. the polynomial algebra $K[X_{1},\ldots , X_{n}]$ over a field $K$) is $0$.
\smallskip

\item
{\bf Integral Extensions}. Let $A\subseteq B$ be an extension of rings. We say that an element $b\in B$ is {\it integral over $A$} if $b$ is a zero of a monic polynomial $a_{0}\!+\!\cdots \!+\! a_{n-1}X^{n-1}\!\!+X^{n}\!\!\in A[X]$,  i.\,e. if
$a_{0}\!+\!\cdots \!+\! a_{n-1}b^{n-1}\!+\!b^{n}\!\!=\!0$ with $a_{0},\ldots , a_{n-1}\!\!\in\!A$. We say that $B$ is {\it integral over $A$} if every element of $B$ is integral over $A$. The concept of an integral extension is a generalization of that of an algebraic extension. For example, an algebraic field extension $E\,\vert \,K$ is an integral extension. Moreover, if a ring extension $A\subseteq B$ is an integral extension, then the polynomial extension $A[X_{1}, \ldots , X_{n}]\subseteq B[X_{1},\ldots , X_{n}]$ is also integral.  It is easy to see that\,:  {\it If $B$ is  a finite type algebra over a ring $A$, then $B$ is integral over $A$ if and only if $B$ is a finite $A$-module.} \, Later we shall use the following simple proposition in the proof of the classical form of HNS\,:
\smallskip

{\bf Proposition}\,
{\it Let $\,A\subseteq B$ be an integral extension of rings and $\gotha\subsetneq A$ be a non-unit ideal in $A$. Then the extended ideal $\gotha\,B$ (in $B$) is also a non-unit ideal.}
\smallskip

{\bf Proof}\, Note that $\gotha\,B=B$ if and only if $1\in\gotha\,B$. Moreover,  if $1\in\gotha\,B$ then since $B$ is integral over $A$, already  $1\in\gotha\,B'$ for some finite $A$-subalgebra $B'$ of $B$. Therefore, we may assume that $B$ is a finite $A$-module. But, then by the Lemma\footnote{\label{foot:1}{\bf Lemma  of Dedekind-Krull-Nakayama}\, {\it Let $\gotha$ be an ideal in a commutative ring $A$ and $V$ be a finite $A$-module. If $\gotha\,V=V$, then there exists an element $a\in \gotha$ such that $(1-a)V=0$, {\rm i.\,e.} $ (1-a)\in \Ann_{A} V$.} For a proof one uses the well-known ``Cayley-Hamilton trick''.},  there exists an element $a\in\gotha$ such that $(1-a)\,B=0$, in~particular, $(1-a)\cdot 1=0$, i.\,e. $1=a\in\gotha$ which contradicts the assumption. \dppqed
\smallskip

\item {\bf The ${\bf K}$-Spectrum of a ${\bf K}$-algebra}\, (see \cite{PS2010})
Let $K$ be a field. Then using the universal property of the polynomial algebra $K[X_{1},\ldots , X_{n}]$,   the affine space $K^{n}$ can be identified with the set of $K$-algebra homomorphisms
$\Hom_{\Kalg}(K[X_{1},\ldots,X_{n}]\,,K)$ by identifying $a =
(a_{1},\ldots,a_{n}) \in K^{n}$ with the substitution homomorphism
$\xi_{\,a}: K[X_{1},\ldots,X_{n}] \to K$, $X_{i} \mapsto a_{\,i}$.  The kernel of $\xi_{\,a}$ is the maximal ideal $\mathfrak{m}_{a} = \langle X_{1}-a_{1},\ldots,X_{n}-a_{n}\rangle $ in $K[X_{1},\ldots,X_{n}]$. Moreover, every maximal ideal $\mathfrak{m}$ in $K[X_{1},\ldots,X_{n}]$ with $K[X_{1},\ldots,X_{n}]/\mathfrak{m} =K$ is of the type $\mathfrak{m}_{a}$ for a unique $a=(a_{1},\ldots , a_{n})\in K^{n}$; the component $a_{\,i}$ is determined by the congruence $X_{i} \equiv~a_{i}~ {\rm mod}~ \mathfrak{m}$.
\smallskip

The subset $\,\KSpec  K[X_{1},\ldots,X_{n}]:=\{\mathfrak{m}_{a} \mid a\in K^{n}\}$ of $\Spm K[X_{1},\ldots,X_{n}]$ is called the {\it $K$-spectrum} of $K[X_{1},\ldots,X_{n}]$.  We have the identifications\,:
\begin{align*}
K^{n}  & \xlongleftrightarrow{\hspace*{20mm}}  \enskip \Hom_{\Kalg}(K[X_{1},\ldots,X_{n}]\,,K) \enskip  \xlongleftrightarrow{\hspace*{20mm}}  \enskip \KSpec K[X_{1},\ldots,X_{n}]\,,\\
a  \enskip & \xlongleftrightarrow{\hspace*{42.5mm}} \enskip   \xi_{\,a}  \enskip \xlongleftrightarrow{\hspace*{42.5mm}} \enskip  \mathfrak{m}_{a}=\Ker
\xi_{\,a}\,.
\end{align*}
More generally, for any $K$-algebra $A$, the map $\,\Hom_{\Kalg}(A\,,K)\longrightarrow
\{\gothm\in\Spm A\mid A/\gothm =K\}$, $\xi\mapsto \Ker\xi$, is bijective. Therefore we make the following definition\,:
\smallskip

For any $K$-algebra $A$, the subset
$\,\KSpec A:= \{\gothm\in\Spm A\mid A/\gothm =K\}\,$ is
called the $K${\it -spectrum} of $A$ and is denoted by $\KSpec A$. Under the above bijective map, we have the identification $\KSpec A = \Hom_{\Kalg}(A\,,K)$.
\smallskip

For example, since $\C$ is an algebraically closed field, $\Spm \C[X] = \CSpec \C[X]$, but $\RSpec\R[X] \subsetneq \Spm\R[X]$.  In fact, the maximal ideal $\mathfrak{m}:=\langle X^{2}+1\rangle$ does not belong to $\RSpec\R[X]$. More generally, a field $K$ is algebraically closed
\footnote{\label{foot:2}A field $K$ is called {\it algebraically closed} if every non-constant polynomial in $K[X]$ has a zero in $K$ or equivalently, every irreducible polynomial in $K[X]$ is linear. The {\it Fundamental Theorem of Algebra} asserts that\,: {\it  the field of complex numbers $\C$ is algebraically closed.} This was first stated  in 1746  by J.\,d'Alembert (1717-1783), who gave an incomplete proof \,---\, with gaps at that time. The first complete proof was given in 1799 by Carl Friedrich Gauss (1777--1855).}
if and only if $\Spm K[X] = \KSpec K[X]$.
\smallskip

\item {\bf Polynomial maps}\,
Let $A$ be an algebra over a field $K$. For a polynomial $f\in K[X_{1},\ldots,X_{n}]$, the function $\varphi^{*}_{f}:A^{n} \to A$, $a \mapsto f(a)$, is called the {\it polynomial function (over $K$)} defined by $f$.
If $A$ is an infinite integral domain, then the polynomial function $\varphi^{*}_{f}$ defined by $f$ determines the polynomial $f$ uniquely. This  follows from the following more general observation \,:
\smallskip

{\bf Identity Theorem for Polynomials}\, {\it   Let $A$ be an integral domain and $f\in A[X_{1}, \ldots , X_{n}]$, $f\neq 0$. If  $\Lambda_{1},\ldots , \Lambda_{n}\subseteq A$ with $|\Lambda_{i}|>\deg_{X_{i}} f$ for all $i=1,\ldots , n$, then  $\Lambda:=\Lambda_{1}\times \cdots \times \Lambda_{n}\not\subseteq {\rm V}_{A}(f):=\{a\in A^{n}\mid f(a)=0\}$, that is, there exists $(a_{1},\ldots , a_{n})\in \Lambda$ such that $f(a_{1},\ldots , a_{n})\neq 0$. In~particular, if $A$ is infinite, then $f:A^{n}\to A$, $a\mapsto f(a)$, is not a zero function. If $A$ is infinite, then the evaluation map
$\,\varepsilon: A[X_{1}, \ldots , X_{n}] \longrightarrow {\rm Maps}\,(A^{n}, A)$, $\,f\mapsto \varepsilon(f): a\mapsto f(a)\,$ is injective.}

\begin{proof}
We prove the assertion by induction on $n$. If $n=0$, it is trivial. For a proof of the inductive step from $n-1$ to $n$, write $f=\sum_{k=0}^{d}f_{k}(X_{1},\ldots , X_{n-1})X_{n}^{k}$ with $f_{d}(X_{1}, \ldots , X_{n-1})\neq 0$ in $A[X_{1},\ldots , X_{n-1}]$. Since $\deg_{X_{i}} f_{d} \leq \deg_{X_{i}} f <|\Lambda_{i}|$ for all $i=1,\ldots , n-1$, by induction hypothesis, there exists $(a_{1},\ldots , a_{n-1})\in A^{n-1}$ with $f_{d}(a_{1},\ldots , a_{n-1})\neq 0$. Therefore $f(a_{1},\ldots , a_{n-1}, X_{n})$ is a non-zero polynomial in $A[X_{n}]$ of degree $d<|\Lambda_{n}|$ and hence there exists $a_{n}\in \Lambda_{n}$ with $f(a_{1},\ldots , a_{n-1},a_{n})\neq 0$.
\end{proof}

\smallskip

If $A=K$, then the identifications in (5) above  allow us to write $\,f(a)=\xi_{\,a}(f)\equiv f ~{\rm mod}~\mathfrak{m}_{a}$ for any $a\in K^{n}$; $f(a)$ is called the {\it value of $f$}  at $a\,$,  or  at $ \xi_{\,a}\,$,  or  at $\mathfrak{m}_{a}$.
\smallskip

Let $\varphi:K[Y_{1},\ldots,Y_{m}] \to K[X_{1},\ldots,X_{n}]$ be a $K$-algebra
homomorphism and let $f_{i}:=\varphi(Y_{i}), 1\leq i\leq m$. Then the map
$\varphi^{*}:K^{n} \to K^{m}$ defined by $\varphi^{*}(a_{1},\ldots,a_{n}) = (f_{1}(a), \ldots ,f_{m}(a))$ is called the {\it polynomial map} associated to $\varphi$.
Under the identifications in (5), the polynomial map $\varphi^{*}$ is described as follows\,: \ $\xi_{\,a} \mapsto \varphi^{*}\xi_{\,a}=\xi_{\,a} \circ \varphi\,$  or\,  by
$\mathfrak{m}_{a} \mapsto \varphi^{*}\mathfrak{m}_{a} =\varphi^{-1}(\mathfrak{m}_{a}) =\mathfrak{m}_{f(a)}\,$, $a\in K^{n}$. For every $G\in K[Y_{1},\ldots,Y_{m}]$, we have $\varphi^{*}_{G}\circ\varphi^{*} = \varphi^{*}_{\varphi(G)}$.
\smallskip

More generally, for any $K$-algebra homomorphism $\varphi\!:\!A\to B$, we define
the map $\varphi^{*}\!:\!\KSpec B \to \KSpec A\,$ by $\varphi^{*}\xi\!:=\!\xi\circ \varphi$ or by  $\varphi^{*}\mathfrak{m}\!=\!\varphi^{-1}\!(\mathfrak{m})$,
$\mathfrak{m}\!=\!\Ker\xi\!\in\!\KSpec B\!=
\Hom_{\Kalg}(B,K)$. Further, if $\psi\!:\!B\to C$ is an another $K$-algebra homomorphism then  $(\psi\circ\varphi)^{*}\!=\! \varphi^{*}\circ\psi^{*}$.

\end{arlist}
}
\end{mypar}

\begin{mypar}\label{mypar:2.2}\,
In general, we are interested in studying the solution set of  a finite system of polynomials $f_{1}, \ldots , f_{m}\in K[X_{1}, \ldots , X_{n}]$ over a given field $K$ (for example,  $K=\Q$, $\R$, $\C$,  or any finite field, more generally, even over a commutative ring, e.g. the ring of integers $\Z$)  in the affine $n$-space $K^{n}$ over $K$ or even in bigger affine $n$-space $L^{n}$ over a field extension $L$ of $K$. Typical cases are\,:

\begin{alist}

\item   $K=\R$, $L=\C$. ({\it Classical Algebraic Geometry}).

\item  $K=\Q$, $L=\C$ or $\overline{\Q}:=$ the algebraic closure of $\Q$ in $\C$. ({\it  Arithmetic  Geometry})

\item  $K$ is  a finite field, $L=\overline{K}:=$ the algebraic closure of $K$.
\end{alist}

 \end{mypar}

\begin{mypar}\label{mypar:2.3}\,{\bf Affine $K$-algebraic sets}\,
Let $L\,\vert\, K$ be a field extension of a field $K$. The solution space
\begin{align*}
{\rm V}_{L}(f_{j}, j\in J) = \{a\in L^{n}\mid f_{j}(a)=0 \ \hbox{for all} \ j\in J\} \subseteq L^{n}
\end{align*}
of a family $f_{j}$, $j\in J$, of polynomials in $K[X_{1},\ldots , X_{n}]$
is called an {\it affine $K$-algebraic set} in $L^{n}\!$,   the family $f_{j}$, $j\in J$ is called a  {\it system of defining equations}, the field $K$ is called the {\it field of definition} and the field $L$ is called the {\it coordinate field} of $\,{\rm V}_{L}(f_{j}, j\in J)$. The points of
$\,{\rm V}_{L}(f_{j}, j\in J)\cap K^{n}$ are called the {\it $K$-rational points of $V$}.
\smallskip

Note that ${\rm V}_{L}(f_{j}, j\in J)\! =\!\cap_{j\in J} {\rm V}_{L}(f_{j})$ and ${\rm V}_{L}(f_{j}, j\in J)$ depends only on the radical $\sqrt{\gotha}$ of the ideal $\gotha\!:=\!\langle f_{j}\mid j\in J\,\rangle$ generated by the family $f_{j}, j\!\in\!J$ in $K[X_{1},\ldots , X_{n}]$. By Hilbert's Basis Theorem every ideal in the polynomial ring $K[X_{1}, \ldots , X_{n}]$ is finitely generated and so there exists a finite subset $J'\subseteq J$ such that
$\gotha\!:=\!\langle f_{j}\mid j\in J'\,\rangle$. This shows that
${\rm V}_{L}(f_{j}, j\in J) \!=\!{\rm V}_{L}(f_{j}, j\in J')\!=\!\cap_{j\in J'} {\rm V}_{L}(f_{j})$. In other words, {\it every affine $K$-algebraic set in $L^{n}$ is a set of common zeros of  finitely \hbox{many polynomials.}}

\end{mypar}

\begin{examples}\label{exs:2.4}\,
Let $\,L\,\vert \,K$ be a field extension of a field $K$.
\smallskip

\begin{arlist}

\item {\bf Linear $K$-algebraic sets}\,
For linear polynomials $f_{i} = \sum_{j=1}^{n} a_{ij} X_{j} -b_{i}$,  $a_{ij}\,,\, b_{i}\in K$, $i=1,\ldots , m$, $j=1,\ldots n$, the affine $K$-algebraic set
${\rm V}_{L}(f_{1},\ldots , f_{m})$ is called a linear $K$-algebraic set defined by the $m$ linear equations $f_{1}, \ldots , f_{m}$ over $K$.  This is precisely the solution space of the system of $m$ linear equations in $X_{1},\ldots , X_{n}$ written in the matrix notation\,:
\vspace*{-2mm}
\begin{align*}
\gothA\cdot X=b\,, \  \hbox{where} \ \gothA=(a_{ij})_{{1\leq i\leq m}\atop {1\leq j \leq n}} \in
{\rm M}_{m\,,\,n}(K)\,, \
X= \begin{pmatrix} X_{1} \cr  \vdots \cr  X_{n}  \end{pmatrix} \ \hbox{and} \
b= \begin{pmatrix} b_{1} \cr  \vdots \cr  b_{m}  \end{pmatrix} \in {\rm M}_{m\,,\,1}(K)\,.
\end{align*}
Their  investigation is part of Linear algebra.  For example, if $L=K$, $r$ is the rank of the matrix $\gothA$, then ${\rm V}_{K}(f_{1},\ldots , f_{m})$ has $d=n-r$ linearly independent solutions. In fact, there is a parametric representation\,:
\vspace*{-3mm}
\begin{align*}
{\rm V}_{K}(f_{1},\ldots , f_{m}) = \{ x_{0}+\sum_{i=1}^{d}\,  t_{i}\cdot x_{i} \mid t_{1}, \ldots , t_{d}\in K\}\,,
\end{align*}
where $x_{0}\in K^{n}$ is a special solution and $ x_{i}\in K^{n}$, $i=1,\ldots , d\,$ are $d$ linearly independent solutions of the given system $\gothA\,X=b$.
\smallskip

\item {\bf $K$-Hypersurfaces}\,
For  $f\in K[X_{1},\ldots , X_{n}]$, the affine $K$-algebraic set ${\rm V}_{L}(f)=\{a\in L^{n}\mid f(a)=0\}$ is called the {\it $K$-hypersurface}  defined by $f$.
For $n=1$, since $K[X]$ is a PID, every affine $K$-algebraic set is defined by one polynomial $f\in K[X]$. Moreover,  if $L$ is algebraically closed and if $\deg f$ is positive, then  ${\rm V}_{L}(f)$ is a non-empty finite subset of $L$ of cardinality $\leq \deg f$. Furthermore, if every $a\in {\rm V}_{L}(f)$ is counted with its multiplicity $\nu_{a}(f):={\rm Min} \{r\in\N\mid f^{(r)}(a)\neq 0\}$, where for $r\in\N$, $f^{(r)}\in K[X]$ denote the (formal) $r$-th derivative of $f$, then we have a nice formula\,: $\,\displaystyle \deg f = \sum_{a\in {\rm V}_{L}(f)}\,\nu_{a}(f)\,.$ Therefore ${\rm V}_{L}(f) =\{a_{1}, \ldots , a_{r}\}$ if $f=a(X-a_{1})^{\nu_{1}}\cdots (X-a_{r})^{\nu_{r}}$ with $a\in K$ and $a_{1},\ldots , a_{r}\in L$  distinct and $\nu_{i}:=\nu_{a_{i}}(f)$, $i=1,\ldots , r$.
\smallskip

For $n=2$, $3$, $4$,  $K$-hypersurfaces are  called {\it plane curves}, {\it surfaces}, {\it $3$-folds},  defined over $K$, respectively.
\smallskip

\item
{\it If $L$ is infinite and $n\geq 1$, then the complement $L^{n}\!\smallsetminus {\rm V}_{L}(f)$ of the $K$-hypersurface ${\rm V}_{L}(f)$, $f\in K[X_{1}, \ldots , X_{n}]$, is infinite. In~particular, if $V=V_{L}(\gotha)\subsetneq L^{n}$ is a proper $K$-algebraic set, then its complement $L^{n}\!\smallsetminus\!{\rm V}_{L}(\gotha)$ is infinite.}\\[1mm]
{\bf Proof}\,  By induction on $n$. If $n=1$, then clearly ${\rm V}_{L}(f)$ is finite and hence the assertion is trivial, since $L$ is infinite. Assume that $n\geq 2$. We may assume that the indeterminate $X_{n}$ appears in $f$\,; then we have a representation\,: $f= f_{0}+f_{1}X_{n}+\cdots + f_{d} X_{n}^{d}$ with $f_{0},\ldots , f_{d}\in K[X_{1},\ldots , X_{n-1}]$,  $d>0$  and $f_{d}\neq 0$. By the induction hypothesis, we may assume that there is $(a_{1},\ldots , a_{n-1})\in L^{n}\!\smallsetminus\!{\rm V}_{L}(f_{d})$.  Then the polynomial $f(a_{1},\ldots , a_{n-1}, X_{n})\neq 0$  and hence it has only finitely many zeroes in $L$. In other words, there are infinitely many $a_{n}\in L$ such that $f(a_{1},\ldots , a_{n-1}, a_{n})\neq 0$. \dppqed
\smallskip

\item
{\it If $L$ is algebraically closed and $n\geq 2$, then every $K$-hypersurface ${\rm V}_{L}(f)$, $f\in K[X_{1},\ldots , X_{n}]$ contains infinitely many points.}\\[1mm]
{\bf Proof}\,  Since $L$ is algebraically closed, it is infinite. Further, since $n\geq 2$,  we may assume that $f$ has representation as in the above proof of (3) and hence by (3) $f_{d}(a_{1},\ldots , a_{n-1})\neq 0$ for infinitely many
$(a_{1},\ldots , a_{n-1})\in L^{n-1}$. Now, since $L$ is algebraically closed, for each of these $(a_{1},\ldots , a_{n-1})$, there exists $a_{n}$ such that $f(a_{1},\ldots , a_{n-1}, a_{n})=0$.
 \dppqed

\smallskip

\item {\bf Conic Sections}\,
 The $K$-hypersurfaces $V_{K}(f)\subseteq K^{2}$ defined by  polynomials $f(X,Y)\in K[X,Y]$ of degree $2$ are called  {\it conic sections}.\footnote{\label{foot:3}The discovery of conic sections is attributed to Menaechmus (350~B.~C.). They were intensively investigated by Apollonius of Perga (225~B.~C.).}
There are two possibilities. First $f$ is not prime, then the (degenerated)
conic $f(x,y) =0$ is a double line, or a union of two distinct straight lines. Second, $f$ is prime, in this case, we assume that $K$ is an infinite  field of ${\rm char}\,K\neq 2$. Then by an {\it affine $K$-automorphism}\,\footnote{\label{foot:4}An {\it affine $A$-automorphism} of a polynomial algebra $A[X_{1},\ldots,X_{n}]$ is an $A$-algebra automorphism $\varphi$ defined by $\varphi(X_{j})\!= \!\sum_{i=1}^{n}a_{ij}X_{i} + b_{j}\,$,
$\,1\leq j\leq n$, where $(a_{ij})\!\in\! {\rm GL}_{n}(A)$ and $(b_j)\!\in\!A^n\!$. If
$(a_{ij})$ is the identity matrix then $\varphi$ is called a
{\it translation}, if $(b_j)\!=\! 0$, then $\varphi$ is called a {\it linear $K$-automorphism}.}  of $K[X,Y]$, $f(X,Y)$ can be brought into one of the forms
$Y^2-X$, $aX^2+bY^2-1$, $a,b\in K^\times\!\!$, see Lemma\,\ref{lem:2.9}. These are called
{\it parabola}, {\it ellipse} or {\it hyperbola}
according as
$aX^2+bY^2$ is prime or not prime. Further, the defining polynomial of a hyperbola
can be transformed into $\,XY-1\,$. Note that a polynomial $aX^2+bY^2 -1$, $a,b\in K^\times\!$, is always prime and, if it has at least one zero
\footnote{\label{foot:5}Depending on the ground field $K$, it can be very difficult to decide whether such a polynomial has a zero or not.}, then it has infinitely many zeros and hence is a defining polynomial of a hyperbola or an ellipse.
\begin{align*}
\includegraphics[scale=1.2]{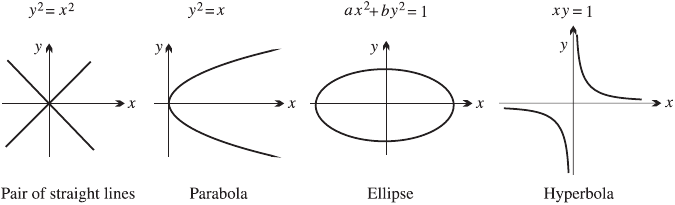}
\end{align*}
\end{arlist}
\end{examples}

\begin{mypar}\label{mypar:2.5}\,
{\bf The $K$-Ideals},  {\bf $K$-coordinate rings and (Classical) Algebra-Geometry correspondences}\,
\smallskip

We use the following notation\,: For a set $S$, let $\gothP(S)$ denote the power set of $S$,
for any ring $A$, let $\mathscr{I}(A)$ (resp. $\mathscr{R}ad\hbox{-}\mathscr{I}(A)$) denote the set of ideals (resp. radical ideals) in $A\,$  and
for a field extension $L\,\vert\, K$, let $\mathscr{A}{\rm ff}\hbox{-}\mathscr{A}{\rm lg}_{K}(L^{n})\,\subseteq \gothP(L^{n})\,$ denote the set of all affine $K$-algebraic sets in $L^{n}$.
\medskip

The definition in \ref{mypar:2.3} defines a  map
\begin{align*}
{\rm V}_{L}: \radI{K[X_{1}, \ldots, X_{n}]} \longrightarrow  \affalg{K}{L^{n}} \  \subseteq \  \gothP(L^{n}), \  \  \gotha \longmapsto {\rm V}_{L}(\gotha)\,,
\end{align*}
Note that if $L$ is infinite,  then not every subset $V\in\gothP(L^{n})$ is an affine algebraic $K$-set. For instance, if $n=1$ and if $V$ is infinite and $\neq L$.  On the other hand if $K=L$ is finite, then  every subset $V\in\gothP(L^{n})$ is an affine algebraic $K$-set.
\smallskip

To understand the map ${\rm V}_{L}$ better, for a subset $W\in\gothP(L^{n})$,  we define the {\it vanishing $K$-ideal} of $W$\,:
\vspace*{-2mm}
\begin{align*}
{\rm I}_{K}(W):= \{f\in K[X_{1},\ldots , X_{n}]\mid f(a)=0 \ \hbox{\,for all\,} \  a\in W\}.
\end{align*}
Clearly,  ${\rm I}_{K}(W)$ is a radical ideal in $K[X_{1},\ldots , X_{n}]$. The affine $K$-algebra $K[V]:=K[X_{1},\ldots , X_{n}]/{\rm I}_{K}(V)$ is called the {\it $K$-coordinate ring of $V$}. We therefore have defined the map\,:
\begin{align*}
{\rm I}_{K}: \gothP(L^{n})  \longrightarrow  \radI{K[X_{1}, \ldots, X_{n}]}, \  \  W \longmapsto {\rm I}_{K}(W)\,.
\end{align*}
With these definitions, we are looking for answers to the following questions\,:
\begin{alist}
\item
For which subsets $W\in\gothP(L^{n})$ the equality ${\rm V}_{L}\circ {\rm I}_{K}(W) = W$ holds?   The answer to this question is provided by the introduction of the {\it Zariski $K$-topology} on $L^{n}$ (\,see\,\ref{mypar:2.6} below). This topology is weaker than the usual topology (for instance, if $L=K=\R$ or $\C$).
The Zariski topology on $L^{n}$ reflects the algebraic structure of $K$-regular $L$-valued functions on $L^{n}$.   More generally, a   function $\varphi:V\to L$ on a $K$-algebraic set $V\subseteq L^{n}$ is called {\it $K$-regular function} if there exists a polynomial $f\in K[X_{1},\ldots , X_{n}]$ such that $\varphi(a)=f(a)$ for all $a\in V$. The set $\Gamma_{K}(V,L)$  of all $K$-regular functions on $V$ with values in $L$ is, obviously,  an affine $K$-algebra.  Further,  {\it for every affine $K$-algebraic set $V\subseteq L^{n}$, the canonical map
$K[X_{1},\ldots , X_{n}/{\rm I}_{K}(V) \iso \Gamma_{K}(V,L)$ is an isomorphism of $K$-algebras and the map $V \iso \Hom_{\Kalg}(K[V], L)$, $a\mapsto \xi_{a}:\overline{f}\mapsto f(a)$ is well-defined and  is bijective.}

\item
When is the composite map ${\rm I}_{K}\circ {\rm V}_{L} = \id_{\radI{K[X_{1},\ldots , X_{n}]}}$? The answer to this question is provided by the   Hilbert's Nullstellensatz (HNS\,2\, Geometric version),  see\,Theorem\,\ref{thm:2.10}\,(2).

\item
When does the system $f_{1}=0, \ldots , f_{m}=0\,$ have a common solution, i.\,e. ${\rm V}_{L}(f_{1}, \ldots , f_{m})\neq \emptyset$. The answer to this question is provided by the Hilbert's Nullstellensatz (HNS\,1),   see\,Theorem\,\ref{thm:2.10}\,(1).

 \item
 When exactly is ${\rm V}_{L}(f_{1}, \ldots , f_{m})$  a finite set?  The answer to this question is provided by the Finiteness Theorem, see\,Theorem\,\ref{thm:2.12}.
\end{alist}
\smallskip

Therefore if $L$ is algebraically closed,  then we have the algebra-geometry bijective correspondences ${\rm V}_{L}$ and ${\rm I}_{K}$ which are inclusion reversing inverses of each other\,:
\vspace*{-2mm}
\begin{align*}
\radI{K[X_{1}, \ldots, X_{n}]}  & \enskip \xlongleftrightarrow[\text{\hspace*{15mm}${\rm I}_{K}$\hspace*{15mm}}]
{\text{${\rm V}_{L}$}}    \enskip  \affalg{K}{L^{n}} \\[-1mm]
 \hspace*{-20mm} \gotha= {\rm I}_{K}({\rm V}_{L}(\gotha))  & \enskip \xlongleftrightarrow{\hspace*{33mm}}  \enskip {\rm V}_{L}(\gotha)={\rm V}_{L}({\rm I}_{K}(\gotha))\,,
\end{align*}
One can therefore study algebraic objects\,---\,ideals in the polynomial ring $K[X_{1},\ldots , X_{n}]$  and geometric objects\,---\,affine $K$-algebraic sets together by using these algebra-geometry bijective correspondences.
These correspondences are the starting point of {\it classical algebraic geometry}. HNS\,1$'$ extends the Fundamental Theorem of Algebra   to certain polynomials in many variables over algebraically closed fields.
To establish the above bijective correspondences, the fundamental step is provided by the Hilbert's Nullstellensatz (HNS\,2).
\end{mypar}

\begin{mypar}\label{mypar:2.6}\,{\bf Zariski $K$-topology}\,
It  is easy to check that the set $ \affalg{K}{L^{n}}$ of affine $K$-algebraic sets in $L^{n}$ satisfy the axioms for closed sets in a topological space.
Therefore the affine $K$-algebraic sets in $L^{n}$ form the closed sets
of a topology on $L^{n}$. This topology is called the {\it Zariski $K$-topology} on $L^{n}$. The open sets are the complements
\vspace*{-1mm}
\begin{align*}
L^{n} \!\smallsetminus\! {\rm V}_L(F_j\,,\, j\in J) =: {\rm D}_L(F_j\,,\, j\in J)
=\bigcup\nolimits_{j\in J} {\rm D}_L(F_j).
\end{align*}

In~particular, ${\rm D}_L(F)=\{a\in L^n\mid F(a)\neq 0\} = L^n\!\smallsetminus\! {\rm V}_L(F)$ for every polynomial $F\in K[X_1,\ldots , X_n]$.
These open subsets are called the {\it distinguished open subsets} in $L^n$.
{\it They form a basis for the Zariski topology on $L^n$}.
\end{mypar}

\begin{mypar}\label{mypar:2.7}\, {\bf Abstract algebraic geometry}\,
In a more general set-up one can replace affine $K$-algebraic set by the {\it prime spectrum}  $\Spec A$ of a commutative ring $A$.
\smallskip

The subsets of the form ${\rm V}(\gotha):=\{\gothp\in\Spec A\mid \gotha\subseteq \gothp\}$, $\gotha\in\radI{A}$  of $\Spec A$ are called {\it affine algebraic sets} in $\Spec A$. The subset $\zclosed{\Spec A}=\{{\rm V}(\gotha)\mid \gotha\in \radI{A}\}\subseteq \gothP(\Spec A)$   form the closed sets of a topology on $\Spec A$ which is called the {\it Zariski topology} on $\Spec A$.    This topology is not Hausdorff in general, but it is compact. The open subsets  ${\rm D}(f):=\Spec A\!\smallsetminus\!{\rm V}(Af)$, $f\in A$,  are basic open sets for the  Zariski topology on $\Spec A$.
\smallskip

Similarly, for every subset $W\subseteq \Spec A$,  we define the {\it ideal    ${\rm I}(W):=\cap_{\gothp\in W}\,\gothp$  of $W$.} This is clearly a radical ideal in $A$. Further,  the equalities\,:  ${\rm V}({\rm I}(W) ) =\overline{W}=$ the closure of $W$ in the Zariski topology of $\Spec A$ and
({\it Formal Hilbert's Nullstellensatz}) $\gotha= {\rm I}({\rm V}(\gotha))$  are rather easy to prove, see \ref{mypar:2.1}\,(3).
\smallskip

With these general definitions, we have the abstract algebra-geometry bijective correspondences  ${\rm V}$ and ${\rm I}$ which are inclusion reversing inverses of each other\,:
\vspace*{-2mm}
\begin{align*}
\radI{A} & \enskip \xlongleftrightarrow[\text{\hspace*{15mm}${\rm I}$\hspace*{15mm}}]
{\text{${\rm V}$}}    \enskip  \zclosed{\Spec A} \\[-1mm]
\hspace*{-20mm} \gotha= {\rm I}({\rm V}(\gotha))  & \enskip \xlongleftrightarrow{\hspace*{32mm}}  \enskip {\rm V}(\gotha)={\rm V}({\rm I}(\gotha))\,,
\end{align*}
One can therefore study algebra and geometry together by using this abstract algebra-geometry bijective correspondence which is the starting point of  {\it abstract algebraic geometry}.
\end{mypar}
We now prove the classical version of Hilbert's Nullstellensatz (HNS\,1$'$),  see\,Theorem\,\ref{thm:2.10}\,(1$'$).
We shall present a proof based on the ideas of the proof given by  E. Arrondo  in \cite{arrondo}. It uses the classical notion of resultant of two polynomials over  a commutative ring and the so-called ``tilting of axes lemma''  (see Lemma\,\ref{lem:2.9} below) which is of independent interest.

\begin{mypar}\label{mypar:2.8}{\bf Hilbert's Nullstellensatz}\, (\,HNS\,1$'$: \so{Classical Version}\,)\,
{\it Let $K$ be an algebraically closed field   and let  $\gotha \subsetneq K[X_{1},\ldots , X_{n}]$ be a non-unit ideal. Then $\,\mathrm{V}_{K}(\gotha)\neq \emptyset\,$.}
\end{mypar}

\begin{proof}
Let  $\gotha\subsetneq K[X_{1},\ldots , X_{n}]$ be a non-unit ideal. We shall prove that ${\rm V}_{K}(\gotha)\neq \emptyset$ by induction on  $n$.  For $n=0$, there is nothing to prove. If $n=1$,  then $\gotha =\langle f\rangle$ for some polynomial $f\in K[X_{1}]$ which is not a unit in $K[X_{1}]$, i.\,e. $f\not\in K^\times:= K\!\smallsetminus\! \{0\}$. Since $K$ is algebraically closed and $f\in K[X_{1}]\!\smallsetminus\!K^\times$,  obviously ${\rm V}_{K}(f)\neq \emptyset$.
Now, assume that  $n\geq 2$.
Since $\gotha \subsetneq K[X_{1}, \ldots , X_{n}]$, the contraction $\,\gotha':=\gotha\cap K[X_{1}, \ldots , X_{n-1}]\subsetneq K[X_{1}, \ldots , X_{n-1}]$  too. Therefore by induction hypothesis
${\rm V}_{K}(\gotha')\neq \emptyset$. Choose $a'=(a'_{1},\ldots , a'_{n-1})\in {\rm V}_{K}(\gotha')$. We consider the  surjective  $K$-algebra homomorphism $\varphi:K[X_{1}, \ldots , X_{n}]\longrightarrow K[X_{n}]$, $f\mapsto f(a'_{1},\ldots , a'_{n-1}, X_{n})$.
\medskip

$(*)$ \ We claim that
the image $\varphi(\gotha) = \{f(a'_1,\dots, a'_{n-1}, X_n)\mid f\in\gotha\}$  is a  non-unit ideal in $K[X_n]$.
\medskip

Since $\varphi$ is surjective,  $\gothb=\varphi(\gotha)$ is an ideal in $K[X_{n}]$. We now prove that $\gothb\neq K[X_{n}]$. Suppose, on the contrary that,  $\gothb=K[X_{n}]$. Then
$f(a'_1,\dots, a'_{n-1}, X_n)=1$ for some  $f=f_0+f_1X_n+\cdots+f_dX_n^d\in\gotha $, where $f_0,\dots, f_d\in K[X_1,\dots, X_{n-1}]$, $d\in\N$  with $f_{0}(a'_1,\dots, a'_{n-1})=1$ and
$f_{i}(a'_1,\dots,a'_{n-1})=0$ for every $i=1,\ldots , d$.
\smallskip

Remember that we are now looking for a contradiction. For this, since $K$ is algebraically closed, it is infinite and hence by Lemma\,\ref{lem:2.9} below,  we may assume that the ideal $\gotha$ contains  a monic polynomial
$\,g=g_{0}+\cdots+g_{r-1}X_{n}^{r-1}+X_{n}^{r}\in \gotha\,$ with  $g_{0},\ldots , g_{r-1}\in   K[X_1,\dots, X_{n-1}] $ and $r\geq 1$.
Now consider the $X_{n}$-resultant of the polynomials $f$ and $g$
\begin{align*}
{\rm Res}_{X_{n}} (f,g)= {\rm Det}\,
\left(\phantom{\begin{matrix}a_0 \\ a_{1} \\ \ddots \\a_0\\b_0\\ b_1 \\ \ddots\\b_0 \end{matrix}}
\right.\hspace{-1em}
\underbrace{\begin{matrix}
f_0 & f_1 & \cdots & f_d & 0        & 0 &  \cdots & 0\\
0   & f_0 & \cdots  & f_{d-1} &f_{d} & 0 &   \cdots  & 0 \\
\vdots  & \vdots  & \ddots & \vdots & \vdots    & \vdots  &  \ddots  & \vdots \\
0   & 0 & \cdots  & f_{0} & f_1 & f_2  &  \cdots  & f_{d}  \\
g_0 & g_1 & \cdots  &  g_{r-1} & 1     & 0 & \cdots   & 0 \\
0    & g_0  &  \cdots & g_{r-2}  &  g_{r-1} &  1  & \cdots & \cdot  \\
\vdots  & \vdots  &  \ddots & \vdots & \vdots  & \vdots   &  \ddots  & \vdots \\
0     &   0   & \cdots     & g_0 & g_1 & g_2 & \dots & 1\\
\end{matrix}}_{r+d\hbox{\small -columns}}
\hspace{-1em}
\left.\phantom{\begin{matrix}a_0\\ a_{1} \\ \ddots\\a_0\\b_0\\  b_{1} \\ \ddots\\b_0 \end{matrix}}\right)\hspace{-1em}
\begin{tabular}{l}
$\left.\lefteqn{\phantom{\begin{matrix} a_0\\ a_{1} \\ \ddots\\ a_0\ \end{matrix}}}\right\}r$-\hbox{rows}\\
$\left.\lefteqn{\phantom{\begin{matrix} b_0\\ b_{1} \\ \ddots\\ b_0\ \end{matrix}}} \right\}d$-\hbox{rows}
\end{tabular}
 \hspace*{-15mm}\in K[X_{1}, \ldots , X_{n-1}]\,.
\end{align*}
Since $f_{0}(a'_1,\dots, a'_{n-1})=1$ and
$f_{i}(a'_1,\dots,a'_{n-1})=0$ for every $i=1,\ldots , d$, ${\rm Res}_{X_{n}} (f,g) (a'_1,\dots, a'_{n-1})$ is the determinant of the lower triangular matrix with all diagonal entries equal to  $1$ and so ${\rm Res}_{X_{n}} (f,g) (a'_1,\dots, a'_{n-1})=1$. On the other hand, note that
expanding the determinant of the above $(d+r)\times (d+r)$ matrix ( the {\it Sylvester's matrix of polynomials} $f$ and $g$) by using the first column after replacing the first column by adding $X_{n}^{i}$-times the $(i+1)$-th column for all $i=1,\ldots d+r-1$, we get   ${\rm Res}_{X_{n}} (f,g) = \Phi f+\psi g$ for some polynomials $\Phi, \Psi\in K[X_{1},\ldots , X_{n-1}]$. In~particular, ${\rm Res}_{X_{n}} (f,g) \in \gotha \cap K[X_{1},\ldots , X_{n-1}]\!=\!\gotha'$ and hence ${\rm Res}_{X_{n}} (f,g)(a'_{1},\ldots , a'_{n-1})\!=\!0$, since $ (a'_{1},\ldots , a'_{n-1})\!\in\!{\rm V}_{K}(\gotha')$. This is a contradiction. Therefore  $\gothb\subsetneq K[X_n]$ and so  $\gothb\!=\!\langle h(X_n)\rangle$ for some $h(X_n)\in K[X_{n}]\!\smallsetminus\!K^{\times}$. Once again, since $K$ is algebraically closed, $h(X_n)$ has a zero $a_n\in K$.  This proves that $f(a'_1,\dots, a'_{n-1}, a_n)\!=\!0$ for all $f\in\gotha$, i.\,e. $(a'_1,\dots, a'_{n-1}, a_n)\!\in\!{\rm V}_{K}(\gotha)$.\end{proof}

\begin{lemma}[Tilting of Axes Lemma] \label{lem:2.9}\,
Let $K$ be a field and $f\in K[X_{1},\ldots,X_{n}]$ be a non-constant polynomial. Then there exists a $K$-automorphism $\varphi:K[X_{1},\ldots,X_{n}]\to K[X_{1},\ldots,X_{n}]$ such that $\varphi(X_{n})=X_{n}$ and
$f= aX_{n}^{d}+g_{d-1}X_{n}^{d-1}+\cdots +g_{0}$, where $a\in K^{\times}$ and
$g_ {j}\in K[Y_{1},\ldots,Y_{n-1}]$,  $0\leq j\leq d-1$, $Y_i:= \varphi(X_i)\,$, $1\leq i\leq n-1$. Moreover, if $K$ is infinite, then one can also choose a linear $K$-automorphism $\varphi$ satisfying the above conclusion. {\rm See Footnote~\ref{foot:4}.}
\end{lemma}

\proof First assume that $K$ is infinite. Let   $f = f_{0} + f_{1} + \cdots + f_{d}$, where
$f_{m}\in K[X_{1},\ldots,X_{n}]$ is the homogeneous component of degree $m$ of $f$, $0\leq m \leq d:=
{\rm deg}\,f$. For any $a_1, \ldots , a_{n-1}\in K$, put
$Y_{i}:= X_{i} - a_{\,i}X_{n}$, $1\leq i\leq n-1$.
Then
\vspace*{-2mm}
\begin{align*}
f & =
\sum_{m=0}^{d}f_{m}(Y_{1}+a_{1}X_{n},\ldots,Y_{n-1}+a_{n-1}X_{n}, X_{n})
   =
\sum_{m=0}^{d}\left(f_{m}(a_{1},\ldots, a_{m-1},1)\, X_{n}^{m} + \sum_{j=0}^{m-1} f_{mj}\,X_{n}^{j}\right)
\end{align*}
where $f_{mj}\in K[Y_{1}, \ldots , Y_{n-1}]$ are homogeneous polynomials of degree $m-j$.  Since  $f_{d}$ is homogeneous,  $f_{d}(X_{1},\ldots,X_{n-1},1)\neq 0$. Therefore, since $K$ is infinite, we can choose  (see Example\,\ref{exs:2.4}\,(3))  $a_{1},\ldots,a_{n-1}\in K$ such that $a:=f_{d}(a_{1},\ldots,a_{n-1},1)\neq 0$.
\smallskip

In the general case,   let $f = \sum_{\alpha\in \Lambda}a_{\alpha}X^{\alpha}$ where $\Lambda$ is a finite subset of $\N^{n}$ and $a_{\alpha}\in K^{\times}$ for every $\alpha = (\alpha_{\,1},\ldots,\alpha_{\,n})\in\Lambda$. For any positive integers $\gamma_{\,1}, \ldots , \gamma_{\,n-1}$, put $Y_{i}:= X_{i} - X_{n}^{\gamma_{\,i}}$, $\,1\leq i\leq n-1$ and $\gamma:=(\gamma_{1},\ldots , \gamma_{n-1}, 1)\in\N^{n}$. Then
\begin{align*}
f = \sum_{\alpha\in \Lambda}a_{\alpha}X_{1}^{\alpha_{\,1}}\cdots X_{n}^{\alpha_{\,n}}
= \sum_{\alpha\in \Lambda}a_{\,\alpha}(Y_{1}+X_{n}^{\gamma_{1}})^{\alpha_{\,1}}
\cdots(Y_{n-1}+X_{n}^{\gamma_{\,n-1}})^{\alpha_{\,n-1}}X_{n}^{\alpha_{\,n}}\,.
\end{align*}
For a natural number $r\in\N$
 bigger than all the components of all $\alpha= (\alpha_{\,1},\ldots ,\alpha_{\,n})\in \Lambda$, we have
 $\deg_{\,\gamma} X^{\alpha} =\alpha_{n}+\alpha_{1} r+\cdots +\alpha_{n-1} r^{n-1}$ is the $r$-adic expansion of $\deg_{\,\gamma} X^{\alpha}$ with digits $\alpha_{n}, \alpha_{1},\ldots , \alpha_{n-1}$. Therefore by the uniqueness of the $r$-adic expansion,
$\deg_{\,\gamma} X^{\alpha}\neq \deg_{\,\gamma} X^{\beta}$ for all
$\alpha, \beta\in \Lambda$, $\alpha\neq \beta$ and hence
  there exists a unique $\nu\in \Lambda$ such that $d:=\deg_{\,\gamma} F = \deg_{\,\gamma} X^{\nu}\,(>0)$. Therefore  $f = a_{\nu}X_{n}^{d} +f_{d-1}X_{n}^{d-1}+\cdots + f_{0}$ with  $f_ {j}\in
K[Y_{1},\ldots,Y_{n-1}]$.\dppqed
\smallskip

In the proof of the Tilting Axes Lemma for an infinite
field $K$, we have used a simple linear transformation of $K[X_1, \ldots , X_n]\,$.
\medskip

We now formulate several versions of Hilbert's Nullstellensatz and prove their equivalence\,:
\smallskip

We say that a field extension $E\,\vert\,K$ is of {\it finite type} if the $K$-algebra $E$ is of finite type.

\begin{theorem}\label{thm:2.10}\,{\rm  \so{(Versions of HNS)}}\,
The following statements are equivalent\,$:$

\begin{arlist}

\item {\bf HNS\,1 \  \, \,:}\,
Let $L\,\vert\, K$ be an algebraically closed field extension of  a field $K$ and let $\gotha \subsetneq K[X_{1},\ldots , X_{n}]$ be a non-unit ideal.
Then $\,\mathrm{V}_{L}(\gotha)\neq \emptyset\,$.
\smallskip

\item[{\bf (1$'$)}]{\bf HNS\,1$'$\,\,:} {\rm \so{(Classical Version)}}\,
Let $K$ be an algebraically closed field   and let $\gotha \subsetneq K[X_{1},\ldots , X_{n}]$ be a non-unit ideal.
Then $\,\mathrm{V}_{K}(\gotha)\neq \emptyset\,$.
\smallskip

\item[{\bf (1$''$)}]{\bf HNS\,1$''$:}\,  Let $L\,\vert\, K$ be an algebraically closed field extension of  a field $K$  and let $A$ be a nonzero $K$-algebra of finite type. Then $\Hom_{\Kalg}(A,L)\neq \emptyset$.
\smallskip

\item{\bf HNS\,2\,\,\,:} {\rm \so{(Geometric Version)}}\,
Let $L\,\vert\, K$ be an algebraically closed field extension of a field  $K$ and let $\gotha \subseteq K[X_{1},\ldots , X_{n}]$ be an ideal. Then $\mathrm{I}_{K}(\mathrm{V}_L(\gotha))= \sqrt{\gotha}\,$.
\smallskip

\item {\bf HNS\,3\,\,\,:} {\rm \so{(Field Theoretic Form\hbox{\,---\,}Zariski's Lemma)}}\,
Let $K$ be a field and $E\,\vert\,K$ be a finite type field extension of $K$. Then  $\,E\,\vert\,K$ is algebraic. In~particular, $E\,\vert\,K$ is finite.
\smallskip

\item[{\bf  (3$'$)}]{\bf HNS\,3$'$\,:}\, If $K$ is an algebraically closed field, then the map
\begin{align*}
\xi: K^{n} \longrightarrow \Spm(K[X_{1},\ldots , X_{n}])\,,\, \ a=(a_{1},\ldots , a_{n})\longmapsto \gothm_{a}:=\langle X_{1}-a_{1}, \ldots , X_{n}-a_{n}\rangle\,,
\end{align*}
is bijective.
\end{arlist}

\end{theorem}

\begin{proof} To prove the equivalence of these statements, we shall prove the implications\,:
 \begin{center}
\begin{tabular}{clcll}
$(3')$                 & $\Longleftarrow$     &    $(3)$             &          &          \\
$\Downarrow$  &                        &   $\Uparrow$    &         &          \\
$(1') $                &  $\Longrightarrow$   &  $ (1)$              &  $\Longleftrightarrow$  & $(1'')$   \\
   &                        &    $\Updownarrow $  &         &          \\
                &     &  $ (2)$              &    &    \\
\end{tabular}
\end{center}
{\bf (1$'$) $\Rightarrow$ (1)\,:}
Since $L$ is algebraically closed, there is an algebraic closure $\overline{K}$ of $K$ with $\overline{K} \subseteq L$. Further, since $K[X_{1}, \ldots , X_{n}]\subseteq \overline{K}[X_{1}, \ldots , X_{n}]$ is an integral extension and $\gotha \subsetneq K[X_{1},\ldots , X_{n}]$ is a non-unit ideal, by the Proposition in (2.1)\,(4) the extended ideal $\gotha\, \overline{K}[X_{1}, \ldots , X_{n}]$ is also a non-unit ideal in $\overline{K}[X_{1}, \ldots , X_{n}]$. Therefore by (1$'$) we have  $\,\emptyset\neq {\rm V}_{\overline{K}}(\gotha\,\overline{K}[X_{1}, \ldots , X_{n}])={\rm V}_{\overline{K}}(\gotha) \subseteq {\rm V}_{L}(\gotha)$.
\smallskip

{\bf (1) $\Rightarrow$ (3)\,:} By the given condition,  $E=K[X_1, \dots, X_n]/\gothm$  with  $\gothm\in\Spm K[X_1, \dots, X_n]$ and hence by (1) (applied to $L=\overline{K}$ the algebraic closure of $K$) there exists $a=(a_{1},\ldots , a_{n})\in L^{n}$ such that $a\in{\rm V}_{\overline{K}}(\gothm)$. Now, clearly the (substitution) $K$-algebra  homomorphism $\varepsilon_{a}: K[X_1, \dots, X_n]\longrightarrow \overline{K}$, $X_{i}\mapsto a_{i}$, $i=1,\ldots , n$, has kernel $\!=\!\gothm_{a}\!=\!\gothm$ and hence $\varepsilon_{a}$ induces an injective $K$-algebra homomorphism $E\!=\!K[X_1, \dots, X_n]/\gothm \longrightarrow \overline{K}$. Therefore $E$ is algebraic over $K$.
\smallskip

{\bf (3) $\Rightarrow$ (3$'$)\,:}  Clearly, the map $\xi$ is always injective. To prove that $\xi$ is surjective, if $K$ is an algebraically closed field, let $\gothm\in \Spm (K[X_1, \dots, X_n])$. Then $E:=K[X_1, \dots, X_n]/\gothm$ is a finite type field extension of $K$ and hence $E\,\vert\,K$ is algebraic by (3). Since $K$ is algebraically closed,  $E=K$ and hence there exists $(a_{1},\ldots , a_{n})\in  K^{n}$ such that $X_{i} \equiv~a_{i}~ {\rm mod}~ \mathfrak{m}$, for every $i=1,\ldots , n$. This proves that $\gothm_{a}\subseteq \gothm$ and hence $\gothm=\gothm_{a}\in\im\,\xi$, since $\gothm_{a}$ is maximal.
\smallskip

{\bf (3$'$) $\Rightarrow$ (1$'$)\,:} Since $\gotha$ is a non-unit ideal in $K[X_1, \dots, X_n]$,  by Krull's Theorem there exists a maximal ideal $\gothm\in\Spm K[X_1, \dots, X_n]$ with  $\gotha \subseteq \gothm$. Now, by (3$'$) $\gotha\subseteq \gothm=\gothm_{a}$ for some $a\in K^{n}$, or equivalently $a\in{\rm V}_{K}(\gotha)$. \smallskip

{\bf (1) $\iff$ (1$''$)\,:} Let $\,L\,\vert\,K$ be an algebraically closed field extension of $K$.
Let $A=K[x_1,\dots,x_n]$ be a $K$-algebra of finite type and $\varepsilon_{x}: K[X_1,\dots, X_n]\to A$  be the surjective $K$-algebra homomorphism with $\varphi(X_i)=x_i$ for all $i=1,\dots,n$, and  $\gotha=\Ker \varepsilon_{x}$. Note that  for $a=(a_1,\dots, a_n)\in L^{n}$, the substitution homomorphism $\varepsilon_{a} : K[X_1,\dots, X_n]\longrightarrow L$ induces a $K$-algebra homomorphism
$\varphi: A=K[x_{1},\ldots , x_{n}]\longrightarrow L$ such that the diagram
\vspace*{-1mm}
\begin{align*}
\xymatrixcolsep{5pc}\xymatrixrowsep{3.75pc}
\xymatrix{
K[X_{1},\ldots , X_{n}] \hspace*{2mm}   \ar[d]_{\varepsilon_{x}} \ar[r]^{\hspace*{6mm}\varepsilon_{a}}  & \hspace*{3mm} L \\
A=K[x_{1},\ldots , x_{n}] \hspace*{3mm} \ar@{-->}[ru]_{\varphi}  &  \\
}
\end{align*}
is commutative if and only if $\gotha =\Ker \varepsilon_{x} \subseteq \Ker \varepsilon_{a}$, or equivalently $\varepsilon_{a}(\gotha)=0$,i.\,e.  $a\in {\rm V}_{L}(\gotha)$.
\smallskip

{\bf (1) $\Rightarrow$ (2)\,:} Clearly, $\gotha \subseteq {\rm I}_{K}({\rm V}_{L}(\gotha)).$ Conversely, suppose that $f\in {\rm I}_{K}({\rm V}_L(\gotha))$.  Put
$g=1-X_{n+1}f\in K[X_1, \dots, X_{n+1}]$ and consider the $K$-algebraic set $W={\rm V}_L(\langle \gotha, g\rangle)\subseteq L^{n}$. If $(a,a_{n+1})\in W$ with  $a\in L^{n}$ and $a_{n+1}\in L$, then $a\in {\rm V}_L(\gotha)$. Therefore $f(a)=0$, since $f\in{\rm I}_{K}({\rm V}_L(\gotha))$  and so $0=g(a,a_{n+1})=1\!-\!a_{n+1} f(a)\!=\!1$,  a contradiction. This proves that $W\!=\!\emptyset$  and hence $\langle \gotha, g\rangle$ is the unit ideal in $K[X_1,\dots, X_{n+1}]$ by (1). Therefore  there exist $f_1,\dots, f_r\!\in\!\gotha$ and  $h_1, \dots, h_r, h\!\in\! K[X_{1},\ldots , X_{n+1}]$ such that
\vspace*{-2mm}
\begin{align*}
1=\sum_{i=1}^r h_i(X_1, \dots, X_{n+1})\, f_i+ h(X_1, X_2,\dots, X_{n+1})\,g\,.
\end{align*}
Since $g(X_{1},\ldots , X_{n}, 1/f) =0$, substituting  $X_{n+1}=1/f$ in the above equation we get\,:
\begin{align*}
1=\sum_{i=1}^r h_i(X_1, \dots,  X_n, 1/f)\, f_i\,
\end{align*}
Now, clearing the denominator in all $h_i(X_1, \dots,  X_n, 1/f)$, $i=1,\ldots , r$, we get $f^{s} \in \langle f_{1},\ldots , f_{r}\rangle \subseteq \gotha$ for some $s\in\N$, $s\geq 1$. This proves that $f\in\sqrt{\gotha}$.\\
{\bf Remark\,:}
The idea to use an additional indeterminate was introduced by J.\,L.\,Rabinowitsch \cite{rabinowitsch}
and is known as {\it Rabinowitsch's trick}.
\smallskip

{\bf (2) $\Rightarrow$ (1)\,:}
Let  $\gotha$ be a non-unit ideal in $K[X_1, \dots, X_n]$. If ${\rm V}_L(\gotha)=\emptyset$, then by (2) $\sqrt{\gotha} = {\rm I}_{K}({\rm V}_L(\gotha))= K[X_1, \dots, X_n]$ and hence $1\in \gotha$ which contradicts the assumption.
\end{proof}

\begin{remarks}\label{rem:2.11}\,
(1)\, In 1947 Oscar Zariski \cite{zariski} proved the following  elegant result\,:
\so{(Zariski's Lemma)}\, {\it Let   $A$ be an algebra of finite type over a field $K$ and $\gothm\in\Spm A$. Then  $A/\gothm$ is a finite field extension of $K$.} In spite of its innocuous appearance, it is a useful result in affine algebras over any field.
\smallskip

(2)\, Note that since we have already  proved the classical version of Hilbert's Nullstellensatz (HNS\,1$'$)  in \,\ref{mypar:2.8}, Theorem\,\ref{thm:2.10} proves all the versions of Hilbert's Nullstellensatz.
\smallskip

(3)\, One can also prove that the following general form of Hilbert's Nullstellensatz is equivalent to any one of the forms of HNS mentioned in Theorem\,\ref{thm:2.10}\,:
\smallskip

{\bf HNS\,4\,\,\,:} {\rm \so{(General Form)}}\, {\it Let $A$ be a Jacobson ring\footnote{\label{foot:7}Recall that a ring $A$ is called a {\it Jacobson ring} if every prime ideal in $A$ is the intersection of maximal ideals in $A$. Clearly, fields and the ring of integers, $\Z$ are Jacobson rings. Further, by Zariski's Lemma\,\ref{rem:2.11}\,(1)  every finite type algebra over a field $K$ is a Jacobson ring.\,---\,The name {\it Jacobson ring} is used by Wolfgang Krull (1899-1971) to honour Nathan Jacobson (1910-1999). The name {\it Hilbert ring} also appears in the literature.}
 and let $B$ be an $A$-algebra of finite type. If $\gothn\in\Spm B$ is a maximal ideal in $B$, then its  contraction $\gothm=A\cap \gothn\in\Spm A$ is also a maximal ideal in $A$ and the residue  field  $B/\gothn$ of $B$ is a finite field extension of the residue field $A/\gothm$ of $A$.}
\end{remarks}

We give two applications of HNS\,2 (see  \cite{sturmfels})  for solving systems of polynomial equations with finitely many solutions.
First, we prove a criterion for a system of polynomial equations to have finitely many solutions.

\begin{theorem}\label{thm:2.12}\,{\rm \so{(Finiteness Theorem)}}\,   Let $K$ be an algebraically closed field and $\gotha$ be a non-unit ideal in $K[X_{1},\ldots , X_{n}]$. Then the following statements are equivalent$\,:$

\begin{rlist}

\item \ ${\rm V}_{K}(\gotha)$ is a finite set.
\item  $\,\,K[X_{1},\ldots , X_{n}]/\gotha$ is a finite dimensional $K$-vector space.
\item There exist polynomials  $f_{1}(X_{1}),\ldots, f_{n}(X_{n}) \in\gotha$.
\end{rlist}
\end{theorem}

\begin{proof}
Since $\left(K[X_{1},\ldots , X_{n}]/\gotha\right)_{{\rm red}}  = K[X_{1},\ldots , X_{n}]/\sqrt{\gotha}$,
$K[X_{1},\ldots , X_{n}]/\gotha$ is Artinian if and only if $K[X_{1},\ldots , X_{n}]/\!\sqrt{\gotha}$ is Artinian. Also ${\rm V}_{\!K}(\gotha)\!=\!{\rm V}_{\!K}(\sqrt{\gotha})$. So we  may assume that $\gotha$ is a radical ideal.
\smallskip

(i) $\Rightarrow $ (iii)\,:
Suppose that ${\rm V}_{K}(\gotha)=\{a_{1}=(a_{1i}), \ldots , a_{r}=(a_{ri})\} \subseteq  K^{n}$ is finite. Consider the polynomials   $f_{i}(X_{i})=(X_{i}-a_{1i})\cdots (X_{i}-a_{ri})$, $i=1,\ldots , n$.
Clearly,  $f_{i}(X_{i})$ vanishes on ${\rm V}_{K}(\gotha)$ and hence $f_{i}(X_{i})\in {\rm I}_{K}({\rm V}_{K}(\gotha))=\gotha$ by HNS\,2.
\smallskip

(iii) $\Rightarrow $ (i)\,:  By (iii) ${\rm V}_{K}(\gotha) \subseteq \cap_{i=1}^{\,n} {\rm V}_{K}(f_{i}(X_{i}))$ which is  finite of cardinality $\leq \prod_{i=1}^{\,n} \deg f_{i}$.
\smallskip

(ii) $\Rightarrow $ (iii)\,: Since $K[X_{1},\ldots , X_{n}]/\gotha = K[x_1,\dots,x_n]$ is a finite dimensional $K$-vector space by (ii), $x_{1},\ldots , x_{n}$ are algebraic over $K$. Let $\mu_{x_{i}}\in K[X]$ be the minimal polynomial of $x_{i}$  over $K$ and $f_{i}(X_{i}):=\mu_{x_{i}}(X_{i})$,  $i=1,\ldots , n$. Then $f_{i}(X_{i})\in\gotha$, since $f_{i}(x_{i})=\mu_{x_{i}}(x_{i})=0$ in $K[x_{1}, \ldots , x_{n}]=K[X_{1}, \ldots , X_{n}]/\gotha$ for all  $i=1,\ldots , n$.
\smallskip

(iii) $\Rightarrow $ (ii)\,: Let $d_i = \deg f_i$ for all $i=1,\dots,n$, and let $d=\max\{d_{1},\ldots , d_{n}\}$. Then for all $i$, $x_i^d$ can be expressed as a polynomial in $x_i^m$ for $m=0,1,\dots,d_i-1$. Since the monomials in $x_1,\dots,x_n$ form a generating set of the $K$-vector space $K[x_1,\dots,x_n]$, it follows that it is a finite dimensional $K$-vector space.
\end{proof}

As an application of HNS\,2 we describe finite $K$-algebraic sets in $K^{n}$ over an algebraically closed field $K$ using the eigenvalues of some commuting linear operators. For $n=1$, this is  taught in the undergraduate course on linear algebra, namely\,: {\it Let
$ f(X)=X^n+a_{n-1}X^{n-1}+\cdots+a_1X+a_0\in K[X]$ be a monic polynomial of degree $n$ over a field $K$ and $a\in K$. Then $a\in {\rm V}_{K}(f)$ if and only if $a$ is an eigenvalue of the $K$-linear operator $\lambda_{x}:K[x]\to K[x]$, $y\mapsto x\,y$ on the $K$-vector space  $K[x]:=K[X]/\langle f(X)\rangle$ of dimension $n=\deg f$.} For a proof  use division with remainder in $K[X]$ to note that $a\in {\rm V}_{K}(f)$ if and only if $f(X)=(X-a)g(X)$ with  $g(X)\in K[X]\!\smallsetminus\!\langle f(X)\rangle$, equivalently, $0=f(x) = (x-a)g(x) =\lambda_{x}(g(x)) - a\,g(x)$ with
$g(x)\neq 0$ in $K[x]$, i.\,e. $\lambda_{x}(g(x)) = a\,g(x)$ with
$g(x)\neq 0$ which means $a$ is an eigenvalue of $\lambda_{x}$ with eigenvector $g(x)$.
\smallskip

Ludwig Stickelberger
generalized the above observation to the case of polynomials in $n$ indeterminates over an algebraically closed field. More precisely, we prove the following\,:

\begin{theorem}\label{thm:2.13}{\rm \so{(Stickelberger})}\,
Let $K$ be an algebraically closed field and let $\gotha\subseteq K[X_1, \ldots,X_n]$ be a radical ideal in $K[X_1, \ldots,X_n]$ with ${\rm V}_{K}(\gotha)$ finite. Then there exists $0\neq g(x):=g(x_{1},\ldots , x_{n})\in K[x_{1},\ldots , x_{n}]=K[X_{1},\ldots , X_{n}]/\gotha$ such that $\,{\rm V}_{K}(\gotha) =\{a_1=(a_{1\,i})\,\ldots , a_{r}=(a_{r\,i})\}$, where for each $j=1,\ldots ,r$, the $i$-th coordinate $a_{j\,i}$ of $a_{j}$ is an eigenvalue of $\lambda_{x_{i}}$ with eigenvector $g(x)$, {\rm i.\,e.}
$x_{i}\,g(x)= \lambda_{x_{i}}(g(x)) =a_{j\,i}\,g(x)$ for each $j=1,\ldots , r$ and for all $i=1, \ldots , n$.
\end{theorem}

\begin{proof}
Suppose that ${\rm V}_{K}(\gotha)=\{a_{1}=(a_{1\,i}), \ldots , a_{r}=(a_{r\,i})\} \subseteq  K^{n}$ is finite.   For every $j=2,\ldots, r$, there exists $k_j$ such that $a_{1\,k_j} \neq a_{j\,k_j}$ and hence $g_1(X_1,\dots,X_n) := \prod_{j=2}^r (X_{k_j} - a_{j\,k_{j}})/(a_{1\,k_{j}} - a_{j\,k_{j}})\not\in\gotha$, since $g_{1}(a_{1})\neq 0$. Further, $(X_{i}g_{1}-a_{j\,i}\,g_{1})(a_{j})= a_{j\,i}\,g_{1}(a_{j}) -a_{j\,i}\,g_{1}(a_{j})= 0$ for all $i=1,\ldots , n$. Therefore $g_{1}(x)=g_{1}(x_{1},\ldots , x_{n})\neq 0$,  $X_{i}g_{1}-a_{j\,i}g_{1}\in {\rm I}_{K}({\rm V}_{K}(\gotha))=\sqrt{\gotha} =\gotha$ by HNS\,2 and hence $x_{i}\,g_{1}(x)= \lambda_{x_{i}}(g_{1}(x)) = a_{j\,i}\,g_{1}(x) =0$ for all $j=1,\ldots r$.
Conversely, let $b=(b_1, \ldots,b_n)\in K^{n}$ be such that
$x_i\,g(x)=\lambda_{x_{i}}( g(x) )=b_{i}\, g(x)$ for some  $0\neq g(x) \in K[X_{1},\ldots , X_{n}]/\gotha$ for all $i=1,\ldots , n$. Then
 $\,b_{1}^{\nu_{1}}\cdots b_{n}^{\nu_{n}}\, g(x)= (\lambda_{x_{1}}^{\nu_{1}}\circ\cdots \circ\lambda_{x_{n}}^{\nu_{n}})(g(x))=x_{1}^{\nu_{1}}\cdots x_{n}^{\nu_{n}} g(x)$ for every $\nu=(\nu_{1},\ldots , \nu_{n})\in\N^{n}$, and hence   $f(b_{1},\ldots , b_{n})\,g(x)= f(x)\, g(x)=0$ for every $f\in\gotha$. Therefore, since   $f(b_{1},\ldots , b_{n})\in K$ and $g(x)\neq 0$ in the $K$-vector space $K[x_{1}, \ldots , x_{n}]$, it follows that  $f(b_{1},\ldots , b_{n})=0$ for every $f\in\gotha$, i.\,e. $(b_{1},\ldots , b_{n})\in {\rm V}_{K}(\gotha)$.
\end{proof}

\begin{mypar}\label{mypar:2.14}{\bf Examples of PIDs which are not EDs}\,
In most textbooks it is stated that there are examples of principal ideal domains which are not Euclidean domains. However, concrete examples are almost never presented with full details.  In this subsection we use HNS\,3 to give a family of  such examples with full  proofs which are  accessible even to undergraduate students. The main ingredients are computations of the unit group $A^\times$  and  the $\KSpec A$  for affine algebras over a field $K$. First we recall some standard definitions and preliminary results\,:

{\small

\begin{arlist}

\item {\bf Unit Groups}\,
For a ring $A$, the group $A^\times$ of the invertible elements in the
multiplicative monoid $(A,\cdot)$ of the ring $A$ is called the
{\it unit group}\,;  its elements are called the {\it units} in $A$.
The determination of the unit group of a ring is an interesting problem which is not always easy.
Some~simple examples are\,:
$\Z^\times\!\!=\!\{-1,1\}$;  if  $n\geq 2$, then $\Z_{n}^\times \!\!=\!\{m\in\N\mid 0\leq m < n  \ \hbox{and}\ \gcd(m,n)\!=\!1\}$;  if $K$ is
a field then $K^\times\!=\!K\!\smallsetminus\!\{0\}\,$; if $A$ is an integral domain, then  $(A[X_1,\ldots, X_n])^\times\!\!=\!A^\times$; if $K$ is a field, then
$(K[T, T^{-1}])^\times\!\!=\!\{\lambda T^n\mid \lambda
\in K^\times \hbox{and} \ n\in\Z\} \cong$ the product group
$K^\times \times \Z\,$;
$(A[\![X_1,\ldots, X_n]\!])^\times\!\! = \!\{f\in A[\![X_1,\ldots, X_n]\!]\mid f(0) \in A^\times \}$.
 \smallskip

\item {\bf Norm}\,
The notion of the {\it norm} is very useful for the determination of the unit groups of some domains.  Let $R$ be a (commutative) ring and let
$A$ be a finite free $R$-algebra. For $x\in A$, let
$\lambda_x:A\to A$ denote the (left) multiplication by $x$. The {\it norm} map
${\rm N}^{A}_{R}: A\to R$, $x\mapsto \Det \lambda_{x}$,  contains important information about the multiplicative structure of $A$ over $R$. The following
properties of the norm map are easy to verify\,:
\smallskip

{\it   The norm map ${\rm N}^{A}_{R}:A\to R$ is
multiplicative, {\rm i.\,e.} ${\rm N}^{A}_{R}(xy)\!= \!{\rm N}^{A}_{R}(x) \cdot {\rm
N}^{A}_{R}(y)\,$ for all $\,x,y\in A$,
${\rm N}^{A}_{R}(a) \!=\! a^{n}$ for every $a\in R$, where $
n:=\!{\rm Rank}_{R}(A)$.  Further, for  an element $x\in A$,  $x\in A^\times$
if and only if ${\rm N}^{A}_{R}(x)\in R^\times$.}
\smallskip

\item  In the following examples we shall illustrate the use of the norm map to compute the unit group.

\begin{alist}

\item {\it Let $\varphi(X)\!\in \!\R[X]$ be a non-constant polynomial with positive leading coefficient,  $\Phi:\!=\!Y^2\!+\!\varphi(X)\!\in\!
\R[X,Y]$ and let $A:=\R[X,Y]/(\Phi)$. Then  $A$ is an affine
domain (over $\R$) of (Krull) dimension $1$ and $A^\times = \R^\times$.}
\smallskip

{\bf Proof}\,  Let $x,y\in A$ denote the images of $X, Y$ in $A$
respectively. Then  $A$ is a free $\R[X]$-algebra of rank $2$ with $R$-basis
 $1,y$, i.\,e. $A=\R[X] + \R[X]\cdot y$ and $y^2
= - \varphi(X)$. Further, let ${\rm N}:={\rm N}^{A}_{\R[X]}:A\to
\R[X]$ denote the norm-map of $A$ over $\R[X]$. Then
${\rm N}(F+Gy) = {\rm Det}\,
\begin{pmatrix}
F & -G\, \varphi \cr
G & F
\end{pmatrix} =
F^2+G^2\varphi$ for every $F,G\in \R[X]$. Therefore $F+Gy \in
A^\times$ if and only if $F^2+G^2\varphi \in \R[X]^\times
=\R^\times\!\!$, equivalently, $F\in \R^\times$ and $G=0$, since
the leading coefficient of $\varphi$ is positive by assumption. This
proves that $A^\times = \R^\times$.
\smallskip

\item  The $\R$-algebras ${\rm P}:= \R[X,Y]/(Y^2-X) \cong \R[Y]$, ${\rm H}:=
\R[X,Y]/(X^2-Y^2-1) \cong   \R[X,Y]/(XY-1)\cong \R[Z, Z^{-1}] $,
${\rm K}:=\R[X,Y]/(X^2+Y^2-1)$  and ${\rm
L}_{b,c}:=\R[X,Y]/(Y^2+bX^2+c)$ with $b,c\in\R, b> 0$, are all affine domains (over $\R$) of dimension one and
$\,{\rm H}^\times\cong \R^\times \times \Z$, \
${\rm P}^\times ={\rm K}^\times = {\rm L}_{b\,,\,c}^\times = \R^\times$.
\end{alist}
\smallskip

\item {\bf Euclidean functions and Euclidean domains}\,
Let $A$ be an integral  domain. A {\it Euclidean function} on $A$ is a map
$\delta:A\!\smallsetminus\!\{0\}\to \N$ which satisfies the following
property\,:  for every two elements $a,b\in A$ with $b\neq 0$
there exist elements $q$ and $r$ in $A$ such that
$a=qb+r$ and either \ $r=0$ \  or \ $\delta(r)<\delta(b)$.
If there is a Euclidean function $\delta$ on $A$, then $A$ is called a
{\it Euclidean domain} (with respect to $\delta$). For example,
the usual absolute value function $| \cdot |:\Z\!\smallsetminus\!\{0\} \to
\N$, \ $a\mapsto |a|$ is a Euclidean function on the
the ring of integers $\Z$\,; For a field $K$, the degree
function $f\mapsto \deg f$ is a Euclidean function on the
the polynomial ring $K[X]$\,; the order function $f\mapsto \ord \,f$, is a Euclidean function on the the formal power series ring $K[\![X]\!]$.
\smallskip

Note that in the definition of a Euclidean function on $A$,
many authors also include the condition that $\delta$ respect the multiplication, i.\,e. $\delta(ab) \geq
\delta(a)$ for all $a,b\in A\!\smallsetminus\!\{0\}$. However, if $A$
is a Euclidean domain, then there exists a so-called {\it minimal Euclidean function} $\delta$ on
$A$ which respects the multiplication and the equality $\delta(ab) =
\delta(a)$ for $a,b \in A\!\smallsetminus\!\{0\}$  holds if and only if $b\in A^\times$.  For  a proof we recommend the reader to see the beautiful article by P.\,Samuel \cite{samuel}.

In a Euclidean domain, any two elements have a gcd which can be effectively computed by {\it Euclidean algorithm.} In~particular, Euclidean domains are principal ideal domains and hence unique factorization domains.
\medskip

\item
We shall show that the coordinate ring of the circle over real numbers is not a principal ideal domain. More precisely\,:  {\it The $\R$-algebra ${\rm K}
=\R[X,Y]/(X^2+Y^2-1)$ is not a principal ideal domain.}
\smallskip

{\bf Proof} Let $x,y$ denote the images of $X,Y$ in ${\rm K}$,
respectively. In fact we will show that the maximal ideals
$\gothm_{(a,b)}={\rm K}(x-a)+{\rm K}(y-b)\,, a,b\in\R\,, a^2+b^2=1$,
corresponding to the $\R$-rational points
of ${\rm K}$ are not principal. To prove this we may assume  without loss of
generality that $a=1$ and $b=0$. We use the fact
that ${\rm K}$ is a quadratic free algebra over $\R[X]\,$ with
basis $1,y$ with $y^2= 1-X^2$. Let ${\rm N} = {\rm N}^{{\rm
K}}_{\R[X]}:{\rm K}\to \R[X]$ be the norm-map of ${\rm K}$ over
$\R[X]$. First note that ${\rm N}(y) =X^2-1$ and  if $f=\varphi +\psi y\in {\rm K}^{\times}$,
$\varphi, \psi\in \R[X]$, then ${\rm N}(f) = \varphi^2+\psi^2
(X^2-1)$, in~particular, either ${\rm N}(f) \in \R^\times\!$ or  $\deg\, ({\rm N}(f)) \geq 2$.
Now, suppose that
$\gothm:=\gothm_{(0,1)}$ is principal, say generated by an element
$f\in {\rm K}$. Then $f$ is a non-unit in ${\rm K}$ and  $x-1 = gf$, $y =hf$ for some
$g,h\in {\rm K}$. Therefore $(X-1)^2={\rm N}(x-1) = {\rm N}(g) {\rm
N}(f)$ and $X^2-1= {\rm N}(y) ={\rm N}(h){\rm N}(f)$ and hence
${\rm N}(f)\,$ divides $\,\gcd\left((X-1)^2, X^2-1\right) = X-1$,
in~particular, $\deg ({\rm N}(f)) =1$ (since $f\not\in{\rm
K}^\times$, ${\rm N}(f)\not\in \R^\times$)  which is impossible.
Therefore $\gothm$ is not a principal ideal in ${\rm K}$. \dppqed
\end{arlist}

}

\end{mypar}

Now we prove the following simple key observation that the $K$-Spectrum of an affine domain $A$ over a field $K$ which is a  Euclidean domain with a small unit group is non-empty. More precisely\,:

\begin{proposition}\label{prop:2.15}\,
Let $A$ be an affine domain over a field $K$. If $A$ is an Euclidean domain, then there exists a maximal ideal $\gothm\in\Spm A$ such that the  natural group homomorphism  $\pi^\times: A^\times \rightarrow (A/\gothm)^\times$ $($which is the restriction of  the canonical surjective map $\pi:A\rightarrow A/\gothm\,) $ is surjective. In~particular, if $A$ is an Euclidean domain with $A^\times =K^\times\!\!$, then $\KSpec A\neq \emptyset$.
\end{proposition}

\begin{proof}
Suppose that $A$ is an Euclidean domain and that $\delta \colon A\setminus\{0\} \rightarrow \N$ is a minimal Euclidean function on $A$. Then choose an element $f \in A$ such that $\delta(f):=\Min \{ \delta(a)\mid 0\neq a\in \left(A\!\smallsetminus\!A^\times\right)\,\}$. Such an element $f$ exists, since the ordered set $(\N,\leq)$ where $\leq$  is the usual order on $\N$, is well ordered.
We claim that $f$ is irreducible. For, if $f=gh$ with $g,h \in A$,  then
$\delta(f)=\delta(gh) \ge \delta(g)$. In the case $\delta(f) > \delta(g)$,  $g \in A^\times$ by the minimality of $\delta(f)$. In the case $\delta(gh)=\delta(f)=\delta(g)$,  $h\in A^\times$.  Therefore, $\gothm=Af$ is a non-zero prime ideal and hence a maximal ideal in $A$.
To prove that $\pi^{\times}: A^\times \to (A/\gothm)^\times$  is surjective, let  $z\in (A/\gothm)^\times$. Then  $z=\pi(g)$ for some $g\in A$ and $g\notin\gothm$. Use the Euclidean function $\delta$ to  write $g=fq+r$ with $q,r \in A$ and either $r=0$ or $\delta(r)<\delta(f)$.  Since $z\neq 0$, i.\,e. $g\not\in\gothm$,  we must have  $r\neq 0$ and hence $\delta(r)<\delta(f)$. But, then by the minimality of $\delta(f)$,  $r\in A^\times$ and $z=\pi(g)=\pi(r)=\pi^{\times}(r)$.
\end{proof}

We can reformulate the above Proposition in the language of  algebraic
geometry as\,:

\begin{corollary}\label{cor:2.16}\,  Let ${\cal C}$ be an affine algebraic irreducible
curve over a field $K$. If ${\cal C}$ has no $K$-rational points and the unit group of the coordinate ring $K[{\cal C}]$ of $\,{\cal C}$ is $K^\times\!\!\!$, then  $K[{\cal C}]$ is not  a Euclidean domain.
\end{corollary}

For the $\R$-affine domains $\,{\rm H}$ and ${\rm K}$ in Examples\,\ref{mypar:2.14}\,(3) (b), the assumptions in Proposition\,\ref{prop:2.15} are not satisfied, but
${\rm H}$ is a Euclidean domain and ${\rm K}$ is not a  Euclidean
domain, in fact,  not even a PID,  see\,\ref{mypar:2.14}\,(5).

\begin{corollary}\label{cor:2.17}\,
Let $\varphi(X)\in \R[X]$ be a non-constant polynomial with $\varphi(\alpha)>0$ for every
$\alpha\in \R$ and let $\Phi:=Y^2+\varphi(X)\in
\R[X,Y]$. Then the affine domain $A:=\R[X,Y]/(\Phi)$ is not a
Euclidean domain. In~particular, ${\rm L}_{b,c}=\R[X,Y]/(Y^2+bX^2+c)$ with $b$, $c\in\R$, $b\!>\!0$, $c\!>\!0$  is not  a Euclidean domain.
\end{corollary}

\begin{proof}
Note that  $A^\times= \R^\times$ by \ref{mypar:2.14}\,(3)\,(a) and
$\RSpec A = \{(\alpha, \beta) \in\R^2\mid \Phi(\alpha, \beta) =0\} =\emptyset\,$ by the assumption on $\varphi$.
Therefore $A$ can not be a Euclidean domain by Corollary\,\ref{cor:2.16}.
\end{proof}

In the following theorem, we give a criterion for the affine $\R$-domain  ${\rm L}_{b\,,\,c}$ to be  a principal ideal domain, see\,\ref{mypar:2.14}\,(3)\,(b).

\begin{theorem}\label{thm:2.18}\,
Let $b$, $c\in\R$, $b\!>\!0$ and  $c\!\neq\! 0$. Then the affine domain ${\rm L}_{b\,,\,c}\!:=\!\R[X,Y]/\langle Y^2\!+bX^2\!+c\rangle$ over $\R$ is a principal ideal domain if and only if    $c>0$.
\end{theorem}
\begin{proof}
By replacing $X$ by $\sqrt{|c|/b}\,X$ and $Y$ by $\sqrt{|c|}\,Y$,  it follows that
$\,{\rm L}_{b\,,\,c} \cong
\begin{cases}
{\rm L}_{1\,,\, 1} & \text{if} \   \ c>0,\cr
{\rm L}_{1\,,\, -1} & \text{if} \   \ c<0,\cr
\end{cases}\,$ as $\R$-algebras and hence
 we may assume that $b=1$ and $c=\pm 1$.   Since ${\rm L}_{1\,,\,-1}$ is not a principal ideal domain by \ref{mypar:2.14}\,(5), it is enough to prove that $A:={\rm L}_{1\,,\,1}$ is a principal ideal domain.  Note that $B:= \C\otimes_{\R} A=\C[X,Y]/\langle X^{2}+Y^{2}+1\rangle \iso \C[U,V]/\langle UV-1\rangle \cong \C[T, T^{-1}]$ is a principal ideal domain and that $B$ is a free $A$-algebra with basis $1$, ${\rm i}$, where ${\rm i}\in\C$ with ${\rm i}^{2}+1=0$.  Let $x$ $,y\in B$ denote the images of $X$, $Y$ in $B$ respectively and let $\sigma:B\to B$, ${\rm i}\mapsto -{\rm i}$, denote the conjugation automorphism of $B$ over $A$. Then  $\sigma^2\!\!=\!\id_B$ and
$(x+{\rm i}y)\cdot\sigma(x+{\rm i}y)\!=\!(x+{\rm i}y)(x-{\rm i}y)\!=\!-1$,  \hbox{in~particular, $\sigma(x+{\rm i}y)\!=$} $-(x+{\rm i}y)^{-1}\!$. Further, an element $f \in B$ belongs to $A$ if and only if $\sigma(f)\!=\!f$. Moreover, $B^\times\!\!= \!\{ \lambda(x+iy)^n \mid \lambda \in \C^\times  \  \text{and} \ n \in \Z \}$.
\smallskip

Let $\gothA$ be any ideal in  $A$. To show that $\gothA$ is principal, we may assume that $\gothA\neq 0$ and $\gothA\neq A$.  Since $B$ is a PID,  the ideal $\gothA B\,(\neq 0)$ generated by $\gothA$ in $B$ is principal. We claim that there exists $f\in A$ such that $\gothA B=B f$. First choose  $g \in B$, $g\neq 0$  such that $\gothA B= Bg$.  Since $B\sigma(g) = \sigma(Bg) = \sigma(\gothA B) = \sigma(\gothA) B = \gothA B = Bg$  and since  $B$ is an integral domain, there exists a unit $u \in B^\times $ such that $\sigma(g)=u\cdot g$. Further, since $\sigma^{2}=\id_{B}$ and $g\neq 0$, we have $u\cdot \sigma(u)=1$. Therefore
$u = \lambda (x+{\rm i}y)^n$ for some $(\lambda, n) \in \C^\times \times \Z$ and
\[ 1 = u\cdot \sigma(u)  =  \lambda (x+{\rm i}y)^{n}\cdot  \sigma (\lambda) (-1)^{n}  (x+{\rm i}y)^{-n }= (-1)^{n|}\,|\lambda |^{2}. \] 
This proves that $n$  is even  and $|\lambda|^{2}=1$,  i.\,e.  $n=2m$ and $\lambda=e^{{\rm i}t}$ with $m\in \Z$ and $t \in \R$.

Now, put $f:={\rm i}^{m} e^{{\rm i}t/2} (x+{\rm i}y)^{m}\cdot g$. Then  $\gothA B=Bg=Bf$. To show that  $f\in A$, it is enough to prove that $\sigma(f)=f$. We have
\begin{align*}
\sigma(f) = (-{\rm i})^{m} e^{-{\rm i}t/2} (x-{\rm i}y)^{m} \cdot  \sigma(g)
& =  (-{\rm i})^{m} e^{-{\rm i}t/2} \cdot  (x-{\rm i}y)^{m} \cdot u \cdot g \\
&	 = (-{\rm i})^{m} e^{-{\rm i}t/2} (x-{\rm i}y)^{m} \cdot e^{{\rm i}t} (x+{\rm i}y)^{2m}\cdot   g \\
	&  = (-{\rm i})^m e^{it/2} (x-{\rm i}y)^{m} (x+{\rm i}y)^{m}  (x+{\rm i}y)^{m}\cdot g \\
& = (-{\rm i})^m (-1)^{m} e^{it/2}   (x+{\rm i}y)^{m}\cdot g
={\rm i}^{m} e^{it/2}   (x+{\rm i}y)^{m}\cdot g =
f\,.
\end{align*}
Therefore, since $B$ is a free $A$-module with basis $1, {\rm i}$, it follows that   $\gothA = \gothA B \cap A = Bf \cap A= Af$   is a principal ideal.
\end{proof}

Finally, we come to a class of affine domains over $\R$ which are principal ideal domains, but not Euclidean domains\,:

\begin{theorem}\label{thm:2.19}\,
Let $b$, $c\in\R$ with $b>0$ and $c>0$ and let $\Phi:=Y^2+bX^2+c\in
\R[X,Y]$. Then the affine $\R$-domain ${\rm L}_{b\,,\,c}:=\R[X,Y]/\langle \Phi\rangle$ is a principal ideal domain and is not an Euclidean domain.
\end{theorem}

\begin{proof}
By Theorem~\ref{thm:2.18} ${\rm L}_{b\,,\,c}$ is a principal ideal domain and by Corollary\,\ref{cor:2.17} ${\rm L}_{b\,,\,c}$ is not a Euclidean domain.
\end{proof}

\section{Real Nullstellensatz}

Real Algebra\,---\,the study of ``real objects''  such as real rings (resp.  real varieties)  in the  category of rings (resp. real algebraic varieties)\,---\,has attracted considerable interest because of its use in the development of algebraic geometry over the field $\R$ of real numbers,  more generally, over a real closed field.  Real algebra plays a role analogous to the one played by commutative algebra in the development of classical (and abstract) algebraic geometry. Therefore, real algebra has many applications to geometric problems.
\smallskip

In the category of fields, the real objects, namely, the formally real fields have been studied by \'Emil Artin and Otto Schreier. They recognized that formally real fields are precisely fields which can be ordered. The idea of exploiting the orders in a real field played a central role in Artin's solution to Hilbert's 17th problem.
\smallskip

The Real Nullstellensatz has a weak version and a strong version which are similar to the corresponding versions of the classical HNS for algebraically closed fields. In this section we present the proofs of both the versions assuming the Artin-Lang homomorphism theorem.

{\small
\begin{mypar}\label{mypar:3.1} {\bf Notation and Preliminaries}\,
 In the category of commutative rings, two notions ``semi-real'' and ``real'' of  ``reality'' play an important role. We recall these concepts and basic results concerning them. For details the reader is recommended to see  N. Jacobson~\cite{jacobson2} or an article by T.\,Y.\,Lam~\cite{lam}

\begin{arlist}

\item {\bf Reality}\,
Let $A$ be a (commutative) ring. The set $\,\{ a_1^2+\cdots +a_n^2 \mid n \ge 1,a_1,\dots,a_n \in A \}\,$ of {\it sums of squares} in $A$ is denoted by $\sum A^2$.
It is a {\it semiring}\,\footnote{\label{foot:9}A {\it semiring} is an algebraic structure $(R,+,\cdot)$  similar to a ring, except that $(R,+)$ is a commutative monoid and is not necessarily an abelian group. A motivating example of a semiring is the set of natural numbers $\N$ with usual addition and multiplication. Similarly, the sets $\Q_{\geq 0}$ and $\R_{\geq 0}$ of the non-negative rational numbers and the non-negative real numbers, respectively, form semirings.} contained in $A$.\,
A ring $A$ is called {\it semi-real} if $-1 \notin \sum A^2$.  If $A$ is not semi-real, then there exists $a_1,\dots,a_n \in A$ with $1\!+\!a_1^2\!+\!\cdots\!+\!a_n^2\!=\!0\,;$ in this case, we say that $A$ is {\it unreal.}\,   A ring $A$ is called {\it (formally) real} if for all $a_1,\dots,a_n \in A$,  $a_1^2\!+\!\cdots\!+\!a_n^2\!=\!0$ implies $a_1\!=\!\cdots\!=\!a_n\!=\!0$.\,
\smallskip

We can also define  these notions of reality for ideals in $A$.\,
An ideal $\gotha \subseteq A$ is called {\it semi\hbox{-}real} (resp. {\it real}) if the residue class ring $A/\gotha$ is semi-real (resp. real).
\smallskip

The zero ring is real but not semi-real. A nonzero real ring is a semi-real.
If a non-unit ideal $\gotha$ in a ring $A$ is real, then it is also semi-real.
The characteristic of  a real field is  $0$.
The two notions of reality for fields and maximal ideals  coincide. A field is semi-real if and only if it is real. Similarly, for every maximal ideal $\gothm$, $A/\gothm$ is semi-real if and only if it is real.
\smallskip

The ring $\R[X,Y]/(X^2+Y^2)$ is semi-real but not real. More generally,  $\R[X_1,\dots,X_n]/(X_1^2+\cdots+X_n^2)$ is semi-real ring but not real.\,
The ideal $\gotha =\langle X^2-2\rangle$ in the ring $\Q[X]$ is a real prime ideal, since  $\Q[X]/\gotha \simeq \Q[\sqrt{2}]$ is a real field.\,
If $A$ is real, then $A$ is reduced and every subring is also real.\,
An integral domain $A$ is real if and only if its quotient field ${\rm Q}(A)$ is real.\,
\smallskip

For a local ring,  it is convenient to introduce the third notion of reality\,:  a local ring $(A,\gothm_{A})$ is called {\it residually real} if the maximal ideal $\gothm_A{}$ is real, i.\,e. if the residue field $A/\gothm_{A}$ is a formally real field. Note that if $(A,\gothm_{A})$ is semi-real or even real, then it does not follow that $(A,\gothm_{A})$ is residually real. For example, the local ring $\Z_{(p)}$, where $p$ is a prime number,  is not residually real. On the other hand, if a local ring $(A,\gothm_A)$ is residually real, then it is semi-real, but not necessarily real. For example, the localization of the ring $\R[X_1,\dots,X_n]/\langle X_1^2+\cdots+X_n^2\rangle =\R[x_{1},\ldots , x_{n}]$ at the maximal ideal $\gothm:=\langle x_{1},\ldots , x_{n}\rangle$ is  residually real and  semi-real, but not real.
\medskip

For convenience we note the following observations for future reference (for proofs see\,\cite{lam})\,:
\smallskip

{\bf (i)}\, If $\varphi:A \rightarrow B$ is a ring homomorphism and if $B$ is semi-real, then $A$ is also semi-real. Moreover, if $\varphi$ is injective and if $B$ is real, then $A$ is also real. In~particular, an integral domain $A$ is real if and only if its quotient field ${\rm Q}(A)$ is real.\,
\vskip2pt

{\bf (ii)}\, If $\varphi:A \rightarrow B$ is a ring homomorphism and if $\gothb\subseteq B$ is semi-real (resp. real) ideal, then $\varphi^{-1}(\gothb)$ is also semi-real (resp. real). If $\varphi$ is surjective, then an ideal $\gothb\subseteq B$ is semi-real (resp. real) if and only if $\varphi^{-1}(\gothb)$ is semi-real (resp. real).\,
 \vskip2pt

{\bf (iii)}\, A direct product $A=A_{1}\times \cdots \times A_{r}$ (with $A_{i}\neq 0$) is semi-real (resp. real) if and only if all factors $A_{1}, \ldots , A_{r}$ are semi-real (resp. real).\,
 \vskip2pt

{\bf (iv)}\, A valuation ring $R$ is real if and only if it is semi-real.\,
 \vskip2pt

{\bf (v)}\, Let $S\subseteq A$ be a multiplicatively closed subset in $A$ with $0\!\not\in\! S$ (so that  $S^{-1}\!A \!\neq\!0$). Then the implications\,:
\[ A \text{ real } \Rightarrow S^{-1}A \text{ real } \Rightarrow S^{-1}A \text{ semi-real } \Rightarrow A \text{ semi-real} \]
hold, but, in general, the reverse implications do not hold.
\vskip3pt

{\bf (vi)}\, If $A$ is a  regular local ring which is residually real, then $A$ is real.
\smallskip

\item {\bf Artin-Schreier theory for fields}\,
In 1927 Artin-Schreier discovered that for fields there is a connection between the notion of formal reality and the existence of orders\,\footnote{\label{foot:10}Recall that a field $K$ together with an order $\leq$ on  $K$ is called an {\it ordered field} if\,  (i) $\leq$ is a total order on $K$.\, (ii) The monotonicity of addition and multiplication holds, i.\,e. for all $a$, $b$, $c\in K$, the implications\,:
$a\leq b$ $\Rightarrow$ $a+c\leq b+c$ and
$a\leq b$ and $0\leq c$  $\Rightarrow$ $ac\leq bc$.\, Sometimes we also write that $K$ has a field order $\leq$ if $(K,\leq)$ is an ordered field. On a field $K$ there may be  many field orders.}.\,
\smallskip

\begin{alist}

\item {\bf Theorem}\,\so{(Artin-Schreier):}\,
{\it A field $K$ is real if and only if there is an order $\leq$ on $K$ such that $(K,\leq)$ is an ordered field.}
\smallskip

For a proof, one considers the set of {\it preorders}\,\footnote{\label{foot:11}Let $K$ be a field. A subset $T\subseteq K$ is called a {\it preorder} on $K$ if $T+T\subseteq T$, $T\cdot T\subseteq T$, $\{a^{2}\mid a\in K\} \subseteq T$ and $-1\not\in T$. Note that  if $T$ is a preorder on  $K$, then $T\cap -T =\{0\}$. A preorder on $K$ is an order if and only if $T\cup -T=K$.}  on $K$ and proves (by using Zorn's Lemma) that it has a maximal element (with respect to the natural inclusion). Finally, note that maximal preorders on a field $K$ are orders on $K$.
\smallskip

More generally, we have\,:
\smallskip

\item {\bf Theorem}\,\so{(Artin-Schreier criterion for sums of squares)}\,
{\it Let $K$ be a field of characteristic $\neq 2$ and $a\in K$. Then the following statements are equivalent\,$:$\,
{\rm (i)} $\,a\in \sum K^{2}$.\,
{\rm (ii)} $a$ is a totally positive element, {\rm i.\,e.} $a$ is positive for any field order on $K$.\,
In~particular, if $K$ has no field order, then every element of $K$ is sum of squares.}
\end{alist}

\item {\bf Real closed fields}\,
Let $K$ be a field. Then we say that
 $K$ is  {\it real closed} if it is real and if it has  no nontrivial real algebraic extension $L\,\vert\, K,$ $L \neq K$.  For example, the field $\R$ of real numbers is real closed. The algebraic closure of $\Q$ in $\R$ is real closed. The field $\Q$  is  real, but not real closed.
 \smallskip

 We list some basic results on  real and real closed fields without proofs\,:

 \begin{alist}
 \item
 Let $K$ be a real closed field. Then\,:\,
 (i)\, Every polynomial $f\in K[X]$ of odd degree has a zero in $K$.\,
 (ii)\, $K$ has exactly one field order. (this order is called the {\it unique order of the real closed field $K$}).\,
(iii) The set $K^{2}:=\{a^{2}\mid a\in K\}$ of squares in $K$ is a field order on $K$.

\item
Let $K$ be a real-closed field and $a\in K$ be a positive element in $K$. Then $a$ has a unique positive  square root in $K$ which is denoted by $\sqrt{a}$.

\item
Let $K$ be a real field. Then there exists an algebraic extension $L\,\vert\,K$ such that $L$ is a real closed field. Such a field $L$ is called a {\it real closure} of $K$.\\[1.5mm]
{\bf Remark\,:} If $L\,\vert \,K$ and $L'\,\vert\,K$ be two real closures of a real field  $K$, then it is not necessary that  $L\,\vert \,K$ and $L'\,\vert\,K$  are isomorphic. For an example, it is enough to note that there are fields with at least two field orders. The subfield $\Q(\sqrt{2})$ of $\R$ has exactly two field orders.\,---\,This is a special case of a much deeper result\,: An algebraic number field $L$ has at most $[L:\Q]$ distinct orders and the signature of its trace form is $\geq 0$.  If $L\,\vert\,K$ is a finite field extension  of an ordered field $(K,\leq)$ and if $\leq$ can be extended to a unique order on $L$, then $[L:K]$ is odd. In~particular, every field order on $K$ can be extended to a field order on $L$.

\item
Let $(K,\leq )$ be an ordered field. Then\,:\,
\vskip3pt

{\bf (i)}\, \so{(Existence of real closure)}\,: There exists a real-closed field extension $L\,\vert\,K$ such that the unique order on $L$ extends the given order $\leq$ on $K$. (A real closed field $L$ such that the unique order of $L$ extends the given order $\leq$ on $K$ is called a {\it real closure} of  the ordered field $(K,\leq)$.
\vskip3pt

{\bf (ii)}\,  \so{(Uniqueness of real closure)}\,:  If $L\,\vert\,K$  and $L'\,\vert\,K$ are two real closures of $(K,\leq)$, then the field extensions $L\,\vert\,K$ and $L'\,\vert\,K$ are isomorphic. Indeed, there is a unique $K$-isomorphism which preserves order.

\item {\bf Theorem}\,\so{(Euler-Lagrange)}\,
{\it Let $(K,\leq)$ be an ordered field satisfying the properties\,:\,  {\rm (i)}  Every polynomial $f\in K[X]$ of odd degree has a zero in $K$.\, {\rm (ii)\,}  Every positive element in $K$ is a square in $K$. Then the field $\overline{K}=K({\rm i})$ obtained from $K$ by adjoining  a square root ${\rm i}$ of $-1$ is algebraically closed. In~particular, $K$ itself is real-closed.}
\vskip3pt

{\bf Remark\,:} Since the field $\R$ of real numbers is ordered and satisfies the properties (i) and (ii), the above theorem proves the \so{Fundamental Theorem of Algebra\,:} {\it The field $\C=\R({\rm i})$ of complex numbers is algebraically closed .}

\item The Theorem in (e) has a remarkable complement (see also Theorem\,4.16\,)\,:
\smallskip

{\bf Theorem}\,\so{(Artin)}\,
Let $L$ be an algebraically closed field. If $K\subseteq L$ be a subfield of $L$ such that $\,L\,\vert\,K$ is finite and $K\neq L$, then $L=K({\rm i})$ with ${\rm i}^{2}+1=0$ and $K$ is  a real-closed field.
\end{alist}

\item
{\bf Artin-Schreier theory for commutative rings}\,
To formulate a  generalization for Artin-Schreier Theorem to commutative rings, it is crucial to arrive at the right definition of an ``order'' on  a commutative ring. Let $A$ be a commutative ring. A  preorder on $A$ is  defined in the same way as done for fields, see Footnote\,\ref{foot:11}.  If $T$ is preorder on $A$, then $T\cap -T$ need not be $\{0\}$. However, it is easy to see that $\gotha:=T\cap -T$ is the largest additive subgroup contained in $T\,$; $\gotha$ is called the {\it support} of $T$ denoted by $\Supp (T)$.
If $2\in A^\times$, then the support of a preorder on $A$ is an ideal in $A$. For this, we need to show that if $a\in A$ and $x\in\gotha$, then $ax\in \gotha$. It is enough to write $a=b^{2}-c^{2}$ for some $b$, $c\in A$, which is always possible, since $1/2\in A$, for take $b=(1+a)/2$ and $c=(1-a)/2$.  If a preorder $T$ on $A$ satisfies $T\cup -T=A$, then it is easy to see that the support $\gotha$ of $T$ is an ideal in $A$ even if $2\not\in A^{\times}$. With all this preamble, we are now ready to define an order on a commutative ring $A$\,: A preorder $T$ on $A$ is called an {\it order} on $A$ if $T\cup -T=A$ and the support $\gotha=T\cap -T$ is a prime ideal in $A$. Note that the prescription of an order on $A$ with support $\gothp\in\Spec A$ is equivalent to the prescription of an order on the quotient ring $A/\gothp$. Therefore possible supports of orders on $A$ are precisely all the real prime ideals.
\smallskip

With the preparation as above, we are ready to state the following result of A.\,Prestel \cite{prestel}.

\begin{alist}
\item
{\bf Theorem}\,{\rm \so{(Prestel)}\,:} {\it Let $A$ be a ring and $T$ be a maximal $($with respect to the natural inclusion$)$ preorder on $A$, then $T$ is an order on $A$.}
\smallskip

As a consequence we have\,: (1)\, {\it Every preorder on a ring $A$ is contained in an order on $A$.} (2)\, {\it Let $T$ be an order on $A$. Then $T$ is a maximal as an order on $A$ if and only if $T$ is a maximal as a preorder on $A$.}
\smallskip

We now state a generalization of the Artin-Schreier Theorem for commutative rings\,:
\smallskip

\item {\bf Theorem} {\it Let $A$ be a commutative ring. Then $A$ is semi-real if and only if there exists an order on $A$.}
\end{alist}

\end{arlist}

\end{mypar}

}

M.\,Coste and M.\,-F.\,Coste-Royer, have introduced the notion of the real spectrum
(the set $X_{A}$ of orders on $A$ with Harrison topology) of a ring $A$ in \cite{coste}. On the one hand this is the correct generalization of the space of orders of a field, and on the other hand, this is the  ``real'' analogue of the prime spectrum (with Zariski topology) of  a ring. We shall restrict ourselves to the ideal theoretic view rather than orders and use only the basic properties of the real spectrum and show that the study of real spectrum is an indispensable tool in real algebraic geometry. We begin with\,:

\begin{definition}\label{def:3.2} Let $A$ be a ring. The set
\begin{align*}
\rSpec(A) \colon= \{ \p \in \Spec(A) \mid k(\p)={\rm Q}(A/\p) \text{ is real} \}.
\end{align*}
of all real prime ideals in $A$ is called the {\it real prime spectrum} of $A$.
\end{definition}

With this definition we can give a characterization of real rings\,:

\begin{theorem}\label{thm:3.3}
For a ring $A$, the following are equivalent$:$

\begin{rlist}
\item \  $\,\,A$ is real.
\item \  $A$ is reduced and all minimal prime ideals of $A$ are real.
\item  $A$ is reduced and $\rSpec(A)$ is dense in $\Spec(A).$
\item  The intersection of all $\gothp \in \rSpec(A)$ is $0$, {\rm i.\,e.}  $\bigcap_{\,\gothp \in \rSpec(A)}\, \gothp=0.$
\end{rlist}
\end{theorem}

\begin{proof}
(i) $\Rightarrow $ (ii)\,:
If $A$ is real, then $A$ is reduced by \ref{mypar:3.1}\,(1) and for every prime ideal $\gothp$, $A_{\gothp}$ is real by \ref{mypar:3.1}\,(1) \,(v). Therefore $\gothp\,A_{\gothp}=0$ for every minimal prime ideal $\gothp$ in $A$ and hence $k(\gothp) = A_{\gothp}$ is real, i.e. $\gothp\in \rSpec A$.
\smallskip

(ii) $\Rightarrow $ (iii)\,: 	
This is clear as the set of minimal prime ideals is dense in $\Spec(A)$.
\smallskip

(iii) $\Rightarrow $ (iv)\,:
Since $\rSpec(A)$ is dense in $\Spec(A)$, $\bigcap_{,\gothp \in \rSpec(A)} \gothp = \bigcap_{\,\gothp \in \Spec(A)} \gothp ={\rm nil}\,(A)=0$.
\smallskip

(iv) $\Rightarrow $ (i)\,: 	
Let $a_1,\dots,a_r \in A$ be such that $a_i \not= 0$ for all $i=1,\dots,r$.  Then by (iv), there exists $\gothp\in \rSpec(A)$ with $a_1 \notin \gothp$.  Since $\gothp$ is real, $a_1^2+\cdots+a_r^2 \notin \gothp$. In particular, $a_1^2+\cdots+a_r^2 \neq 0$. This proves that $A$ is real.
\end{proof}

\begin{corollary}\label{cor:3.5}
Real rings are subrings of a direct product of formally real fields.
\end{corollary}

\begin{remark}\label{rem:3.4}
One may also ask\,: If $A$ is semi-real, then is $A$ reduced? And are all of its minimal prime ideals  semi-real? We give examples to show that both these questions have negative answers.\, (i) The ring $\R[X]/\langle X^{2}\rangle$ is semi-real, but clearly not reduced. \, (ii) Let $K$ be a real field and $K'$ be a non-semi-real field. Then the product ring $A:=K\times K'$ is semi-real, since the first projection $A\to K$ is a ring homomorphism. Further, $A$ is reduced with two minimal prime ideals $\gothp_{1}=K\times \{0\}$ and
 $\gothp_{2}=\{0\}\times K'$. Since $A/\gothp_{1}\cong K'$ and $A/\gothp_{2}\cong K$, $\gothp_{2}$ is real but $\gothp_{1}$ is not semi-real.
\end{remark}

\smallskip

One can also characterize semi-real rings. For this the following definition is  useful\,:

\begin{definition}\label{def:3.6}
Let $A$ be a   ring.
An ideal $\gotha\subseteq A$ is called {\it maximal real} if it is real, $\gotha\neq A$ and it is maximal with respect to this property. In other words, $\gotha$ is a maximal element in the set of non-unit real ideals in $A$ ordered by the natural inclusion, i.\,e. if $\gotha \subseteq \gotha' \subseteq A$ with $\gotha'$ real, then either $\gotha=\gotha'$ or $\gotha'=A$. Maximal semi-real ideals are defined analogously.
\end{definition}

\begin{theorem}\label{thm:3.7}
Let $A$ be a   ring and $\gotha\subseteq A$ an ideal.

\begin{alist}

\item
Let $S\subseteq  A$ be a multiplicatively closed subset of $A$ with $1\in S$, $0\not\in S$ and $S+\sum A^{2} \subseteq S$ and let $\gothp$ be an ideal in $A$ maximal with repsect to the property $S\cap \gothp=\emptyset$. Then $\gothp$ is a prime ideal in $A$ and the quotient field ${\rm Q}(A/\gothp)$ is a real field. In~particular, $\gothp$ is a real ideal.

\item
Let $\gotha$ be an ideal in $A$. Then
$\gotha$ is maximal semi-real if and only if $\gotha$ is maximal real.
The maximal real (resp. semi-real) ideals are precisely the (prime) ideals $\gothp$ in $A$ which are maximal with respect to the property that $\gothp \cap (1+\sum A^2)=\emptyset$.
\end{alist}

\end{theorem}

\begin{proof}
{\bf (a)}\,  By well-known arguments from commutative algebra one can show that $\gothp$ is a prime ideal in $A$. To prove that the quotient field ${\rm Q}(A/\gothp)$ is real, suppose that
$b_{1}^{2}+\cdots + b_{r}^{2}\in \gothp$ with $b_{i}\in A$, $i=1,\ldots , r$.  We need to show that $b_{i}\in\gothp$ for all $i=1,\ldots , r$. On the contrary, if (by renumbering) some $b_{1}\not\in\gothp$, then $S\cap (\gothp+\langle b_{1}\rangle)\neq\emptyset$ by the maximality of $\gothp$. Therefore $s\equiv ab_{1}\,({\rm mod\,}\gothp)$ for some $s\in S$ and $a\in A$. But, then
$s^{2}\equiv a^{2}b^{2}_{1}\,({\rm mod\,}\gothp)$ and hence
$s^{2}+a^{2}b^{2}_{2}+\cdots +a^{2}b^{2}_{r} = a^{2}b^{2}_{1}+a^{2}b^{2}_{2}+\cdots +a^{2}b^{2}_{r}\in (S+\sum A^{2}) \cap \gothp \subseteq S\cap \gothp$, which is a contradiction. The supplement is immediate from \ref{mypar:3.1}\,(1)\,(i).
\vskip3pt

{\bf (b)}\, We may assume that $A\neq 0$. Note that
the semi-real ideals of $A$ are the ideals $\gotha$ of $A$ with $\gotha \cap S =\emptyset$, where $S:=1+\sum A^{2}$.
($\Rightarrow$)\,: Suppose that $\gotha$ is a  maximal semi-real ideal. To prove that $\gotha$ is maximal real, we may assume that $0\not\in S$ (otherwise, there is nothing to prove), i.\,e. $A$ is semi-real ring. Therefore by (a) (applied to the multiplicative set $S$), $\gotha$ is a real ideal in $A$. Moreover, $\gotha$ must be maximal real, since real ideals are also semi-real.
($\Leftarrow$)\,: If $\gotha$ is a maximal real ideal, then (since non-unit real ideals are semi-real) by already proved implication ($\Rightarrow$) $\gotha$ is also maximal semi-real.
With this the last assertion is immediate from (a).
\end{proof}

\begin{corollary}\label{cor:3.8}
Let $A$ be a ring and $\gotha\subseteq A$ an ideal.
An ideal $\gotha\subseteq A$ is semi-real if and only if there exists $\gothp \in \rSpec(A)$ with $\gotha \subseteq \gothp$. In particular, a ring $A$ is semi-real if and only if $\rSpec(A) \neq \emptyset$.
\end{corollary}

\begin{proof}
	If $\gotha$ is semi-real, then $\gotha$ is contained in a maximal semi-real ideal $\gothp$ and $\gothp \in \rSpec(A)$ by the above theorem. The converse is clear.
\end{proof}

\begin{corollary}\label{cor:3.9}
Let $A$ be a ring.
If $-1$ is a sum of squares in every residue field of $A$, then $-1$ is a sum of squares in $A$.
\end{corollary}
\smallskip

To formulate a general version of Real Nullstellensatz, we need the concept of the real radical of an ideal. We shall define this for ideals in commutative rings\,:
\smallskip

Let $A$ be a commutative ring. For an ideal $\gotha\subseteq A$, let ${\rm V}(\gotha):=\{\gothp\in \Spec A\mid \gotha\subseteq \gothp\} \subseteq \Spec A$ and
$\rV{\gotha}=\{\gothp\in \rSpec A\mid \gotha\subseteq \gothp\} \subseteq \rSpec A$, see Definition\,\ref{def:3.2}. The {\it real radical} $\rsqrt{\gotha}$ of $\gotha$ is the intersection of all real prime ideals containing $\gotha$, i.\,e.  $\rsqrt{\gotha}: = \bigcap_{\gothp\in {\footnotesize \rV{\gotha}}}\,\gothp\,$.

\smallskip

The following theorem is an element-wise characterization of the real radical of an ideal in a commutative ring\,:

\begin{theorem}\label{thm:3.10}
Let $\gotha\subseteq A$ be an ideal in a commutative ring $A$ and let $f\in A$.  Then the following statements are equivalent\,$:$\, {\rm (i)}
$f\in\rsqrt{\gotha}\,$.\,\,  {\rm (ii)} There exists $m\in\N$ and $a_{1},\ldots , a_{r}\in A$ such that $f^{\,2m}+a_{1}^{2}+\cdots +a_{r}^{2}\in\gotha$.
\end{theorem}

\begin{proof}
By passing to the residue class ring $A/\gotha$, we may assume that $\gotha=0$.
\smallskip

(i) $\Rightarrow$ (ii)\,:  Note that (i) is equivalent with ${\rm D}(f) \cap \rSpec A=\emptyset$. Therefore $\rSpec S^{-1} A=\emptyset$, where $S=\{f^{\,n}\mid n\in \N\}$ and hence by Corollary~\ref{cor:3.8}\, there is an equation $1+(a_{1}/f^{\,n})^{2}+\cdots +(a_{r}/f^{\,n})^{2} =0$ for some $r\in\N$ and $a_{1},\ldots , a_{r}\in A$. It follows that $f^{\,2m}(f^{\,2n} +a_{1}^{2}+\cdots +a_{n}^{2})=0$ for some $m\in\N$. This proves (ii).
\smallskip

(ii) $\Rightarrow$ (i)\,:  For every $\gothp\in\rV{\gotha}$, from (ii) it follows that $f^{m}\in\gothp$ and hence $f\in\gothp$.
\end{proof}

\smallskip

Note that a ring $A$ is real if and only if $\rsqrt{0}=0$ (see Theorem~\ref{thm:3.3}), i.\,e. ``real reduced''. More generally, an ideal $\gotha\subseteq A$ is a real ideal if and only if $\rsqrt{\gotha}=\gotha$.

\smallskip

Affine algebras over a field are important in algebraic geometry because they are coordinate rings of algebraic sets. The Artin-Lang theory of affine algebras (over a field $K$) and their associated function fields provide applications to real algebraic geometry. To simplify matters, we shall always assume that the base field $K$ is a real closed field. The case when $K$ is an ordered field can be treated by passing to the real closure of $K$.

\begin{theorem}\label{thm:3.11}
{\rm \so{(Artin-Lang Homomorphism Theorem)}}\,  Let $K$ be a real closed field and let $A$ be a real  affine domain over $K$.  Then there exists a $K$-algebra homomorphism $\varphi:A\longrightarrow K$.
\end{theorem}

For a proof we refer the reader to the article by S.\,Lang \cite{langRealPlaces}
, see also \cite{lam}. In fact, Lang proved a stronger result than the above Theorem\,: {\it Let $V$ be an affine irreducible $K$-variety over a real closed field $K$ and $K[V]=K[X_{1},\ldots , X_{n}]/\gothp$ be the coordinate ring of $V$ over $K$. Then the function field $K(V)$ of $\,V$ over $K$ is formally real if and only if $V$ has a non-singular $K$-rational point.}
\smallskip

For further use in our exposition we note the following three improved supplements of the Artin-Lang Homomorphism Theorem.

\begin{corollary} \label{cor:3.12}
Let $K$ be a real closed field,  $A$ be a real affine domain over $K$ and let $f_{1},\ldots , f_{n}\in A$ be non-zero elements.  Then there exists a $K$-algebra homomorphism $\varphi:A\longrightarrow K$ such that $\varphi(f_{i})\neq 0$ for all $i=1,\ldots , n$.
\end{corollary}

\begin{proof}
Apply Theorem~\ref{thm:3.11} to the real affine domain $A[1/(f_{1}\cdots f_{n})]$.
\end{proof}

\begin{corollary} \label{cor:3.13}
Let $K$ be a real closed field and let $A$ be a semi-real affine algebra over $K$.   Then there exists a $K$-algebra homomorphism $\varphi:A\longrightarrow K$.
\end{corollary}

\begin{proof}
Note that by Theorem~\ref{thm:3.7} there exists $\gothp\in\rSpec A$.  Now, apply Theorem~\ref{thm:3.11} to the real affine domain $A/\gothp$.\end{proof}

\begin{corollary} \label{cor:3.14}
Let $K$ be a real closed field,  let $A$ be an affine algebra over $K$ and let $f_{1},\ldots , f_{m}\in A$, $g_{1},\ldots , g_{n}\in A$.  If there exists an order $T$ on $A$ such that $f_{i}>_{T}\,0$ and  $g_{j}\geq_{\,T}\,0$ for all $1\leq i\leq m$ and $1\leq j\leq n$, then there exists a $K$-algebra homomorphism $\varphi:A\longrightarrow K$ such that $\varphi(f_{i})>0$ for all $i=1,\ldots , m$ and $\varphi(g_{j})\geq 0$ for all $j=1,\ldots , n$, where $\leq$ is the unique order on $K$, {\rm see\,\ref{mypar:3.1}\,(3)\,(a)}.
\end{corollary}

\begin{proof}
Let $\gothp$ be the support of $T$, see \ref{mypar:3.1}\,(4). By passing to $A/\gothp$, we may assume that $\gothp=0$. Then $T$ extends uniquely to an order $P$ on the quotient field ${\rm Q}(A)$ of $A$ with $f_{i}\in P\!\smallsetminus\!\{0\}$ for all $i=1,\ldots , m$ and $g_{j}\in P$ for all $j=1,\ldots , n$.  Now, apply Corollary~\ref{cor:3.12} to the real affine domain $A[1/f_{1},\ldots , 1/f_{m},\sqrt{f_{1}},\ldots , \sqrt{f_{m}},\sqrt{g_{1}},\ldots , \sqrt{g_{n}}]$ in the real closure of the ordered field $({\rm Q}(A), P)$, see \ref{mypar:3.1}\,(3)\,(d).
\end{proof}

\smallskip

The Artin-Lang theory lays the foundations for real algebraic geometry, i.\,e.  the study of real algebraic varieties. The Artin-Lang homomorphism theorem is used to give affirmative answer to the Hilbert's 17th Problem for a real closed base field, see\,\ref{mypar:3.15} below.
\medskip

We consider this problem over the real closed base fields only. Let $K$ be a real closed field. Let $f\in K[X_{1},\ldots , X_{n}]$ be a polynomial in $n$ indeterminates $X_{1},\ldots , X_{n}$ over $K$. We say that $f$ is positive semi-definite if $f(a)\geq 0$ for all $a\in K^{n}$, where $\leq$ is the unique order on $K$, see\,\ref{mypar:3.1}\,(3)\,(a). If $f$ is positive semi-definite, then it need not be a sum of squares in $K[X_{1},\ldots , X_{n}]$. For instance, the Motzkin polynomial $M(X, Y)\!=\! X^{4}Y^{2}\!\!+\!X^{2}Y^{4}\!\!-\!3X^{2}Y^{2}\!+\!1$ gives such an example. Indeed, the arithmetic-geometric mean inequality implies that $M\!\geq\! 0 $ on $\R^{2}$. Suppose, on the contrary that $M\! =\!\sum_{j} f_{\!j}^{2}$ is  a sum of squares of real polynomials. Since $M(0,Y) \!=\!M(X,0)\!=\!1$, the polynomials $f_{\!j}(0,Y)$ and $f_{\!j}(X,0)$ are constants. Therefore each $f_{\!j}$ is of the form $f_{\!j}\!= \!a_{\!j}\!+\!b_{\!j}XY\!+\!c_{\!j}X^{2}Y\!+\!d_{\!j} XY^{2}$ with  $a_{\!j}, b_{\!j}, c_{\!j}, d_{\!j}\in K$.   Then the coefficient of $X^{2}Y^{2}$ in the equality $M\!=\!\sum_{j} f_{\!j}^{2}$ is equal to $-3\!=\!\sum_{j}b_{\!j}^{2}$ which is a contradiction.

\smallskip

Motivated by his previous work, David Hilbert posed  the following  problem which was the 17th in the list of 23 challenging problems presented  in his celebrated address to the International Congress of Mathematicians in Paris  (1900)\,:\,

\begin{mypar}\label{mypar:3.15}
{\bf Hilbert's 17th Problem}\,\footnote{\label{foot:12}The starting point of the history of Hilbert's 17th Problem was the oral defense
of the doctoral dissertation of Hermann Minkowski (1864-1909) at the University of
K\"onigsberg in 1885. The 21 year old Minkowski expressed his opinion that
there exist real polynomials which are nonnegative on the whole $\R^{n}$ and cannot
be written as finite sums of squares of real polynomials. David Hilbert was
an official opponent in this defense. In his `` Ged\"achtnisrede''
[D. Hilbert, Hermann Minkowski. Ged\"achtnisrede, Math. Ann. {\bf 68} \,(1910),
445-471.] in memorial
of H. Minkowski he said later that Minkowski had convinced him about the
truth of this statement. In 1888 [D. Hilbert, \"Uber die Darstellung definiter Formen als Summe von Formenquadraten,
Math. Ann. {\bf 32} \,(1888), 342-350.] Hilbert proved  the
existence of a real polynomial in two variables of degree six which is nonnegative
on $\R^2$ but not a sum of squares of real polynomials. Hilbert's proof used
some basic results from the theory of algebraic curves. Apart from this his
construction is completely elementary. The first explicit example of this kind
was given by T. Motzkin  only in 1967 [T.\,S.\,Motzkin, The arithmetic-geometric inequality. In\,: Proc. Symposium
on Inequalities, edited by O. Shisha, Academic Press, New York, 1967, pp.
205-224.].
}\,{\it If $f$ is positive semi-definite, then must it be a sum of squares in the rational function field $K(X_{1},\ldots , X_{n})\,$}?
\end{mypar}

The affirmative solution to this problem was given by Artin in 1927 \cite{EArtin}.

\begin{proof}
\so{(Modern version of Artin's proof):}\,  Put  $L:=K(X_{1},\ldots , X_{n})$. Suppose, on the contrary that $f\not\in \sum L^{2}$. Then by \ref{mypar:3.1}\,(2)\,(b) there exists an order $T$ on $L$ \hbox{such that $f<_{T}\,0$, i.\,e.} $-f \,>_{T}\,0$.  Therefore by applying Corollary~\ref{cor:3.14}, there exists a $K$-algebra homomorphism $\varphi:K[X_{1},\ldots , X_{n}]\longrightarrow K$ such that $\varphi(-f)>0$ (where $\leq$ is the unique order on the real closed field $K$, see\,\ref{mypar:3.1}\,(3)\,(a)).  Then, for $a_{i}:=\varphi(X_{i})$, $i=1,\ldots , n$, we have $f(a_{1},\ldots  , a_{n}) =f(\varphi(X_{1}), \ldots , \varphi(X_{n})) = \varphi(f(X_{1}, \ldots , X_{n}))=\varphi(f)\,<0$, a contradiction.
\end{proof}

Our main goal in the section is to formulate and prove Real Nullstellensatz. Let $K$ be a real closed field, $\overline{K}\!=\!K({\rm i})$ with ${\rm i}^{2}\!\!+\!1\!=\!0$ an algebraic closure of $K$ (see  Theorem of Euler-Lagrange in (3.1)\,(3)\,(e))  and $\gotha\subseteq K[X_{1},\ldots , X_{n}]$. We consider the $K$-algebraic set ${\rm V}_{\overline{K}}(\gotha)\subseteq \overline{K}^{n}$ and the set of {\it real points}  ${\rm V}_{K}(\gotha)={\rm V}_{\overline{K}}(\gotha)\cap K^{n}$.  The analogue of HNS\,1$'$ \,---\,the Real Nullstellensatz provides a geometric criterion for the semi-reality of  the affine $K$-algebra $K[X_{1},\ldots , X_{n}]/\gotha$ which is the  $K$-{\it coordinate ring} of ${\rm V}_{K}(\gotha)$\,:

\smallskip

\begin{theorem}\label{thm:3.16}
{\rm \so{(Classical Real Nullstellensatz)}}\,
 Let $K$ be a real closed field {\rm (e.\,g. $K=\R$)}  and let $A=K[x_{1},\ldots , x_{n}]=K[X_{1},\ldots , X_{n}]/\gotha$ be a $K$-algebra of finite type. Then
$A$ is semi-real if and only if \ ${\rm V}_{K}(\gotha)\neq \emptyset$.
\end{theorem}

\begin{proof}
Let $a=(a_1,\dots,a_n) \in {\rm V}_K(\gotha)$.  Then the evaluation map $\varepsilon_{a}:K[X_{1},\ldots , X_{n}]\longrightarrow K$, $f\mapsto f(a)$, induces a  $K$-algebra homomorphism $\overline{\varepsilon}_{a} : A \rightarrow K$  and hence by \ref{mypar:3.1}\,(1)\,(i)  $A$ is semi-real. Conversely, if $A$ is semi-real, then by  Corollary\,\ref{cor:3.13}  there exists a $K$-algebra homomorphism $\varphi : A \rightarrow K$. Then clearly  $(a_{1},\ldots, a_{n})=(\varphi(X_{1}), \ldots , \varphi(X_{n}))\in {\rm V}_{K}(\gotha)$, since $f(a_{1},\ldots  , a_{n}) =f(\varphi(X_{1}), \ldots , \varphi(X_{n})) = \varphi(f(X_{1}, \ldots , X_{n}))=\varphi(f)=0$ for every $f\in \gotha$.
\end{proof}

\begin{lemma}\label{lem:3.17}
Let $K$ be a real field and $\gotha \subseteq K[X_1,\dots,X_n]$ an ideal. Then the ideal ${\rm  I}({\rm V}_{K}(\gotha))$ is a real ideal.
\end{lemma}

\begin{proof}
Let $f_1,\dots,f_r \in A$ be such that $f_1^2+\cdots+f_r^2 \in {\rm I}_{K}({\rm V}_{K}(\gotha))$.  Then   $f_1^2(a)+\cdots+f_r^2(a) =(f_1^2+\cdots+f_r^2)(a)=0$ for all $a \in {\rm V}_K(\gotha)$. Therefore, since $K$ is a real field,   $f_{i}(a)=0$ for all $a \in {\rm V}_K(\gotha)$ and $i=1,\dots,r$,  i.\,e.   $f_{i }\in {\rm I}_{K}({\rm V}_{K}(\gotha))$ for all  $i=1,\ldots, r$.
\end{proof}

Next, we prove  the Real Nullstellensatz for prime ideals.

\begin{theorem}\label{thm:3.18}\,
Let $K$ be a real closed field {\rm (e.\,g. $K=\R$)} and let $\gothp\in$
$\Spec  K[X_1,\dots,X_n]$ be a prime ideal.  Then ${\rm I}_{K}({\rm V}_{K}(\gothp))=  \gothp$ if and only if $\,\gothp\in \rSpec  K[X_1,\dots,X_n]$.
\end{theorem}

\begin{proof}
($\Rightarrow$)\,: Suppose that $\gothp$ is not real. Then $f_{1}^{2}+\cdots +f_{r}^{2}\in\gothp$ for some $f_{1},\ldots , f_{r}\in K[X_1,\dots,X_n]$ with $f_{i}\not\in\gothp$ for every $i=1,\ldots , r$. But, then clearly $f_{i}\in {\rm I}_{K}({\rm V}_{K}(\gothp))$. In ~particular, $\gothp\subsetneq {\rm I}_{K}({\rm V}_{K}(\gothp))$.
\smallskip

($\Leftarrow$)\,: Suppose that $\gothp$ is real. To prove the equality  ${\rm I}_{K}({\rm V}_{K}(\gothp))=  \gothp$, it is enough to  prove that if $f\not\in \gothp$, then $f\not\in {\rm I}_{K}({\rm V}_{K}(\gothp))$. Since $\overline{f}\neq 0$ in the real affine domain  $A:=K[X_{1},\ldots , X_{n}]/\gothp=K[x_{1},\ldots , x_{n}]$ over $K$, by Corollary~\ref{cor:3.12} there exists a $K$-algebra homomorphism $\varphi:A\to K$ such that $\varphi(\overline{f}) \neq 0$. Then   $(a_{1},\ldots, a_{n})=(\varphi(x_{1}), \ldots , \varphi(x_{n}))\in {\rm V}_{K}(\gothp)$, since $g(a_{1},\ldots  , a_{n}) =g(\varphi(x_{1}), \ldots , \varphi(x_{n})) = \varphi(g(x_{1}, \ldots , x_{n}))=\varphi(\overline{g})=0$ for every $g\in \gothp$. Furthermore, $f(a_{1},\ldots  , a_{n}) =f(\varphi(x_{1}), \ldots , \varphi(x_{n})) = \varphi(f(x_{1}, \ldots , x_{n}))=\varphi(\overline{f})\neq 0$, i.\,e. $f\not\in {\rm I}_{K}({\rm V}_{K}(\gothp))$.
\end{proof}
\smallskip

Finally, we prove the analogue of HNS\,2  for  real closed fields which is also known as the Dubois-Risler Nullstellensatz, see\,\cite{duboisOrdered},
\cite{duboisEfroymson}
and \cite{risler}.

\begin{theorem}\label{thm:3.19}
{\rm \so{(Strong Real Nullstellensatz)}}\,  Let $K$ be a real closed field {\rm (e.\,g. $K=\R$)} and  let $\gotha \subseteq  K[X_1,\dots,X_n]$ an ideal. Then ${\rm I}_{K}({\rm V}_{K}(\gotha))=  \rsqrt{\gotha}$.
\end{theorem}

\begin{proof}
Let $A:=K[X_{1},\ldots , X_{n}]$ and $f\in A$. If $f\in\rsqrt{\gotha}$, then by \ref{thm:3.10} $f^{2m}+g\in \gotha$ for some $m\in\N$ and some $g\in\sum A^{2}$. Then, for every $a\in {\rm V}_{K}(\gotha)$, $f^{2m}(a)+g(a)=0\in K$ and hence $f(a)=0$, since $g(a)\in\sum K^{2}$ and $K$ is a real field. Therefore $f\in {\rm I}_{K}({\rm V}_{K}(\gotha))$.
\smallskip

If  $\gothp\in \rV{\gotha}$, then ${\rm V}_{K}(\gothp) \subseteq {\rm V}_{K}(\gotha)$ and hence  ${\rm I}_{K}({\rm V}_{K}(\gotha)) \subseteq {\rm I}_{K}({\rm V}_{K}(\gothp))=\gothp$ by Theorem\,\ref{thm:3.18}. Therefore
${\rm I}_{K}({\rm V}_{K}(\gotha))\subseteq \bigcap_{\,\gothp \in {\footnotesize \rV{\gotha}}}\,\gothp =\rsqrt{\gotha}$.
\end{proof}

Note that for a semi-real ideal $\gotha\subsetneq K[X_{1},\ldots , X_{n}]$, where $K$ is a real closed field, ${\rm V}_{K}(\gotha)\neq\emptyset$ by the Classical Real Nullstellensatz\,\ref{thm:3.16}. However, ${\rm V}_{K}(\gotha)$ may be too small to reflect any geometric properties of ${\rm V}_{\overline{K}}(\gotha)$. An extreme example is $\gotha=\langle X_{1}^{2}+\cdots +X_{n}^{2}\rangle$, in this case ${\rm V}_{K}(\gotha)=\{0\}$ which does not reveal any geometric properties of the hypersurface ${\rm V}_{\overline{K}}(\gotha)$ over the algebraic closure $\overline{K}$ of $K$. On the other hand, if $\gotha$ is not only semi-real but  a real ideal, then ${\rm V}_{K}(\gotha)$ is a ``significant'' part of ${\rm V}_{\overline{K}}(\gotha)$. We deduce this from the above Strong Real Nullstellensatz. More precisely, we prove\,:

\begin{corollary}\label{cor:3.20}
Let $K$ be a real closed field and $\gotha \subseteq  K[X_1,\dots,X_n]$ a real ideal. Then ${\rm V}_{K}(\gotha)$ is Zariski dense in the $K$-algebraic set ${\rm V}_{\overline{K}}(\gotha)$.
\end{corollary}

\begin{proof}
Since $K$ is real closed, by Euler-Lagrange Theorem (see \ref{mypar:3.1}\,(3)\,(e))
$\overline{K}= K({\rm i})$ with ${\rm i}^{2}+1=0$  is an algebraic closure of $K$.
It is enough to prove the implication\,: {\it  For every $f\in \overline{K}[X_{1},\ldots , X_{n}]$,  $\,{\rm D}(f)\cap {\rm V}_{\overline{K}}(\gotha) \neq \emptyset$ $\Rightarrow$  ${\rm D}(f)\cap {\rm V}_{K}(\gotha) \neq \emptyset$.} To prove this, write $f=g+{\rm i}h$ with $g$, $h\in K[X_{1},\ldots , X_{n}]$.
If  $f=0$ on ${\rm V}_{K}(\gotha)$, then clearly both $g$, $h\in {\rm I}_{K}({\rm V}_{K}(\gotha))=\rsqrt{\gotha}$ by Strong Real Nullstellensatz\,\ref{thm:3.19}. Now, since $\gotha$ is a real ideal, $\rsqrt{\gotha}=\gotha$ (see remarks after Theorem\,\ref{thm:3.10}) and hence $g$, $h\in\gotha$. Therefore $g$, $h$ and hence $f$ vanish on ${\rm V}_{\overline{K}}(\gotha)$.  \end{proof}

\begin{example}\label{ex:3.21}{
{\it Let $K$ be a real closed field and $\gothp\in\rSpec K[X_{1},\ldots , X_{n}]$ be a real prime ideal. Then $\gothp\,\overline{K}[X_{1},\ldots , X_{n}]$ is also a prime ideal in $\overline{K}[X_{1},\ldots , X_{n}]$. In ~particular, ${\rm V}_{\overline{K}}(\gothp)$  is an irreducible $K$-affine variety over $\overline{K}$.}\, This is seen as follows\,:\,
Since $K$ is real closed, $\overline{K}= K({\rm i})$ with ${\rm i}^{\,2}+1=0$  by Euler-Lagrange Theorem (see \ref{mypar:3.1}\,(3)\,(e)).
Suppose that $(f+{\rm i}g)(f\,'+{\rm i}g') \in \gothp\,\overline{K}[X_{n},\ldots , X_{n}]$, where $f\,,\,g\,,\,f\,'\,,\,g'\in K[X_{1},\ldots , X_{n}]$. Then $ff\,'-gg'\in\gothp$ and $fg'+gf\,'\in \gothp$ and hence $g(f\,'^{2}+g'^{2})\in\gothp$ and $f(f\,'^{\,2}+g'^{\,2})\in\gothp$. If $f+{\rm i}g\not\in \gothp\,\overline{K}[X_{n},\ldots , X_{n}]$, then one of $f$, $g$ is not in $\gothp$ and hence $f\,'^{\,2}+g'^{\,2}\in\gothp$. Since $\gothp$ is real, we have $f\,', g'\in\gothp$ and so $f\,'+{\rm i}g'\in \gothp$.

}
\end{example}

\section{Projective Real Nullstellensatz}

{\small The results proved in this section are based on the personal discussions of second author (Dilip P. Patil) with Professor Uwe Storch, Ruhr-Universit\"at Bochum, Germany and his lecture on 23 January 2003, on the occasion of 141-th birthday of Hilbert at the Ruhr-Universit\"at Bochum, Germany.

}
\smallskip

The main aim of this section is to prove the Projective Real Nullstellensatz\,:  Homogeneous polynomials $f_1,\ldots,f_r\in \R[T_{0}, \ldots , T_{n}]$, $r\leq n$,  of positive {\it odd} degrees in $n+1$ indeterminates have a common non-trivial zero in $\R^{n+1}$, or equivalently\,---\,a common zero in $n$-dimensional projective space $\mathds{P}^{n}(\R)$ over $\R$.
\smallskip

Our proof of the Projective Real  Nullstellensatz is
elementary and uses standard definitions and basic properties of
Poincar\'{e} series, projective (krull) dimension and multiplicity of
(standard) graded algebras over a field. This proof depends on   the fundamental property of the real numbers, namely\,:  every odd degree polynomial over the field of real numbers has a
real root. We say a field $K$ is a {\it $2$-field} if every odd
degree polynomial over $K$ has a root in $K$. Therefore the Projective Real
 Nullstellensatz can be generalized for $2$-fields. As an application, we prove the well-known Borsuk-Ulam Theorem.
\smallskip

{\small

\begin{mypar}\label{mypar:4.1} {\bf Notation and Preliminaries}\,
Let $K$ be a field and let ${\rm P}:=A_{0}[T_{0},\ldots , T_{n}]$ be the polynomial algebra in indeterminates $T_{0},\ldots , T_{n}$ over a commutative ring $A_{0}$. The homogeneous polynomials of degree $m\in\N$ form an $A_{0}$-submodule  ${\rm P}_{m}$ of ${\rm P}$ and ${\rm P}=\bigoplus_{m\in\N}\, {\rm P}_{m}$. Further, ${\rm P}_{m} {\rm P}_{k}\subseteq P_{m+k}$ for all $m, k\in\N$ and as an $A_{0}$-algebra ${\rm P}$ is generated by the homogeneous elements $T_{0},\ldots , T_{n}$ of degree $1$.
\smallskip

\begin{arlist}

\item {\bf Graded rings and Modules}\,
More generally, a ring $A$ is called $\N$-{\it graded} or just {\it graded} if it has a direct sum decomposition $A=\bigoplus_{m\in\N}\, A_{m}$  as an abelian group such that $A_{m} A_{k}\subseteq A_{m+k}$ for all $m,k\in\N$. In~particular, $A_{0}$ is a subring of $A$ and $A_{m}$ is an $A_{0}$-module for all $m\in\N$. For $m\in\N$, $A_{m}$ is called {\it homogeneous component} of $A$ of degree $m$ and its elements are called {\it homogeneous elements of degree} $m$.
\smallskip

A graded ring $A=\bigoplus_{m\in\N}\, A_{m}$ is called  a {\it standard graded $A_{0}$-algebra} if $A$ is generated by finitely many homogeneous elements of degree $1$ as an $A_{0}$-algebra.
A standard example of the standard graded $A_{0}$-algebra (as seen above) is the polynomial algebra ${\rm P}=A_{0}[T_{0}, \ldots , T_{n}]$ with $\deg T_{i}=1$ for all $i=0,\ldots ,n$ over a commutative ring $A_{0}$.
\smallskip

Let $A=\bigoplus_{m\in\N}\, A_{m}$  be a graded ring and $A_{+}:=\bigoplus_{m\in\N^{+}} A_{m}$.
Obviously,  $A_{+}$ is an ideal in $A$ called the {\it irrelevant ideal} of $A$. It follows that the following statements are equivalent\,: (i) $A$ is Noetherian.\, (ii) $A_{0}$ is Noetherian and $A_{+}$ is finitely generated.\, (iii) $A_{0}$ is Noetherian and $A$ is an $A_{0}$-algebra of finite type.\\[1.5mm]
{\bf Remark\,:} Note that finite type algebras over Noetherian ring are Noetherian. However, Noetherian algebras over a Noetherian ring $A_{0}$ are not always of finite type over $A_{0}$. Therefore graded rings are special for which this converse holds.
\smallskip

A {\it graded module} over the graded ring $A=\bigoplus_{m\in\N}\, A_{m}$ is an $A$-module $M$ with a direct sum decomposition $M=\bigoplus_{m\in\Z}\, M_{m}$ as an abelian group such that $A_{m}M_{k}\subseteq M_{m+k}$   for all $m\in\N$ and all $k\in\Z$. In~particular, $M_{m}$ is an $A_{0}$-submodule of $M$ for every $m\in\Z$. For $m\in\Z$, $M_{m}$ is called {\it homogeneous component} of $M$ of degree $m$ and its elements are called {\it homogeneous elements of degree} $m$.
\smallskip

Let $M\!=\!\bigoplus_{m\in\Z}\,M_{m}$ and $N\!=\!\bigoplus_{m\in\Z}\, N_{m}$ be graded $A$-mo\-dules over the graded ring $A\!=\!\bigoplus_{m\in\N}\, A_{m}$. An $A$-module homomorphism $f:M\to N$  is called {\it homogeneous of degree} $r$ if $f(M_{m})\subseteq N_{m+r}$ for every $m\!\in\!\Z$.\\[1mm]
An $A$-submodule $M'$ of the graded $A$-module $M$ is called {\it homogeneous} if $M'_{m}:=\pi_{m}(M') =M'\cap M_{m}\subseteq M'$, where $\pi_{m}:M\to M_{m}$, $m\in\Z$ are the projections of the graded $A$-module $M$.  If the $A$-submodule  $M'\subseteq M$ is homogeneous, then $M'=\bigoplus_{m\in\Z}\, M'_{m}$ is a graded $A$-module and the canonical injective map $M'\to M$ is homogeneous of degree $0$.
An $A$-submodule $M'$ of the graded $A$-module $M$ is homogeneous if and only if $M'$ has a generating system consisting of homogeneous elements.
Further, the residue-class module $M/M'$ has the direct sum decomposition $M/M'=\bigoplus_{m\in\Z}\, \overline{M}_{m}$, where  $\overline{M}_{m}:=M_{m}/M'_{m}$. Obviously, $M/M'$ with this gradation
is a graded $A$-module and the canonical surjective map $M\to M/M'$ is homogeneous of degree $0$.
\smallskip

An ideal $\gotha\subseteq A$ is called {\it homogeneous} if $\gotha$ is a homogeneous submodule of $A$.
\smallskip

Let $f:M\to N$ be a homogeneous homomorphism of degree $r$, then $\Ker f$ and $\img f$ are homogeneous submodules of $M$ and $N$, respectively and the canonical $4$-term sequence
\begin{align*}
0\to \Ker f\longrightarrow M \longrightarrow  N \longrightarrow \Coker f \to 0
\end{align*}
is an exact sequence of graded $A$-modules and homogeneous homomorphisms. Further, for every $m\in\Z$, the sequence of $A_{0}$-modules
\begin{align*}
0\to (\Ker f)_{m} \longrightarrow M_{m} \longrightarrow  N_{m+r} \longrightarrow (\Coker f )_{m+r}\to 0
\end{align*}
is exact.
\smallskip

\begin{alist}

\item {\bf Shifted graded modules}\,
The following shift operation is very useful\,:  For $k\in\Z$, a graded $A$-module $M(k)$ obtained from the graded $A$-module $M=\bigoplus_{m\in\Z}\, M_{m}$ with $M(k)_{n}:=M_{k+n}$ for all $n \in \Z$ is called the $k$-{\it th shifted  graded $A$-module} of $M$.
In~particular, we have the $k$-shifted graded $A$-module $A(k)$ of the graded $A$-module $A$. Clearly,
an $A$-module homomorphism $f:M\to N$ is homogeneous of degree $r$ if and only if
$f:M(-r)\to N$, or $f:M\to N(r)$ is homogeneous of degree $0$.
\smallskip

\item {\bf Noetherian graded modules}\,
We consider the case when  $A_0=K$ is a field and  $A$ is a  standard graded  $K$-Algebra. If  $t_0,\ldots,t_n\in A_1$ generates $A$ as a $K$-algebra, i.\,e.  $A=K[t_0,\ldots,t_n]$, then the  $K$-algebra substitution homomorphism  $\varepsilon:K[T_0,\ldots,T_n]\to A$ with  $T_i\to t_i$, $i=0,\ldots,n$, is homogeneous and surjective, and hence $A$ is isomorphic to the residue-class algebra   $K[T_0,\ldots,T_n]/\mathfrak{A}$ of  ${\rm P}=K[T_0,\ldots,T_n]$ modulo the homogeneous  \emph{relation ideal} $\mathfrak{A}:={\rm Ker}\,\varepsilon$. Every $A$-module is also ${\rm P}$-module by the restriction of scalars by using  $\varepsilon$. We consider graded $A$-modules  $M$ which are finite over $A$, i.\,e. finitely generated over $A$. If   $x_1,\ldots,x_r$ is a homogeneous generating system for $M$ of degrees  $\delta_1,\ldots,\delta_r\in\Z$, then the canonical homomorphism  $A(-\delta_1)\oplus\cdots\oplus A(-\delta_r)\
\longrightarrow  M$ with  $e_\rho\mapsto x_\rho$, $\rho=1,\ldots,r$, is homogeneous  (of degree $0$) and surjective. The standard basis element $e_\rho\in A(-\delta_\rho)$ has the degree  $\delta_\rho$. \emph{If $A$ is a standard graded  $K$-algebra and if $M$ is a finite  $A$-module, then  $M$ is a  Noetherian\footnote{The Noetherian property of modules is named after Emmy Noether (1882-1935) who was the first one to discover the true importance of this property.  Emmy Noether  is best known for her contributions to abstract algebra, in particular, her study of chain conditions on ideals of rings.} $A$-module}, i.\,e. every $A$-submodule of $M$ is also a finite $A$-module. This is equivalent with the condition that in  $M$ there is no infinite proper  \emph{ascending} chain  $M_0\subset M_1\subset M_2\subset\cdots\subseteq M$ of $A$-submodules, or also with the condition that every non-empty set of  $A$-submodules of $M$  has a (at least one) maximal element (with respect to the inclusion).
\smallskip

We will use the following fundamental lemma on the  {\it Lasker-Noether decomposition} \footnote{Due to Emanuel Lasker (1868 -- 1941) and Max Noether (1844-1920), father of Emmy Noether.} \,:
\smallskip

{\bf Lemma}\ {\rm \so{(Lasker\hbox{-}Noether decomposition)}}\,
{\it Let $M$ be a finite graded module over the standard graded  $K$-algebra $A$. Then there exists a chain of graded  submodules
 $\displaystyle 0=M_0\subsetneq M_1\subsetneq\cdots\subsetneq M_r=M\,$,
and homogeneous prime ideals $\mathfrak{p}_1,\ldots,\mathfrak{p}_r\subseteq A$ and integers  $k_1,\ldots,k_r$ with $M_\rho/M_{\rho-1}\cong(A/\mathfrak{p}_\rho)(-k_\rho)$, $\rho=1,\ldots,r$. In~particular, $\mathfrak{p}_{1}\cdots \mathfrak{p}_{r} M=0$.}
\smallskip

{\bf Proof}\, \ First we show that if $M\neq 0$, then it contains a submodule of the isomorphism type $(A/\mathfrak{p})(-k)$, or equivalently, a homogeneous element $0\neq x\in M$ such that the annihilator ideal $\Ann_{A}x :=\{a\in A\mid ax=0\}= \mathfrak{p}$ is prime. Let $0\neq x_{0}\in M.$ If $\Ann_{A} x_{0}$ is not prime, then there exist $a,b\in A$ with $ax_{0}\neq 0$, $bx_{0}\neq 0$ and $abx_{0}=0$. Then $x_{1}:=bx_{0}\neq 0$, $a\in\Ann_{A} x_{1}$, $a\not\in \Ann_{A}x_{0}$ and so $\Ann_{A} x_{0} \subsetneq \Ann_{A} x_{1}$.
Since $A$ is Noetherian, in finitely many steps, we get an element $x(=x_{s})\in M$, $x\neq 0$ with $\Ann_{A} x$ prime. (One can also  directly choose a homogeneous element $0\neq x\in M$ such that $\Ann_{A} x$ is a maximal element in $\{\Ann_{A} y\mid 0\neq y\in M\}$.) Now, we construct the required chain in $M$. If $M\neq 0$, then there exists a submodule $M_{1}\cong (A/\mathfrak{p}_{1})(-k_{1})$. If $M/M_{1}\neq 0$, then there exists $M_{2}/M_{1} \subseteq M/M_{1}$ with $(M_{2}/M_{1})\cong (A/\mathfrak{p}_{2})(-k_{2})$ and so on. The chain $0\subsetneq M_{1}\subsetneq M_{2} \subsetneq \cdots $ after finitely many steps will end at $M$, since $M$ is Noetherian. (One can also choose a maximal homogeneous submodule $N\subseteq M$ for which the required chain of submodules of $N$ exists. Then prove that $N$ is necessarily $M$.) \dppqed

\end{alist}

\item {\bf Poincar\'e series}\,
Let $M=\bigoplus_{m\in\Z}\, M_{m}$ be a finite graded module over the standard graded $K$-algebra $A=\oplus_{m\in\N}\, A_{m}=K[t_{0},\ldots , t_{n}]$ with $A_{0}=K$ and $t_{0},\ldots, t_{n}\in A_{1}$. Then   $M_{m}$, $m\in\Z$, are finite dimensional $K$-vector spaces and $M_{m}=0$ for $m<<0$. Therefore, the {\it Poincar\'e series}
\begin{align*}
\mathscr{P}_{M}(Z):= \sum_{m\in\Z} (\,\Dim_{K} M_{m}\,)\, Z^{m}
\end{align*}
is well-defined and is a Laurent-series (with coefficients in $\N$).
If $K[T_0,\ldots,T_n]\to A=K[t_0,\ldots,t_n]$ is a representation of  $A$ as a residue class algebra of a polynomial algebra, then  $\mathscr{P}_M$ is the same even if $M$ is considered as  a $K[T_0,\ldots,T_n]$-module.
\smallskip

\begin{alist}
\item {\bf Computation rules for Poincar\'e series}\,
We note the following elementary computational rules for Poincar\'e series of finite graded $A$-modules\,:\,
Let $M$, $M_{1}, \ldots , M_{r}$ be a finite graded modules over the standard graded $K$-algebra $A=\oplus_{m\in\N}\, A_{m}=K[t_{0},\ldots , t_{n}]$ with $A_{0}=K$ and $t_{0},\ldots, t_{n}\in A_{1}$.
\smallskip

{\it
{\rm(1)}  $\mathscr{P}_{M(-k)}=Z^k\,\mathscr{P}_M\,$ for all $\,k\in\Z\,$.\,
{\rm(2)} If $0\to M_r\to\cdots\to M_0\to0$ is an exact sequence with  homogeneous homomorphisms of degrees {\rm 0},  then  $\sum_{\rho=0}^r(-1)^\rho\,\mathscr{P}_{M_\rho}=0$.\,
{\rm(3)} If $f\!\in\!A_\delta$ is a homogeneous non-zero divisor for the $A$-module  $M$ of degree  $\!\delta\!>\!0$, \hbox{then  $\mathscr{P}_{M/fM}\!=\!(1\!-\!Z^\delta)\,\mathscr{P}_M$.}\,
{\rm(4)} If $0=M_0\subseteq M_1\subseteq\cdots\subseteq M_r=M$ is a chain of homogeneous submodules of the $A$-module $M$, then  $\mathscr{P}_M=\sum_{\rho=1}^r\mathscr{P}_{M_\rho/M_{\rho-1}}$.
}
\smallskip

\item
The following  fundamental lemma was already in the work of Hilbert with a complicated proof.
\medskip

{\bf Lemma}\,
{\it Let $M$ be a finite graded module over the standard graded $K$-algebra $A=K[t_{0}, \ldots , t_{n}]$, $t_{0}, \ldots, t_{n}\in A_{1}$. Then
$\, \mathscr{P}_{M}= F/(1-Z)^{n+1}\,$
with a Laurent-polynomial $F\in\Z[Z^{\pm 1}]$.}
\smallskip

If $M\neq 0$, then after cancelling the highest possible power of $(1-Z)$, we get a {\it unique} representation
\begin{align*}
\mathscr{P}_{M}= \frac{Q}{(1-Z)^{d+1}}\,, \ d\geq -1
\end{align*}
with a Laurent-polynomial $Q\in\Z[Z^{\pm 1}]$, $Q(1)\neq 0$. For $M=0$, $d=-1$ and $Q=0$. The partial fraction decomposition is
\begin{align*}
\mathscr{P}_{M}= \widetilde{Q}+ \sum_{i=0}^{d} \frac{c_{i}}{(1-Z)^{i+1}} \equiv \sum_{i=0}^{d}\frac{c_{i}}{(1-Z)^{i+1}}
\end{align*}
with a uniquely determined Laurent-polynomial $\widetilde{Q}\in\Z[Z^{\pm 1}]$ and unique integers $c_{0},\ldots , c_{d}\in\Z\,$, where we write $G\equiv H$ for two Laurent-series $G$, $H$ if and only if they differ by a Laurent-polynomial.
\smallskip

Now, using the formula $(1-Z)^{-(n+1)}=\sum_m{{m+n}\choose{n}}Z^m$ which can be proved directly by differentiating (termwise) $n$-times the geometric series $(1-Z)^{-1}=\sum_m Z^m\,$, we get\,:
\smallskip

For $m\gg 0$ (more precisely for $m>\deg \widetilde{Q}$)\,, we have
\begin{align*}
\Dim_{K} M_{m} = \chi_{M}(m):=\sum_{i=0}^{d} c_{i} \binom{m+i}{i} \ \quad  \hbox{for} \ \ m\gg 0 \,,
\end{align*}
where $\chi_{M}:\Z\to\N$  is a polynomial function (over $\Q$) of degree $d$ and in~particular, if $d\geq 0$, then
\begin{align*}
\Dim_{K} M_{m} = \chi_{M}(m) \ \sim \ c_{d}\cdot \frac{m^{d}}{d!} = O(m^{d})\, \ \ \hbox{for} \ \ m\to\infty\,.
\end{align*}
where $O$ is the ``Big O'' symbol\,\footnote{\label{foot:15}The symbol ``Big~O'' was first introduced by the number theorist Paul Bachmann (1837-1920) in 1894. Another number theorist Edmund Landau (1877-1938) adopted it and was inspired to introduce the ``small~o'' notation in 1909. These symbols describe the limiting behaviour of a function. More precisely\,:\,
For $\R$-valued functions $f$, $g:U\to\R$ defined on some subset $U\subseteq \R$, one writes\,:\,
(i) $f(x) =O(g(x))$  ($|f\,|$ is  bounded above  by $|g|$, up to constant factor,  asymptotically)  if there exists a constant $M>0 $ and a real number $x_{0}\in\R$ such that  $|f(x)|\leq M\,|g(x)|$ for all $x\geq x_{0}$, or equivalently $\limsup_{x\to\infty} |f(x)/g(x)| <\infty$.\,
(ii)\, $f(x) =o(g(x))$  ($f$ is dominated by $g$  asymptotically)  if  $\lim_{x\to\infty} |f(x)/g(x)| =0$.}
and  $\,\sim\,$ denote the asymptotic equality. The case $d\!=\!\!-1$ is characterized by $\Dim_{K} M_{m}\!\!=\!0$ for $m\gg 0$, or by $\Dim_{K} M\!=\!\sum_{m\in \Z} \Dim_{K} M_{m}\!\!=\!Q(1)\! < \!\infty$.
\end{alist}
\smallskip

\item {\bf Hilbert series}\,
Incidentally, instead of Poincar\'e-series it is comfortable to consider the {\it Hilbert-series}
\begin{align*}
\mathscr{H}_M=\sum\nolimits_{m\in\Z}h_M(m)Z^m= \mathscr{P}_{M}/(1-Z)\equiv\sum\nolimits_{i=0}^{d+1}e_i/(1-Z)^{i+1}
\end{align*}
with the \emph{Hilbert-Samuel function}\, $h_{M}:\Z \to\N$\,:
\begin{align*}
h_M(m)=\sum\nolimits_{k\le m}\Dim_K M_m=\Dim_K\left(\bigoplus\nolimits_{k\le m}M_k\right)
\end{align*}
and put $e_i:=c_{i-1}$, if $i>0$, and $e_0:=\widetilde{Q}(1)$. For  $m\gg 0$, the values $h_M(m)$ are equal to the values of the  \emph{Hilbert-Samuel Polynomial}
\begin{align*}
H_M(m)=\sum\nolimits_{i=0}^{d+1}e_i{{m+i}\choose{i}}\sim e_{d+1}\cdot m^{d+1}/(d+1)!=O(m^{d+1}).
\end{align*}
The integer $d$ is an approximate measure of the size of $M$. For example, if $M=A=P=K[T_{0}, \ldots , T_{n}]$, $n\in\N$, then $\mathscr{P}_{K[T_{0},\ldots, T_{n}]}\!=\!1/(1-Z)^{n}$ and
$\mathscr{H}_{K[T_{0},\ldots , T_{n}]}\!=\!1/(1-Z)^{n+1}$, so that $d=n$.
\smallskip

\item {\bf Dimension and Multiplicity}\,
The integer $d$ is called the ({\it projective}) {\it dimension} ${\rm pd}(M)$  and $d\!+\!1$ is the ({\it affine} or {\it Krull}-) {\it dimension} ${\rm d}(M)$ of the graded module $M$.
The integer ${\rm e}(M)\!:=\!e_{d+1}\!\!=\!e_{{\rm d}(M)} (=\!c_{{\rm pd}(M)})$ if  ${\rm pd}(M)\!\geq\!0)$ is called the {\it multiplicity of the graded module}  $M$ if ${\rm pd}(M)\!\geq\! 0$. Note that ${\rm e}(M)>0$  if $M\!\neq\! 0$. If $M\!=\!0$, then ${\rm d}(0)\!=\!{\rm e}(0)\!=\!0$. If
$\mathscr{P}_{M}\!=\!Q/(1\!-\!Z)^{{\rm pd}(M)}\!\!$,
then  $\mathscr{H}_{M}\! =\! Q/(1\!-\!Z)^{1\!+\!{\rm pd}(M)}\!\!$
and   ${\rm e}(M)\!=\!Q(1)$.
\smallskip

In particular, the projective dimension ${\rm pd}(K[T_{0},\ldots , T_{n}])=n$, the affine dimension ${\rm d}(K[T_{0},\ldots , T_{n}])=n+1$ and the multiplicity ${\rm e}(K[T_{0},\ldots , T_{n}])=1$.
 \smallskip

The following computational rules for ${\rm d}(M)$ and ${\rm e}(M)$ are easy to verify by using the computational rules for Poincar\'e series given in (2)\,(a)\,: \
\smallskip

\begin{alist}

\item {\bf Computational rules for dimension and multiplicity}\,
{\it Let $K$ be a field and $A=\bigoplus_{n\in\N} A_{n}$ be a
standard graded $K$-algebra. Then for a finite graded
$A$-module $M$, we have\,:
\medskip

{\rm (1)}\,  ${\rm d}(M)={\rm d}(M(-k))$ and
${\rm e}(M)={\rm  e}(M(-k))$, $k\in\Z$.
\smallskip

{\rm (2)}\,
Let $0\to M_{r}\to M_{r-1}\to \cdots \to M_{0}\to 0$ be an exact sequence of homogeneous homomorphisms. Then
\[ \sum _{\rho\,,\, {\rm d}(M_{\rho})=d} (-1)^{\rho}
{\rm e}(M_{\rho})=0\,, \text{ where } d:=\max_{0\leq \rho\leq r} \{{\rm d}(M_{\rho})\}.\]

{\rm (3)}\,
Let $f\in A_{\delta}$ be a homogeneous element of degree $\delta>0$. Then  \ $\displaystyle {\rm d}(M/fM)\geq {\rm d}(M)-1$. \
Moreover, if $f$ is a non-zero divisor for $M$ and $M\neq 0$, then \
${\rm d} (M/fM)={\rm d}(M)-1\,$ and
$\,{\rm e}(M/fM)=\delta \cdot {\rm  e}(M)\,$.
\smallskip

{\rm (4)}\, {\rm (}Associativity formula{\rm )}\,
Let  $0=M_{0}\subseteq  M_{1}\subseteq ...\subseteq M_{r}=M$ be a chain of graded $A$-submodules of $M$.
\[ \text{ Then } d:={\rm d}(M)= \max_{1\leq \rho \leq r}
\{{\rm d}(M_{\rho}/M_{\rho -1})\}\,, \text{ and }
 {\rm e}(M)=\sum _{\rho\,,\, {\rm d}(M_{\rho}/M_{\rho-1})=d} {\rm  e}(M_{\rho}/M_{\rho-1}). \]

{\rm (5)}\,
Moreover, if in (4) there are homogeneous prime ideals $\mathfrak{p}_{1}, \ldots , \mathfrak{p}_{r}$ and integers $k_{1}, \ldots , k_{r}\in \Z\,$ with $\,M_{\rho}/M_{\rho-1} \cong (A/\mathfrak{p}_{\rho})(-k_{\rho})$, $\rho=1,\ldots , r$, are as in {\rm Lemma in\,\ref{mypar:4.1}\,(1)\,(b)}\,,  then
$\displaystyle
{\rm d}(M)= \max_{1\leq \rho \leq r}
\{{\rm d}(A/\mathfrak{p}_{\rho})\}\,$  and
$\,\displaystyle {\rm e}(M)=\!\!\!\!\!\!\!\!\sum _{\rho\,,\, {\rm d}(A/\mathfrak{p}_{\rho})={\rm d}(M)}\!\!\!\!\!\!\!\!{\rm  e}(A/\mathfrak{p}_{\rho})\,$.
In~particular, if $M\neq 0$,    then there are prime ideals $\mathfrak{p}_{\rho}$ with ${\rm d}(A/\mathfrak{p}_{\rho})\!=\!{\rm d}(M)\,.$

}
\end{alist}

 \end{arlist}
 \end{mypar}
 }
 \smallskip

\begin{mypar}\label{mypar:4.2}\, {\bf Projective algebraic sets}\,
Let $K$ be a field and let ${\rm P}:=K[T_{0}, \ldots , T_{n}]$ be the standard polynomial $K$-algebra with the standard gradation
${\rm P}:= \bigoplus_{m\in\N} {\rm P}_{m}$.  Let
\begin{align*}
\mathds{P}_{{\rm P}}(K)  := \mathds{P}^{n}(K) =\left(K^{n+1}\!\smallsetminus\!\{0\}\right)/\!\sim\,=\{\langle\tau\rangle =
\langle\tau_{0},\ldots , \tau_{n}\rangle\mid \tau=(\tau_{0},\ldots , \tau_{n})\in K^{n+1}\setminus\{0\} \}
\end{align*}
be the quotient space of the equivalence relation $\,\sim\,$ on the set $\left(K^{n+1}\!\smallsetminus\!\{0\}\right)$ defined by
$\tau\!=\!(\tau_{0},\ldots , \tau_{n}) \sim \sigma\!=\!(\sigma_{0},\ldots , \sigma_{n})$ if there exists $\lambda\in K^\times$ such that
$\tau_{i} = \lambda \sigma_{i}$ for all $i=0,\ldots , n$.  This is called the {\it $n$-dimensional projective space over} $K$.
\smallskip

For a standard graded $K$-algebra
$\,A= \bigoplus_{m\in\N}\, A_{m}=K[t_{0},...,t_{n}]\,$ with $\,t_{0},...,t_{n}\in A_{1}$, let $\mathfrak{A}$ be the kernel of the substitution homomorphism   $\,\varepsilon:K[T_{0},...,T_{n}]\to A\,$, $\,T_{i}\mapsto t_{i}\,$, $\,i=0,...,n\,$. Then $\varepsilon$ induces a  homogeneous $K$-algebra isomorphism $P/\mathfrak{A}\iso A$ and the set of the common zeroes
\begin{align*}
\mathds{P}_{K}(A)={\rm V}_{+}(\mathfrak{A}):=\{\langle\tau\rangle \in\mathds{P}^{n}(K) \mid & F(\tau)=0 \
\hbox{for all homogeneous} \  F\in \mathfrak{A}\}\subseteq \mathds{P}^{n}(K)
\end{align*}
of the homogeneous relation ideal $\mathfrak{A}$ in $\mathds{P}^{n}(K)$, is called the {\it projective algebraic set} $\,\mathds{P}_{A}(K)$ {\it of $K$-valued points}.
Further, if $F_{1}, \ldots , F_{m}\in\mathfrak{A}$ is a homogeneous system of generators for $\mathfrak{A}$, then
\begin{align*}
\mathds{P}_{A}(K)={\rm V}_{+}(F_{1},\ldots , F_{m}) = \{\langle\tau\rangle \in\mathds{P}^{n}(K) \mid F_{i}(\tau) =0\,,\, i=1,\ldots , m\}\,.
\end{align*}
It is easy to see that the description of $\mathds{P}_{A}(K)$ is independent of the representation $A\leftiso P/\mathfrak{A}$.  If $f\in A$ is a homogeneous element with a homogeneous representative $F\in P$, then the zero set
\begin{align*}
{\rm V}_{+}(f)= \{\langle\tau\rangle \in\mathds{P}_{A}(K) \mid F(\tau)=0\}
\end{align*}
of $f$ in $\mathds{P}_{A}(K)$ is well-defined. In~particular, for a homogeneous ideal $\mathfrak{a}\subseteq A$, generated by homogeneous elements $f_{1}, \ldots , f_{r}\in A$,  we have the representation\,:
\[\mathds{P}_{A/\mathfrak{a}}(K)={\rm V}_{+}(f_{1},\ldots , f_{r}) = \bigcap_{\rho=1}^{r} {\rm V}_{+}(f_{\rho}) \subseteq \mathds{P}_{A}(K)\,.\leqno{(4.2.1)}
\]
\end{mypar}

Now we prove the following very important  and useful lemma\,:

\begin{lemma}\label{lem:4.3}\,
 Let $K$ be a field and let ${\rm P}:=K[T_{0}, \ldots , T_{n}]$ be the standard polynomial $K$-algebra with the standard gradation.

 \begin{alist}

 \item
 For a point    $\langle\tau\rangle=\langle\tau_{0},\ldots , \tau_{n}\rangle\in\mathds{P}^{n}(K)$, the vanishing ideal
\begin{align*}
 \gothP_{\langle\tau\rangle}:=\langle \{F\in {\rm P}\mid F \ \hbox{is a homogeneous polynomial in } \  {\rm P\,} \ \hbox{with} \  F(\tau)=0 \} \rangle
\end{align*}
generated by the homogeneous polynomials which vanish on $\langle\tau\rangle$, is a   homogeneous prime ideal in ${\rm P}$ with ${\rm P}/\gothP_{\langle\tau\rangle} \iso K[T]$ a standard graded polynomial algebra in one indeterminate $T$. In~particular, the projective dimension
${\rm d}(P/\mathfrak{P}_{\langle\tau\rangle})=0$ and the multiplicity
${\rm e}(P/\mathfrak{P}_{\langle\tau\rangle})=1$.

\item
If $\,\mathfrak{P}\subseteq {\rm P}$ is a homogeneous prime ideal with
${\rm d}(P/\mathfrak{P})=0$ and ${\rm e}(P/\mathfrak{P}) =1$, then there exists a unique point  ${\langle\tau\rangle}\in\mathds{P}^{n}(K)$ such that
$\mathfrak{P}=\mathfrak{P}_{\langle\tau\rangle}$.
\end{alist}
 \end{lemma}

 \begin{proof}
{\bf  (a)\,} We may assume that $\tau_{0}=1$. It is easy to verify that $\gothP_{\langle\tau\rangle}$ is generated by $\tau_{j} T_{i}-\tau_{i} T_{j}$, $0\leq i\,,\,j\leq n$, $i\neq j$ and that the surjective $K$-algebra homomorphism ${\rm P} \to K[T]$ defined by $T_{0}\mapsto T$ and $T_{i}\mapsto \tau_{i} T$, $i=1,\ldots , n$, has the kernel $\gothP_{\langle\tau\rangle}$ and hence $\,P/\mathfrak{P}_{\langle\tau\rangle} \cong K[T]\,$.
\smallskip

{\bf  (b)\,}  Let $\mathfrak{P}\subseteq {\rm P}$ be a homogeneous prime ideal with ${\rm d}(P/\mathfrak{P})=0$ and ${\rm e}(P/\mathfrak{P}) =1$. Then the  $K$-subspace $\mathfrak{P}_{1}\subseteq P_{1}$ is of codimension $1$, since ${\rm d}(P/\mathfrak{P})=0$ and  $\,1={\rm e}(P/\mathfrak{P}) \geq \Dim_{K}({\rm P}/\mathfrak{P})_{m}$ for every $m\in\N$.  Therefore $\gothP =\langle \gothP_{1}\rangle = \gothP_{\langle\tau\rangle}$ for a unique point $\langle\tau\rangle\in \mathds{P}^{n}(K)$.
\end{proof}

For a graded ring $A=\bigoplus_{m\in\N}\, A_{m}$, the set of homogeneous prime ideals is denoted by $\hSpec A$.

\begin{corollary}\label{cor:4.4}\,
For a standard graded $K$-algebra
$\,A= \bigoplus_{m\in\N}\, A_{m}=K[t_{0},...,t_{n}]\,$ with $\,t_{0},...,t_{n}\in A_{1}$, and  the substitution homomorphism   $\,\varepsilon:K[T_{0},...,T_{n}]\to A\,$, $\,T_{i}\mapsto t_{i}\,$, $\,i=0,...,n\,$, let $\mathfrak{A}=\Ker\varepsilon$. Then the map
\begin{align*}
\mathds{P}_{A}(K) \longrightarrow \{\gothp\in \hSpec A \mid
{\rm d}(A/\mathfrak{p})=0 \ \hbox{and} \  {\rm e}(A/\mathfrak{p})=1\}\,,
\langle\tau\rangle \longmapsto \gothp_{\langle \tau\rangle}:=\mathfrak{P}_{\langle\tau\rangle}/\gothA
\end{align*}
is bijective.
\end{corollary}

\begin{proof}
Immediate from  Lemma\,\ref{lem:4.3}, since
$\langle\tau\rangle \in\mathds{P}_{A}(K)$ if and only if $\mathfrak{A} \subseteq\mathfrak{P}_{\langle\tau\rangle}$.
\end{proof}

 \begin{lemma} \label{lem:4.5}\,
Let $K$ be a field and $C=\bigoplus_{m\in\N}\, C_{m}$ be a standard graded $K$-algebra such that $C$ is an integral domain with ${\rm pd}(C)=0$.  Then there exists a finite field extension $L\,\vert\,K$ such that the multiplicity ${\rm e}(C)$ is equal to $[L:K]$.
\end{lemma}

\begin{proof}
Since $C$ is a standard graded $K$-algebra, $C_{1}\neq 0$.  Choose  $t \in C_{1}$, $t\neq 0$. Then, since  $t C_{m}\subseteq C_{m+1}$ for all $m\in\N$ and $t$ is a non-zero divisor in $C$,  the numerical function $m\mapsto \Dim_{K} C_{m}$, is monotone increasing and hence is stationary with the value ${\rm e}(C)=\Dim_{K} C_{m}$ for  $m\gg 0$. But, then there exists a unique integer $s\in\N$ such that the ascending chain of finite dimensional $K$-vector  spaces
$\,C_0=K \subsetneq  C_{1}/t \subsetneq C_{2}/t^{2} \subsetneq \cdots \subsetneq C_{s}/t^{s} =C_{s+1}/t^{s+1} = \cdots \,$  is stationary and hence $L:=C_{s}/t^{s}$ is an integral domain which is a finite $K$-algebra of  the dimension $\Dim_{K} C_{s}= {\rm e}(C)$. Therefore $L$ is a finite field extension of $K$ with $[L:K]={\rm e}(C)$.
\end{proof}

First note the following  classical Hilbert's Nullstellensatz for an algebraically closed field (see\,\cite[\S\,3]{Hilbert1893})\,:

\begin{theorem}\label{thm:4.6}{\rm \so{(Hilbert's Nullstellensatz)}}\,
Let $K$ be an algebraically closed field
and  $A$ be  a standard graded $K$-algebra of projective dimension $d={\rm pd}(A)\geq 0$. Further,  let $f_{1}, \ldots , f_{r}\in A$ be homogeneous elements of positive degrees, \, $r\leq d$.
Then $f_{1}, \ldots , f_{r}$ have a common zero in $\mathds{P}_{A}(K)$, {\rm i.\,e.,} $\emptyset \neq \mathds{P}_{A/\mathfrak{a}}(K) = {\rm V}_{+}(f_{1},\ldots , f_{r}) \subseteq \mathds{P}_{A}(K)$, where $\mathfrak{a} :=Af_{1}+\cdots +Af_{r}$.
\end{theorem}

 \begin{proof}
 By induction on $d$ and $r$.  If $d=0$, then $r=0$. By Lemma in \ref{mypar:4.1}\,(1)\,(b)  there exists a homogeneous prime ideal $\mathfrak{p}\subseteq A$ with
 ${\rm d}\,(A/\mathfrak{p})=0$. By Lemma~\ref{lem:4.5}, necessarily ${\rm e}\,(A/\mathfrak{p})=1$, since $K$ is algebraically closed and hence $\mathfrak{p}$ defines\,---\,by Corollary~\ref{cor:4.4}\,---\,a point  in $\mathds{P}_A(K)$.
\smallskip

For the inductive step from $d$ to $d+1$, consider a prime ideal $\mathfrak{p}\subseteq A$ with $d={\rm d}\,(A/\mathfrak{p})$. It is enough to prove that $\emptyset\not=V_+(\overline{f}_1,\ldots,
\overline{f}_r)\subseteq\mathds{P}_{A/\mathfrak{p}}(K)\subseteq\mathds{P}_A(K)$, where $\overline{f}_1,\ldots,\overline{f}_r$ denote the residue classes of $f_1,\ldots,f_r$ in $A/\mathfrak{p}$. We may therefore assume that $A$ is an integral domain and $f_r\not=0$. Then ${\rm d}\,(A/Af_r)=d-1$. By  induction hypothesis it follows  that $\emptyset\not=V_+(\overline{f}_1,\ldots,\overline{f}_{r-1})=
V_+(f_1,\ldots,f_r)\subseteq\mathds{P}_{A/Af_r}(K)$, where now $\overline{f}_1,\ldots,\overline{f}_{r-1}$ are the residue classes in $A/Af_r$. \end{proof}

The following theorem is also called Hilbert's Nullstellensatz. It is also
known as the {\it Identity theorem} for polynomial functions.

\begin{theorem}\label{thm:4.6a}{\rm \so{(Identity theorem)}}\,
Let $K$ be an algebraically closed field and let $A$ be  a standard graded $K$-algebra of projective dimension $d={\rm pd}(A)\geq 0$ which is an integral domain.
If a homogeneous element $f\in A$ vanishes at all points of $\,\mathds{P}_{A}(K)$, then $f=0$.
\end{theorem}

\begin{proof}
By induction on $d$.  For $d=0$, by Lemma~\ref{lem:4.5}, necessarily ${\rm e}\,(A)=1$, since $K$ is algebraically closed and hence $\mathds{P}_A(K)=\{\tau_0\}\subseteq \mathds{P}^{n}(K)$, where $\tau_{0}$ correspond to the zero homogeneous prime ideal (since $A$ is an integral domain) by Corollary~\ref{cor:4.4}. Therefore $f=0$, since $f(\tau_{0})=0$.  Assume that  $d>0$ and $\deg f>0$. Suppose on the contrary that $f\not=0$.
\vskip2pt

\noindent {\rm (*)}\hspace{5mm}  We claim that\,: there exists a homogeneous  prime ideal $\mathfrak{q}\neq 0$ in $A$ with $f\not\in\mathfrak{q}$.
\vskip2pt

For a proof of (*) consider  $M:=\overline{A}:=A/Af$ which has the projective dimension ${\rm pd}(\overline{A})=d-1$. Further, by the Lemma in \ref{mypar:4.1}\,(1)\,(b), there exists a chain  $\,0=M_0\subsetneq M_1\subsetneq\cdots\subsetneq M_r=M\,$ of graded submodules,  homogeneous prime ideals $\gothp_{1},\ldots , \gothp_{r}$ in $A$ and integers $k_{1},\ldots , k_{r}$ with $M_\rho/M_{\rho-1}=(A/\mathfrak{p}_\rho)(-k_\rho)$, $i=1,\ldots , r$ and $\gothp_{1}\cdots \gothp_{r} M =0$, i.\,e. $\gothp_{1}\cdots \gothp_{r}\subseteq Af$. Therefore it follows that there exists a {\it finite} subset $\mathscr{P}\subseteq \{\gothp_{1}, \ldots, \gothp_{r}\}$ of homogeneous prime ideals in $A$ such that
$f\in \gothp$ and $d(A/\gothp)=d-1\geq 0$ for every $\gothp\in\mathscr{P}$. Further, $\gothp\cap A_{1}\subsetneq A_{1}$ for every $\gothp\in\mathscr{P}$. Therefore by prime avoidance there exists $g\in A_{1}$ with $g\not\in \cup_{\gothp\in\mathscr{P}}\,\gothp$. Now, since $d(A/Ag)=d-1$, we can choose a homogeneous prime ideal $\gothq$ in $A$ with $g\in\gothq$ and $d(A/\gothq)=d-1$. It is clear that $f\not\in \gothq$. This proves the claim (*).
On the other hand, since $d(A/\gothq)=d-1$ and $\mathds{P}_{A/\gothq}(K)\subseteq \mathds{P}_{A}(K)$, by induction hypothesis, it follows that $f\in\gothq$ which contradicts the claim (*).
\end{proof}

\begin{remark}\label{rem:4.6b}\,
In the Theorem\,\ref{thm:4.6a}, it is enough to assume that $A$ is reduced. Then the zero ideal in $A$ is an intersection of finitely many homogeneous prime ideals, namely as in the Lemma in \ref{mypar:4.1}\,(1)\,(b) for $M=A$.
\end{remark}

Now, we prove the analogues  of  the above Theorems\,\ref{thm:4.6} and \ref{thm:4.6a} for $2$-fields.
For this, first we recall a definition and some basic results for $2$-fields.
The only property of the field $\R$ of real numbers which will be used in the following is\,: {\it every polynomial of odd degree with coefficients  in $\R$ has a  zero in $\R$.} We would like to formulate this property axiomatically\,:

\begin{definition}\label{def:4.7}\,
A field $K$ is called a {\it $2$-field} if   every  polynomial $F\in K[X]$ of odd degree has a zero in $K$. The $2$-fields are defined in \cite{pfister}.
For example, the fields $\R$ and $\C$ of real and complex numbers are $2$-fields. More generally,  algebraically closed fields are $2$-fields and every real closed field is a $2$-field, see \ref{mypar:3.1}\,(3)\,(e).
\end{definition}

The following elementary characterization of $2$-fields is useful\,:

 \begin{lemma}\label{lem:4.8}\,
 For a field $K$, the following statements are equivalent\,$:$
{\rm (i)} $\,K$ is a $2$-field.\,
{\rm (ii)}  If $\,\pi\in K[X]$ is a prime polynomial of degree $>1$, then $\deg \pi$ is even.\,
{\rm (iii)}  If $\,L\,\vert\, K$ is a  non-trivial finite field extension  of $K$, then  $[L:K]=\Dim_{K} L$ is even.
\end{lemma}

\begin{proof}
The reader is  recommended to prove the implications\,:\, (i) $\Rightarrow$ (iii) $\Rightarrow$ (ii)$\Rightarrow$(i)\,.
\end{proof}

Now, we shall prove the analogue of Theorem\,\ref{thm:4.6}\,---\,Hilbert's Nullstellensatz for $2$-fields.

\begin{theorem}\label{thm:4.9}\,{\rm \so{(Hilbert's Nullstellensatz for $2$-fields
\hbox{---\,U.\,Storch,\,2003})}}\,
Let $K$ be a $2$-field and  $A$ be  a standard graded $K$-algebra of projective dimension $d={\rm pd}(A)\geq 0$ and of odd multiplicity ${\rm e}(A)$. Further,  let $f_{1}, \ldots , f_{r}\in A$ be homogeneous elements of positive odd degrees, $r\leq d$.
Then $f_{1}, \ldots , f_{r}$ have a common zero in $\mathds{P}_{A}(K)$, {\rm i.\,e.}\,  $\emptyset \neq \mathds{P}_{A/\mathfrak{a}}(K) = {\rm V}_{+}(f_{1},\ldots , f_{r}) \subseteq \mathds{P}_{A}(K)$, $\mathfrak{a} :=Af_{1}+\cdots +Af_{r}$.
\end{theorem}

\begin{proof}
For $M:=A$, let $0=M_0\subsetneq M_1\subsetneq\cdots\subsetneq M_r=M$ be a chain with $M_\rho/M_{\rho-1}=(A/\mathfrak{p}_\rho)(-k_\rho)$ as in Lemma   in \ref{mypar:4.1}\,(1)\,(b). Then by \ref{mypar:4.1}\,(4)\,(a)\,(5) we have
${\rm e}\,(A)=\sum_{\rho\,,\,{\rm d}\,(A/\mathfrak{p}_\rho)=d}\;{\rm e}\,(A/\mathfrak{p}_\rho)$.  Since the multiplicity ${\rm e}(A)$ is odd by assumption, it follows that   at least one of  ${\rm e}$\,$(A/\mathfrak{p}_\rho)$ with  ${\rm pd}\,(A/\mathfrak{p}_\rho)=d$ is also odd.
If $d=0$, then by Lemma~\ref{lem:4.5} and Lemma~\ref{lem:4.8} necessarily
${\rm e}\,(A/\mathfrak{p}_\rho)=1$ for one $\mathfrak{p}_\rho$ with
${\rm d}\,(A/\mathfrak{p}_\rho)=0$, and such a prime ideal $\mathfrak{p}_\rho$ defines a point in $\mathds{P}_A(K)$. For the inductive step from $d$ to $d+1$, we may assume that $A$ is an integral domain and $f_r\not=0$. Then ${\rm e}\,(A/Af_r)=
{\rm e}\,(A)\cdot{\rm deg}\,f_r$ is also odd and ${\rm d}\,(A/Af_r)=d$, and by applying the induction hypothesis to $A/Af_r$ and the  residue classes $\overline{f}
_1,\ldots,\overline{f}_{r-1}$, the assertion follows.
\end{proof}

For $A=K[T_{0},\ldots , T_{n}]$,  we have $\,{\rm d}(A)=n$, $\,{\rm e}(A)=1$ and $\mathds{P}_{A}(K)=\mathds{P}^{n}(K)$. This  special case of Theorem\,\ref{thm:4.9} was already proved by Albrecht Pfister in \cite{pfister} as Theorem\,3\,:

\begin{corollary}\label{cor:4.10}{\rm\so{ (Projective Nullstellensatz for $2$-fields)}} Let $K$ be a $2$-field. Then
homogeneous polynomials $f_{1}, \ldots ,  f_{r} \in K[T_{0}, \ldots , T_{n}]$, $\,r\leq n\,$ of  odd degrees have a common non-trivial zero in $K^{n+1}$.
\end{corollary}

Since the field  $\R$ is a $2$-field, in~particular, we have\,:

\begin{corollary}\label{cor:4.11}{\rm\so{(Real Projective Nullstellensatz)}}
Homogeneous polynomials
$f_{1},\ldots, f_{n}\!\in\!\R[T_{0},\ldots , T_{n}]$ of odd degrees  have a common non-trivial zero in $\R^{n+1}$.
\end{corollary}

Now, we shall prove the analogue of Theorem\,\ref{thm:4.6a}\,---\,Hilbert's Nullstellensatz for $2$-fields.

\begin{theorem}\label{thm:4.11a}\,
Let $K$ be a $2$-field and let $A$ be  a standard graded $K$-algebra of projective dimension $d={\rm pd}(A)\geq 0$ and of odd multiplicity ${\rm e}(A)$. If a homogeneous element $f\in A$ vanishes at all points of $\mathds{P}_{A}(K)$, then $f=0$.
\end{theorem}

\begin{proof}
We proceed as in the proof of Theorem~\ref{thm:4.6a}.
 For $d=0$, by Lemma~\ref{lem:4.5}, necessarily ${\rm e}\,(A)=1$, since $K$ is a $2$-field and hence $f=0$ by the same argument as in \ref{thm:4.6a}.  Assume that  $d>0$ and $\deg f>0$.  Suppose on the contrary that $f\not=0$.
\vskip2pt

\noindent {\rm (*)}\hspace{3mm}  We claim that\,: there exists a homogeneous  prime ideal $\mathfrak{q}\neq 0$ in $A$ with $f\not\in\mathfrak{q}$ and $e(A/\gothq)$ odd.
\vskip2pt

As  in the proof of the Theorem \ref{thm:4.6a} one can prove the claim (*) by constructing a homogeneous prime  ideal $\mathfrak{q}$, since $e(A/Ag)=e(A)$  for  a suitable chosen $g\in A_1$, $g\not=0$ with  ${\rm e}\,(A/Ag)={\rm e}\,(A)$. On the other hand, since $d(A/\gothq)=d-1$, $e(A/\gothq)$ odd and $\mathds{P}_{A/\gothq}(K)\subseteq \mathds{P}_{A}(K)$, by induction hypothesis, it follows that $f\in\gothq$ which contradicts the claim (*).
\end{proof}

We use the Real Projective Nullstellensatz\,\ref{cor:4.11} to provide  an algebraic proof  of  the well-known Bor\-suk-Ulam theorem which states that\, :

\begin{theorem}\label{thm:4.12}{\rm (\,\so{Borsuk\hbox{-}Ulam}\,\footnote{\label{foot:16}This was conjectured by Stanislaw Ulam (1909--1984) and was proved  by Karol Borsuk (1905 -- 1982) in 1933 by elementary  but technically involved methods.
Borsuk presented the theorem at the International congress of
mathematicians at Z\"{u}rich in 1932 and it was published in {\it
Fundamentae Mathematicae} {\bf 20}, 177-190 (1933) with the title {\it Drei
S\"{a}tze \"{u}ber $n$-dimentionale euclidische Sph\"{a}re.}}\,)}\,
For every continuous map $\,g:{\rm S}^{n}\to\R^{n}$, $\,n\in\N,$ there exist anti-podal points $t$, $-t\in {\rm S}^{n}$ with $g(t)=g(-t)$.
\end{theorem}

The proof of Borsuk-Ulam theorem for the case $n=1$ is an easy application of the intermediate value theorem. The case $n=2$ is already non-trivial and it needs the concept of the first fundamental group which was introduced by Henri Poincar\'{e} (1854--1912)\,---\,who was responsible for formulating the Poincar\'{e} conjecture. The general case is usually proved by using higher homology
groups.
\smallskip

Borsuk-Ulam Theorem is fascinating even today. It implies   the classical Theorem of  Brou\-wer\,\footnote{\label{foot:17}{\bf Theorem}\,\,{\rm (\,\so{Brouwer's fixed point theorem}\,--\,Brouwer L. E. J. (1881-1966))}\, {\it Every continuous map
$\,f: \overline{{\rm B}}^{\,n} \longrightarrow \overline{{\rm B}}^{\,n}$ of the unit ball $\overline{{\rm B}}^{\,n}:= \{t\in\R^{n}\mid |\!|t|\!|\leq 1\}\,$ has a fixed point.}
{\bf Proof}\, If  $f$ has no fixed point, then the continuous map $h:\overline{\rm B}^{\,n}\to {\rm S}^{n-1}\subseteq\R^n$,  which maps the point $t\in\overline{\rm B}^{\,n}$ to the point of intersection of the line-segment ${\rm L}(f(t)\,,\,t):=\{f(t)+\lambda\,t\mid \lambda\in [0,1]\}\subseteq \R^{n}\,$ with the sphere  ${\rm S}^{n-1}\subseteq\overline{\rm B}^{\,n}$  has no zero.  But,  $h\vert_{{\rm S}^{n-1}}={\rm id}$ and in~particular,   $h(-t)=-h(t)$ for  $t\in {\rm S}^{n-1}$. Therefore, for  $n=1$, the Nullstellensatz is equivalent with the {\it Intermediate Value theorem}\,: Every continuous map  $h:[-1,1]\to\R$ with $h(-1)=-h(1)$ has a zero.
} and the Invariance of  Dimension Theorem\,\footnote{\label{foot:18}{\bf Theorem}\,\,{\rm (\,\so{Invariance of dimension}\,)}\,  {\it For $m>n$, there is no injective continuous map from an open subset $U\subseteq \R^{m}$ into $\R^{n}$. In~particular,
if $m\neq n$, then the Euclidean
spaces $\R^{m}$ and $\R^{n}$ are not homeomorphic.}}.

\begin{proof} 
	Recall that ${\rm S}^{n}= \{t=(t_{0}, \ldots , t_{n})\in\R^{n+1}\mid |\!|t|\!|^{2}= \sum_{i=0}^{n} t_{i}^{2} =1 \}\subseteq \R^{n+1}$ is the $n$-sphere.
Consider the odd continuous map $f:\mathrm{S}^{n} \to \R^{n}$, $t\mapsto f(t):=g(t)-g(-t)$ and the Borsuk-Ulam's
Nullstellensatz (see (i) in Theorem\,\ref{thm:4.13} below).
\end{proof}
We now prove the equivalence of  Borsuk-Ulam's Nullstellensatz with some other statements\,:

\begin{theorem}\label{thm:4.13}
Let $n\in\N$. Then the following statements are
equivalent:

\begin{rlist}
\item\, \  \,{\rm (\,\so{Borsuk\hbox{-}Ulam's Nullstellensatz}\,)}\,
 Every continuous odd map\,\footnote{A map $f:\mathrm{S}^{n}\longrightarrow\R^{n}\,$ is called an \textit{odd map} if $f(x)= - f(-x)$ for every $x\in\mathrm{S}^{n}$.} $\,f:{\rm S}^{n} \longrightarrow \R^{n}$, $n\in\N$, has a zero.

\item\,   {\rm\, (\,\so{Borsuk's antipodal theorem}\,)}\,
Every continuous map  $\,h:\overline{{\rm B}}^{\,n} \longrightarrow \R^{n}\,$ with $n\geq 1$ and the restriction $\,h\vert_{\,{{\rm S}^{n-1}}}:{\rm S}^{n-1}\longrightarrow\R^{n}\,$ odd, has a zero.

\item {\rm (\,\so{Real Projective Nullstellensatz}\,---\,Corollary\,\ref{cor:4.11}\,)}\,
Homogeneous polynomials \ $f_{1},$ $\ldots , f_{n}\in\R[T_{0},...,T_{n}]$ of  odd degree have a common non-trivial zero in $\R^{n+1}$.
\end{rlist}
\end{theorem}

\begin{proof}
(i)$\iff$(ii)\,  Note that for $n\geq 1$,
the odd continuous maps $f:{\rm S}^n\to\R^n$ correspond to the continuous maps $h: \overline{{\rm B}}^{\,n}\to\R^{n}$ such that the restriction $h\vert_{\, {\rm S}^{n-1}}:{\rm S}^{n-1}\to \R^{n}$ of $h$ to the subset ${\rm S}^{n-1}\subseteq \overline{\rm B}^{\,n}$  is odd. For a given $f:{\rm S}^n\to\R^n$ define $h(t):=f(\sqrt{1-\Vert t\Vert^2},t)$, $t\in\overline{\rm B}^n$, and conversely for a given $h: \overline{{\rm B}}^{\,n}\to\R^{n}$ define $f(t_0,t):=h(t)$, if   $t_0\ge0$,  $f(t_0,t):=-h(-t)$, if  $t_0\le0$, $(t_0,t)\in {\rm S}^n\subseteq\R\times\R^n$.
\smallskip

(i) $\Rightarrow$ (iii)
From (i) in~particular, it follows that $n$ odd polynomial functions $f_{1},\ldots , f_{n}:\R^{n+1}\longrightarrow \R$ have a common zero on ${\rm S}^{n}$.  If $F\in \R[T_{0},\ldots, T_{n}]$ defines an odd polynomial function $F:\R^{n+1}\to\R$, then all homogeneous components of even degree in $F$ are zero, i\,.e. in the homogeneous decomposition of $F$ only odd degree homogeneous components can occur. Suppose that  $F= \sum_{i=0}^{m} F_{2i+1}$, $F_{2m+1}\not=0$, is the homogeneous decomposition of $F$ with homogeneous components $F_{1},\ldots , F_{2m+1}$ of odd degrees $1,\ldots , 2m+1$, respectively.
Now, observe that $F$ and the \emph{homogeneous} polynomial $\,Q^{\,m}F_1+Q^{\,m-1}F_3+\cdots+F_{2m+1}$, $\,Q:=T_0^{\,2}+\cdots+T_n^{\,2}$, have the same values on the sphere ${\rm S}^n$.
\smallskip

(iii) $\Rightarrow$ (i)\,:   Let $f=(f_1,\ldots,f_n):{\rm S}^n\to\R^n$, $n\in\N$,  with $f_i: {\rm S}^n\to\R$, $i=1,\ldots , n$,  odd and continuous. Then by the well-known Weierstrass Approximation Theorem\,\footnote{\label{foot:20}{\small
{\bf Theorem}\,\,{\rm (\,\so{Weierstrass}\,)}\,
{\it Let $X\subseteq  \R^{n}$, $n\in\N$ be a compact
subset. Then the set of polynomial functions
$\R[T_{1},...,T_{n}]$ is dense in
$({\rm C}(X,\R),|\!|\cdot |\!|_{\rm sup}),$ where ${\rm C}(X,\R)$ is the $\R$-algebra of $\R$-valued continuous functions on $X$ and for every $f\in {\rm C}(X,\R)$,  $|\!| f|\!|_{\rm sup}:={\rm Sup}\{|\!|f(x)|\!| \mid x\in X\}$.}
(\,Karl Weierstrass (1815-1897)  is known as the father of modern analysis, and  contributed to the theory of periodic functions, functions of real variables, elliptic functions, Abelian functions, converging infinite products, and the calculus of variations. He also advanced the theory of bilinear and quadratic forms.\,)}}
for every  $k\in\N^*$, there exist \emph{polynomial functions} $g_{ik}$ with  $|g_{ik}(t)-f_i(t)|\le 1/k$ for $i=1,\ldots,n$ and all $t\in {\rm S}^n$. For the odd parts  $f_{ik}(t):=(g_{ik}(t)-g_{ik}(-t))/2$, it follows
$|f_{ik}(t)-f_i(t)|=\frac{1}{2}\,|(g_{ik}(t)-f_i(t))-(g_{ik}(-t)-f_i(-t))|\le1/k$.
By the Real Algebraic Nullstellensatz\,\ref{cor:4.11}, the $f_{ik}$, $i=1,\ldots,n$, have a common zero $t_k\in {\rm S}^n$. Then an accumulation point $t\in {\rm S}^n$ of  $t_k$, $k\in\N^*$, is a common zero of  $f_1,\ldots,f_n$.
\end{proof}

\begin{remark}\label{rem:4.14}\,
Note that we have proved Real Projective Nullstellensatz in Corollary\,\ref{cor:4.11} and hence the equivalence in Theorem\,\ref{thm:4.13} proves the Borsuk-Ulam's Nullstellensatz\,\ref{thm:4.13}\,(ii) also. In~particular, we have proved the Borsuk-Ulam Theorem.
\end{remark}

\section{Combinatorial Nullstellensatz}

In this section we prove one of the most recent Nullstellens\"atze\,---\,
Combinatorial Nullstellensatz  a celebrated result of Noga Alon  (see \cite{alon}) proved in 1999 that has served as a powerful technical tool  in combinatorics, graph theory  and additive number theory. We use HNS 2 to prove the commonly used versions of Combinatorial Nullstellensatz. As an illustration we will use it to prove  Erd\"{o}s-Heilbronn conjecture and Dyson's conjecture.

\begin{theorem}\label{thm:5.1}\,{\rm\so{(Combinatorial Nullstellensatz\hbox{---\,N.\,Alon,\,1999})}}\,
Let $K$ be a field, 
 $\Lambda_1,\ldots ,\Lambda_{\,n}\subseteq K$ be finite subsets of $K$, $\Lambda=\Lambda_1 \times \cdots \times \Lambda_{\,n}\subseteq K^{n}$ and  $g_i(X_i) = \prod_{\,a_{\,i} \in \Lambda_{\,i} }(X_i-a_{i})\in K[X_{i}] \subseteq K[X_{1},\ldots , X_{n}]$, $i=1,\dots,n$. Then \ $\,{\rm I}_K(\Lambda) = \langle g_1(X_1),\ldots, g_n(X_n) \rangle =: \gothA\,$.
\end{theorem}

\begin{proof}
Let $\overline{K}$ denote the algebraic closure of the field $K$. Note that ${\rm V}_{\overline{K}}(\gothA)=\Lambda$ and so $\mathrm{I}_K({\rm V}_{\overline{K}}(\gothA))= \sqrt{\gothA}$ by Theorem \ref{thm:2.10}\,(2)\,(HNS\,2). Therefore it is enough to prove that $\gothA$ is a radical ideal, i.\,e. $K[X_1,\dots,X_n]/\gothA$ is reduced. For this, it is enough to note that the kernel of the $K$-algebra homomorphism $\varepsilon:K[X_{1},\ldots  , X_{n}]\longrightarrow K^{|\Lambda|}$, $f\longmapsto (f(a))_{a\in\Lambda}$, is the ideal $\gothA$. Clearly,  by definitions  $g_{1}(X_{1}),\ldots , g_{n}(X_{n})\in\Ker\varepsilon$ and so $\gothA\subseteq \Ker\varepsilon$. To prove the reverse inclusion, let $f\in\Ker\varepsilon$. Using division with remainder by $g_{i}(X_{i})$,  $i=1,\ldots , n$, we can write $f=h+f\,'$ with $h$, $f\,'\in K[X_{1},\ldots , X_{n}]$, $h\in \langle g_{1}, \ldots , g_{n}\rangle=\gothA$  and $\deg_{X_{i}}f\,'<\deg\, g_{i}(X_{i})=|\Lambda_{i}|$ for all $i=1,\ldots , n$. Now, it follows from the Identity Theorem for Polynomials in \ref{mypar:2.1}\,(6) that $f\,'=0$ and so $f=h\in\gothA$.
\end{proof}

We shall deduce a variant of Combinatorial Nullstellensatz which is suitable for applications.

\begin{corollary}\label{cor:5.2}{\rm\so{(Combinatorial Nullstellensatz)}}\,
Let $K$ be a  fi\-eld,  $f(X_1,\ldots,X_n)\!\in\!K[X_{1},\ldots , X_{n}]$  and $d_{1},\ldots , d_{n}\!\in\!\N$. Suppose that\, {\rm (i)}\, $\deg(f)\!=\!d_{1}\!+\!\cdots\!+\! d_{n}$.
{\rm (ii)}  The coefficient of the monomial $X_{1}^{d_{1}}\cdots X_{n}^{d_{n}}$ in $f$ is non-zero.\,
Then, for subsets $\Lambda_1,\ldots ,\Lambda_{\,n}\subseteq K$ with $|\Lambda_{\,i}|> d_i\,$ for every $\,i=1,\dots, n$, there exist $(a_1,\ldots, a_n) \in \prod_{i=1}^{\,n} \Lambda_{\,i}$ such that $f(a_1,\ldots, a_n) \neq  0$.
\end{corollary}

\begin{proof}
We may assume that $|\Lambda_{\,i}|=d_{i}+1$ for every $i=1,\ldots , n$. We shall prove that\,: if (i) holds and if  $f(a_1,\ldots, a_n) =0$ for every $(a_1,\ldots, a_n) \in \prod_{\,i=1}^{\,n} \Lambda_{\,i}\,$,  then (ii) does not hold, i.\,e. the coefficient of $X_{1}^{d_{1}}\cdots X_{n}^{d_{\,n}}$ in $f$ is $0$. Note that, since $f\equiv 0$ on $\Lambda:=\prod_{\,i=1}^{\,n} \Lambda_{\,i}$, $f\in {\rm I}_{K}(\Lambda)=\gothA:=\langle g_{1}(X_{1}), \ldots , g_{n}(X_{n})\rangle$, where $g_i(X_i) = \prod_{\,a_{i} \in \Lambda_{i} }(X_i-a_{i})\in K[X_{i}] \subseteq K[X_{1},\ldots , X_{n}]$, $i=1,\dots,n$. Therefore, we can write
$f=h_{1}g_{1}+\cdots + h_{n}g_{n}$ with $h_{1},\ldots , h_{n}\in K[X_{1},\ldots , X_{n}]$. Now, as in the proof of Theorem \ref{thm:5.1}, using division with remainder by $g_{i}(X_{i})$,  $i=1,\ldots , n$, we may assume that $\deg h_{i}\leq \deg\, f - \deg\, g_{i}(X_{i})$ and hence $\deg\, h_{i}\,g_{i}\leq \deg\, f= d_{1}+\cdots +d_{\,n}$ for all $i=1,\ldots , n$. If $h_{i}\,g_{i}$ contains any monomial of degree $\deg\,f$, then such a monomial would be of maximal degree in $h_{i}\,g_{i} = h_{i}\,\prod_{a_{i} \in \Lambda_{i} }(X_i-a_{i})$ and hence will be divisible by $X_{i}^{d_{\,i}+1}$. This proves that the coefficient of the monomial $X_{1}^{d_{1}}\cdots X_{n}^{d_{n}}$ in $h_{i}\,g_{i}$ is $0$ for every $i=1,\ldots , n$ and hence the coefficient of  $X_{1}^{d_{1}}\cdots X_{n}^{d_{\,n}}$ in $f$ must be $0$.
\end{proof}

{\it Erd\"os-Heilbronn conjecture}\footnote{\label{foot:19}The {\it Cauchy-Davenport theorem} (\,named after Augustin Cauchy (1789-1857) and Harlod Davenport (1907-1969)\,) states that\,:\,{\it  if $M$, $N$ are non-empty subsets of $Z/p\Z$, then the sum-set $M\!+\!N$ has cardinality $\!\geq\!\min\{p, |M|\!+\!|N\!|-\!1\}$. In~particular, $|2\,M|\!\geq\!\min\{p, 2|M|\!-\!1\}$.} See [Davenport,\,H.: On the addition of residue classes. {\em J. London Math. Soc.}{\bf 30}\,(1935), 30-32.] and [Cauchy, A.\,L.\,: Recherches sur les nombers. {\em J. \'Ecole polytech.} {\bf 9}\,(1813), 99-116.]
In 1964 Paul Erd\"os (1913-1996) and Hans Heilbronn (1908-1975) conjectured that $|2\,M|\geq \min\{p, 2|M|\!-\!3\}$ (see \cite{erdos}). Erd\"os was one of the most prolific mathematicians of the 20th century and  was known both for his social practice of mathematics (he engaged more than 500 collaborators) and for his eccentric lifestyle.} was recently proved by Dias da Silva and Hamidoune (see\,\cite{silva}) using linear algebra and the representation
theory of the symmetric group. We give an elementary  proof of Erd\"os-Heilbronn conjecture  by using the Combinatorial Nullstellensatz Theorem\,\ref{thm:5.1} and Corollary\,\ref{cor:5.2}. This method even yields generalizations of both the Erd\"os-Heilbronn conjecture (see also \cite{alonconjecture}) and  the {\it Cauchy-Davenport theorem.}

\begin{theorem}\label{thm:5.3}{\rm \so{(Micha{\l}ek,\,\hbox{2010})}}\,
Let $p$ be a prime and let $M$ and $N$ be two non-empty subsets of $\Z/p\Z\,$.  Let
\begin{align*}
L = \{ x \in \Z/p\Z \mid x = a + b \ \text{ for some } a \in M\,, \,b \in N, a \not= b\}.
\end{align*}
Then $|L| \ge \min(p, |M| + |N| - 3)$
\end{theorem}
	
\begin{proof}
The assertion is trivial for  $p=2$.  Let $p>2$.
\smallskip

{\bf Case\,1\,:}\, $\min(p, |M|+|N|-3) = p$. In this case, $p+3\leq |M|+|N|$ and so by inclusion-exclusion principle
$|M\cap (g-N)| =|M|+|N|-|M\cup (g-N)| \geq p+3-p=3$ for every $g\in\Z/p\Z$.
We show that $|L|=p$, in particular, $L=\Z/p\Z$.  Let $g \in \Z/p\Z$ and $a \in M \cap (g-N)$ be such that $a \neq g/2$.  Then $g = a + b$ for some $b \in N$ with $b\neq a$, i.\,e.   $g \in L$ and hence $L = \Z/p\Z $.
 \smallskip

{\bf Case\,2\,:}\, $\min(p, |M|+|N|-3))<p$.    In this case,  $|M| + |N| - 3 < p$. Suppose, on the contrary that $|L| < \min(p, |M| + |N| - 3) = |M|+|N|-3$. Then there exists a subset $D\subseteq \Z/p\,\Z$ with   $L \subseteq D$ and $|D| = |M| + |N| - 4$. We would like to apply  Corollary\,\ref{cor:5.2}. For  this, let
\begin{align*}
P(X, Y) = \prod_{d \in D}	(X + Y - d) \text{ and } Q(X, Y) = P(X, Y)\,(X - Y).
\end{align*}
Note that  $P(a, b) = 0$ for every $a \in M$, $b \in N$, $a \neq b$, and hence  $Q(a, b) = 0$ for all $a \in M$, $b \in N$. Further, for $i=0,\ldots, |D|$, the coefficient of $X^{\,i}\,Y^{\,|D|-i}$ in $P(X, Y)$ is equal to $\binom{|D|}{i}$.  Then it follows that for $i=0,\dots,|D|+1$,  the coefficient of $X^{\,i}\,Y^{\,|D|+1-i}$ in $Q(X, Y)$,  is equal to $\binom{\,|D|\,}{i-1} - \binom{\,|D|\,}{i}$. Therefore the coefficient of $X^{\,i}\,Y^{\,|D|+1-i}$ in $Q(X, Y)$ is $0$  if and only if $\,i = (|D|+1)/2$ in $\Z/p\Z$.  Since $\,|D|+1 = |M|+|N|-3$,   the coefficient of either $X^{\,|M|-1}\,Y^{\,|N|-2}$ or of $X^{\,|M|-2}\,Y^{\,|N|-1}$ is non-zero.
But, since $\deg Q\!=\!|D|\!+\!1\!=\!|M| \!+\! |N|\! - \!3$,  by Corollary\,\ref{cor:5.2} (applied either to  $X^{\,|M|-1}\,Y^{\,|N|-2}$ or to  $X^{\,|M|-2}\,Y^{\,|N|-1}$ and $\Lambda_{1}\!=\!M$ , $\Lambda_{2}\!=\!N$) $\,Q(a, b)\neq 0$ for some $\,(a,b)\in M\times N$ which is absurd.
\end{proof}

In 2012 Karasev and  Petrov proved the following  improved version of Corollary\,\ref{cor:5.2}  in \cite[Theorem 4]{Karasev}.
This is used to prove the Dyson's conjecture, see Theorem~\ref{thm:5.5} below.

\begin{theorem}\label{thm:5.4}
Let $K$ be an arbitrary  field and let $f(X_1,\dots,X_n) \in K[X_1,\dots,X_n]$ be such that $\deg f \le |\nu|:=\nu_{1} + \cdots + \nu_{n}$ for a fixed $(\nu_{1},\ldots, \nu_{n}) \in \N^{n}$. For subsets $\Lambda_1,\ldots ,\Lambda_{\,n}\subseteq K$ with $|\Lambda_{\,i}|=\nu_{i}+1\,$, let $g_{i}(X_{i}):= \prod_{a_{i}\in \Lambda_{\,i}} (X_{i}-a_{i})$ for every $\,i=1,\dots, n\,$, and
$\Lambda:=\prod_{i=1}^{\,n} \Lambda_{\,i}\,$. We have
\[	C:= \ \hbox{coefficient of } X_{1}^{\nu_{1}}\cdots X_{n}^{\nu_{n}}  \ \hbox{ in } \  f \ = \
\sum_{(a_1,\ldots, a_n) \in \Lambda} \ \frac{f(a_1,\ldots, a_n)}{g'_1(a_1)\cdots g'_n(a_n)}\,.
\leqno{\rm (\thetheorem.1)}\]
In~particular, if $C \neq 0$,  then there exists $(a_1,\ldots, a_n) \in \Lambda$ such that $f(a_{1},\ldots ,  a_{n})\neq  0$.
\end{theorem}

\proof We consider two cases.
 \smallskip

{\bf Case 1\,:} $f=X_1^{\nu_{1}}\cdots X_n^{\nu_{n}}$. Since $X_{i}$-degrees of  polynomials on  both sides of the equation (\thetheorem.2) below are $\leq \nu_{i} <\deg\, g_{i}(X_{i})=|\Lambda_{i}|$ for every $i=1,\ldots ,n$, it follows from the Identity Theorem for Polynomials in\,\ref{mypar:2.1}\,(6)  that\,:
\vspace*{-4mm}
\[ X_1^{\nu_1}\cdots X_n^{\nu_{n}} =
\sum_{(a_{1},\ldots , a_{n}) \in \Lambda_{1}\times \cdots \times \Lambda_{\,n}} \ \,
 a_1^{\nu_{1}}\cdots  a_{n}^{\nu_{\,n}}\, \left(\prod_{\,i=1}^{\,n}\,\, \frac{g_{i}(X_{i})}{g\,'_{i}(a_{i})\,(X_{i}-a_{i})}\right)
\leqno{\rm (\thetheorem.2)}
\]
In~particular, comparing the coefficient of $X_1^{\nu_1}\cdots X_n^{\nu_n}$ on both sides, we get\,:
\begin{align*}
1= \sum_{(a_{1},\ldots , a_{n}) \in \Lambda_{1}\times \cdots \times \Lambda_{n}} \ \,  \, \frac{a_1^{\nu_{1}}\cdots  a_{n}^{\nu_n}} {g\,'_1(a_1) \cdots g\,'_n(a_n)}.
\end{align*}

{\bf Case 2\,:} We prove the general case by  induction on $n$.  Note that by the linearity of both sides in the formula  (\thetheorem.1), it is enough to prove the formula for
$\,h:=f-C\,X_1^{\nu_1}\cdots X_n^{\nu_n}$. Then
\begin{align*}
h(X_1, a_{2}, \ldots , a_{n}) = f(X_1, a_{2},\ldots , a_{n}) - C\, X_1^{\nu_{1}} a_{2}^{\nu_{2}}\cdots  a_{n}^{\nu_{n}} \quad \hbox{for every} \ \ (a_2,\ldots, a_n) \in K^{n-1}\,.
\end{align*}
By the case $n=1$, since the coefficient of $X_1^{\nu_1}$ is zero in $h(X_1, a_2,\ldots, a_n)$, $(a_2,\ldots, a_n) \in K^{n-1}$, it follows that
\vspace*{-3mm}
\begin{align*}
\sum_{a_1 \in \Lambda_1} \, \ \frac{h(a_1, a_2,\ldots, a_n)}{g\,'_1(a_1)}=0.
\end{align*}
Dividing the above equation by $g\,'_2( a_2)\cdots g\,'_n(a_n)$ and taking the sum over all $(n-1)$-tuples $(a_2,\ldots, a_n) \in \Lambda_{2}\times \cdots \times \Lambda_{\,n}$, we get\,:
\begin{align*}
0 = \sum_{(a_{1},\ldots , a_{\,n}) \in \Lambda_{1}\times \cdots \times \Lambda_{\,n}} \, \ \frac{h(a_1,\ldots, a_n)}{g\,'_1(a_1) \cdots  g\,'_n(a_n)}
= \sum_{(a_{1},\ldots , a_{\,n}) \in \Lambda_{1}\times \cdots \times \Lambda_{\,n}} \, \ \frac{f(a_1,\ldots, a_n)\,-\,C\,a_1^{\nu_1} \cdots a_n^{\nu_n}} {g\,'_1(a_1)\cdots g\,'_n(a_n)}.
\end{align*}
Therefore, by the Case~1 (the second equality in the equation below), we get\,:
\begin{align*}
\sum_{(a_{1},\ldots , a_{\,n}) \in \Lambda_{1}\times \cdots \times \Lambda_{\,n}} \, \ \frac{f(a_1,\ldots, a_n)}{g\,'_1(a_1) \cdots  g\,'_n(a_n)}
= C \sum_{(a_{1},\ldots , a_{\,n}) \in \Lambda_{1}\times \cdots \times \Lambda_{\,n}} \, \ \frac{a_1^{\nu_1} \cdots a_n^{\nu_n}} {g\,'_1(a_1)\cdots g\,'_n(a_n)} \,=\, C\,. \hspace*{1.4cm}\bullet
\end{align*}

Motivated by a problem in statistical physics, Freeman Dyson
in 1962 (see \cite{dyson}) formulated a conjecture which states that\,:  {\it the constant term of the Laurent polynomial $\,\,\prod_{1\,\le\, i\neq j \,\le\, n} (1-X_i/X_j)^{\alpha_i}$ is equal to the multinomial coefficient $(\alpha_1+\cdots+\alpha_n)!/(\alpha_1!\alpha_2!\cdots \alpha_n!)$.}
This  conjecture was first proved in 1962 independently by Kenneth Wilson (1936 - 2013) and J.\,Gunson.
\smallskip

The Combinatorial Nullstellensatz \ref{thm:5.1} and \ref{cor:5.2} are used to get information on the values of polynomials from their coefficients, but (\thetheorem.1)
allows us to use it in the other direction. This is used in the following  proof  of Dyson's conjecture by Karasev and Petrov \cite[Theorem 5]{Karasev}.

\begin{theorem}\label{thm:5.5}{\rm \so{(Dyson's conjecture)}}\,
Let $\alpha_{\,i}\,$, $i=1,\dots,n$ be positive integers and $C$ be the constant term in
\vspace*{-3mm}	
\begin{align*}
\prod_{1\le\, i\neq j \,\le\, n} \left( 1- X_{i}/X_{j} \right)^{\alpha_{\,i}}.
\end{align*}
Or, more generally, let  $\,\alpha=\alpha_1+\cdots+\alpha_{\,n}\,$ and let $\,C\,$ be the coefficient of the monomial $\,\,\prod_{i=1}^{\,n} \,X_i^{\alpha -\alpha_{\,i}}$ in
\[\,f(X_1,\dots,X_n) = \prod_{1\le\, i\, <\, j \,\le\, n} (-1)^{\alpha_j} (X_j-X_i)^{\alpha_{\,i}+\alpha_j}\,. \leqno{\rm (\thetheorem.1)}
\]
Then
\[\,C=\frac{\alpha!}{\alpha_1!\cdots \alpha_{\,n}!}\,. \leqno{\rm (\thetheorem.2)}
\]
\end{theorem}

\noindent {\bf Sketch of a proof}\,   In the notation of Theorem~\ref{thm:5.4}, we have $\nu_{i}=\alpha-\alpha_{\,i}$. The idea is to add terms of lower degree to $f$. It does not change the coefficient $C$ but may significantly change the RHS of (\thetheorem.1).
\smallskip

In order to apply Theorem~\ref{thm:5.4} ($K=\Q$), we are free to choose the sets $\Lambda_{\,i}\subseteq \Z (\subseteq \Q)$ with $|\Lambda_{\,i}|=\alpha-\alpha_{\,i}+1$.  We shall change $f$ to $\widetilde{f}$ so that $\widetilde{f}$ takes a unique non-zero value on $\Lambda:=\prod_{i=1}^{\,n}\,\Lambda_{i}$. For this, we
choose $\Lambda_{\,i}:=[0,\alpha-\alpha_{\,i}]:=\{0,1,\ldots , \alpha-\alpha_{\,i}\}$.  Note that if $a_{i}\in \Lambda_{\,i}$, then the segment\footnote{\label{foot:20} For integers $a\,,\,b\in\Z$, we denote the segment $\{t\in\Z\mid a\leq t\leq b\}$ of integers by $[a,b]$.} $\Delta_{\,i}:=[a_{i}, a_{i}+\alpha_{\,i}-1]\subseteq [0,\alpha-1]$.
\smallskip

Now change $f$ by replacing the terms $(X_j-X_i)^{\alpha_i+\alpha_j}$, $1\leq i <j\leq n$,  in the formula (\thetheorem.1) by the polynomials
\vspace*{-3mm}
\begin{align*}
G_{i\,,\,j}(X_1,\ldots ,X_n) = \prod_{t=- \alpha_i+1}^{\alpha_j}(X_j - X_i+t)\,, \quad 1\leq i <j\leq n\,.
\end{align*}
Therefore
\[\,\widetilde{f}(X_{1},\ldots , X_{n})  = \prod_{1\le\, i\, <\, j \,\le\, n} (-1)^{\alpha_j}  G_{i\,,\,j}(X_1,\ldots ,X_n)  \,.\leqno{(\thetheorem.3)}
\]
Note that $\widetilde{f}$ does not vanish on $\Lambda$ if and only if $G_{i\,,\,j}$ does not vanish on $\Lambda$ for all $1 \le i < j \le n$.
Further, for $1\leq i<j\leq n$, non-vanishing of $G_{i\,\,,j}$ is equivalent to the conditions $\Delta_{\,i}\cap \Delta_{j}=\emptyset$ and $\Delta_{\,i}$ is not the segment following $\Delta_{j}\,$, i.\,e. $\min \Delta_{\,i }\neq  \min \Delta_{j}+1$.
All this together may happen only if   $\Delta_{1},\ldots,\Delta_{\,n}$ are consecutive segments $[0, \alpha_1\!-\!1],  [\alpha_1, \alpha_1\!+\!\alpha_2\!-\!1], \ldots , [\alpha_1\!+\!\cdots\!+\!\alpha_{n-1}, \alpha_1\!+\!\cdots\!+\!\alpha_{\,n}\!-\!1].$ Let $\beta_{\,i}:= \alpha_1+\cdots+\alpha_{\,i-1}$. This proves that
 $\widetilde{f}$ vanishes on all points in $\Lambda$ except the point $(\beta_1, \beta_2, \ldots, \beta_n)$.
 Now, it follows from Theorem~\ref{thm:5.3} (applied to the polynomial $\widetilde{f}$) that
\[	C
\ = \  \frac{\widetilde{f}(\beta_1, \ldots, \beta_{\,n})}{g\,'_i(\beta_1)\cdots g\,'_n(\beta_{\,n})} \ = \ \frac{\prod_{1\,\le\, i\,<\,j \,\le\, n}\,\,(-1)^{\alpha_j}G_{i\,,\,j}(\beta_1,\dots,\beta_{\,n})}{g\,'_i(\beta_1)\cdots g\,'_n(\beta_{\,n})} \, , \leqno{\rm (\thetheorem.4)}
\]
where $\,g_i(X_i) = \prod_{\,a_{i}=0}^{\,\alpha - \alpha_i}\,\,(X_{i}- a_{i})$. The RHS of (\thetheorem.4) may be calculated easily by  substituting the numerator and  the denominator from the equations (\thetheorem.5) and (\thetheorem.6) respectively, which are easy to verify\,:
\[\,g\,'_{i}(\beta_{\,i}) =
 (-1)^{\alpha_{\,i+1}+\cdots+\alpha_{\,n}}(\alpha_1+\cdots+\alpha_{\,i-1})!\, (\alpha_{\,i+1}+\cdots+\alpha_{\,n})!\, \leqno{\rm (\thetheorem.5)}
\]
and
\vspace*{-2mm}	
\[\,G_{i\,,\,j}(\beta_{1},\ldots,\beta_{\,n}) =
\frac{(\alpha_{\,i}+\cdots+\alpha_{\,j})!}{(\alpha_{\,i+1}+\cdots+\alpha_{\,j-1})!}\,. \hspace*{50 mm} \dppqed \leqno{\rm (\thetheorem.6)}
\]

\bibliographystyle{plain}

\end{document}